\documentclass[reqno]{files/gtpart}

\title{On framed rook algebras}
\author[D. Arcis, J. Espinoza and M. Flores]{Diego Arcis\\Jorge Espinoza\\Marcelo Flores}

\address{
	Departamento de Matem\'aticas, Universidad de La Serena, Cisternas 1200 -- 1700000 La Serena, Chile.\textcolor{white}{$\underbrace{1}$}\newline
	Instituto de Matem\'aticas, Universidad de Talca, Campus Norte, Camino Lircay S/N -- 3460000 Talca, Chile.\textcolor{white}{$\underbrace{.}$}\newline
	Instituto de Matem\'aticas, Universidad de Valpara\'iso, Gran Breta\~{n}a 1111 -- 2340000 Valpara\'iso, Chile.}

\email{diego.arcis@userena.cl}
\email{joespinoza@utalca.cl}
\email{marcelo.flores@uv.cl}

 % Connections with quantum groups, Chevalley groups, $p$-adic groups, Hecke algebras, and related topics.
 % Hecke algebras and their representations.
 % Representation theory.
 % Semigroups of transformations, etc.

\usepackage[utf8]{inputenc}

\usepackage{
	accents,
	amsthm,
	amsmath,
	amsfonts,
	cancel, % \cancel
	caption,
    float,
    graphicx,
    mathtools,
    mathrsfs,
    soul, % \st
    subcaption,
    young,
    ytableau,
    enumerate
}

\usepackage[all]{xy}
\usepackage[bottom=3.5cm,left=2.8cm,right=2cm,top=3cm]{geometry}
\usepackage[vcentermath]{youngtab}
\usepackage{bbm}

\newtheorem{thm}{Theorem}[section]
\newtheorem{crl}[thm]{Corollary}
\newtheorem{lem}[thm]{Lemma}
\newtheorem{pro}[thm]{Proposition}

\theoremstyle{definition}

\newtheorem{exm}[thm]{Example}

\newtheorem{rem}[thm]{Remark}

\def\blue{\color{blue}}\def\red{\color{red}}

\numberwithin{equation}{section}

\renewcommand{\operatorname}{\mathsf}

\newcommand{\F}{\mathcal{F}}

\newcommand{\I}{\mathcal{I}}
\newcommand{\K}{\mathcal{K}}

\newcommand{\Y}{\mathcal{Y}}

\renewcommand{\J}{\mathcal{J}}
\renewcommand{\H}{\mathcal{H}}

\renewcommand{\P}{\mathcal{P}}
\renewcommand{\R}{\mathcal{R}}
\renewcommand{\S}{\mathcal{S}}

\newcommand{\QQ}{\mathcal{Q}}
\newcommand{\RY}{\mathcal{RY}}

\newcommand{\PBr}{\mathcal{P\!B}r}

\newcommand{\CC}{\mathfrak{C}}

\newcommand{\s}{\mathfrak{s}}

\renewcommand{\b}{\mathfrak{b}}

\newcommand{\Fbb}{\mathbb{F}}
\newcommand{\Kbb}{\mathbb{K}}
\newcommand{\Nbb}{\mathbb{N}}
\newcommand{\Sbb}{\mathbb{S}}

\newcommand{\Ebf}{\mathbf{E}}
\newcommand{\Fbf}{\mathbf{F}}
\newcommand{\Gbf}{\mathbf{G}}

\newcommand{\Mbf}{\mathbf{M}}
\newcommand{\Nbf}{\mathbf{N}}

\newcommand{\Tbf}{\mathbf{T}}

\newcommand{\GLbf}{\mathbf{GL}}

\newcommand{\mbf}{\mathbf{m}}

\newcommand{\Csf}{\mathsf{C}}
\newcommand{\Esf}{\mathsf{E}}
\newcommand{\Fsf}{\mathsf{F}}

\newcommand{\Nsf}{\mathsf{N}}
\newcommand{\Psf}{\mathsf{P}}
\newcommand{\Rsf}{\mathsf{R}}
\newcommand{\Tsf}{\mathsf{T}}

\newcommand{\End}{\operatorname{End}}
\newcommand{\Ind}{\operatorname{Ind}}

\newcommand{\dom}{{\operatorname{dom}}}
\newcommand{\inv}{{\operatorname{inv}}}

\newcommand{\supp}{{\operatorname{supp}}}

\newcommand{\codom}{{\operatorname{codom}}}

\newcommand{\arc}[1]{\stackrel{_{_\frown}}{#1}}

\newcommand{\bTsf}{\overline{\mathsf{T}}}
\newcommand{\aspdf}{1}
\newcommand{\vcdraw}[1]{\vcenter{\hbox{#1}}}

\newcommand{\zerc}{\textcolor{gray!50}{0}}

\newcommand{\figscale}{0.9}

\newcommand{\figtwenin}{
	\centering
	\ifnum\aspdf=1\[
		\vcdraw{\includegraphics[scale=\figscale]{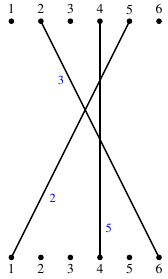}}
		\quad=\quad
		\vcdraw{\includegraphics[scale=\figscale]{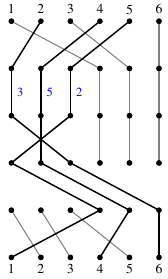}}
	\]\else
		\begin{tikzpicture}
			\vpartition[type=2,height=4,tkzpic=0]{{2,-6},{4,-4},{5,-1}}
			\node at(0.84,3){\blue{$\scriptscriptstyle{3}$}};
			\node at(1.65,0.5){\blue{$\scriptscriptstyle{5}$}};
			\node at(0.7,1){\blue{$\scriptscriptstyle{2}$}};
		\end{tikzpicture}
		\begin{tikzpicture}
			\tie[style=solid,color=black!50,height=0.8]{{1,5},{4,4},{4,3},{4,2}}
			\tie[style=solid,color=black!50,height=0.8]{{3,5},{5,4},{5,3},{5,2}}
			\tie[style=solid,color=black!50,height=0.8]{{6,5},{6,4},{6,3},{6,2}}
			\tie[style=solid,color=black!50,height=0.8]{{1,1},{2,0}}
			\tie[style=solid,color=black!50,height=0.8]{{2,1},{3,0}}
			\tie[style=solid,color=black!50,height=0.8]{{3,1},{5,0}}
			\vpartition[type=-1,tkzpic=0,height=0.8,bulla=0,nstr=6]
			{{4,-1},{5,-4},{6,-6}}
			\vpartition[type=0,tkzpic=0,height=0.8,bulla=0,floor=1,nstr=6]
			{{1,-4},{2,-5},{3,-6}}
			\vpartition[type=0,tkzpic=0,height=0.8,bulla=0,floor=2,nstr=6]
			{{1,-3},{2,-2},{3,-1}}
			\vpartition[type=0,tkzpic=0,height=0.8,bulla=0,floor=3,nstr=6]
			{{1,-1},{2,-2},{3,-3}}
			\vpartition[type=1,tkzpic=0,height=0.8,floor=4,nstr=6]
			{{2,-1},{4,-2},{5,-3}}
			\node at(0.15,2.8){\blue{$\scriptscriptstyle{3}$}};
			\node at(0.65,2.8){\blue{$\scriptscriptstyle{5}$}};
			\node at(1.15,2.8){\blue{$\scriptscriptstyle{2}$}};
		\end{tikzpicture}
	\fi
}

\newcommand{\figtweeig}{
	\centering
	\ifnum\aspdf=1\[
		\omega_A=\vcdraw{\includegraphics[scale=\figscale]{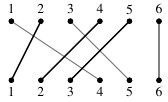}}
		\kern+4em
		\overline{\omega}_A=\vcdraw{\includegraphics[scale=\figscale]{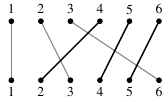}}
	\]\else
		\begin{tikzpicture}
			\tie[style=solid,color=black!50,height=1]{{1,1},{4,0}}
			\tie[style=solid,color=black!50,height=1]{{3,1},{5,0}}
			\tie[style=solid,color=black!50,height=1]{{6,1},{6,0}}
			\vpartition[tkzpic=0,nstr=6,type=2]{{2,-1},{4,-2},{5,-3}}
		\end{tikzpicture}
		\begin{tikzpicture}
			\tie[style=solid,color=black!50,height=1]{{1,1},{1,0}}
			\tie[style=solid,color=black!50,height=1]{{2,1},{3,0}}
			\tie[style=solid,color=black!50,height=1]{{3,1},{6,0}}
			\vpartition[tkzpic=0,nstr=6,type=2]{{4,-2},{5,-4},{6,-5}}
		\end{tikzpicture}
	\fi
}

\newcommand{\figtwesev}{
	\centering
	\ifnum\aspdf=1\[\begin{array}{ccccccc}
		\includegraphics[scale=\figscale]{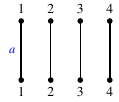}&&
		\includegraphics[scale=\figscale]{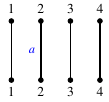}&&
		\includegraphics[scale=\figscale]{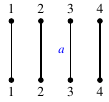}&&
		\includegraphics[scale=\figscale]{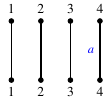}\\[-0.1cm]
		_{\,\,\,\,(1,\,g_1)}&&_{\,(1,\,g_2)}&&_{\,(1,\,g_3)}&&_{\,(1,\,g_4)}
	\end{array}\]\else
		\begin{tikzpicture}
			\vpartition[type=2,tkzpic=0]{{1,-1},{2,-2},{3,-3},{4,-4}}
			\node at(-0.15,0.5){\blue{$\scriptscriptstyle{a}$}};
		\end{tikzpicture}
		\begin{tikzpicture}
			\vpartition[type=2,tkzpic=0]{{1,-1},{2,-2},{3,-3},{4,-4}}
			\node at(0.35,0.5){\blue{$\scriptscriptstyle{a}$}};
		\end{tikzpicture}
		\begin{tikzpicture}
			\vpartition[type=2,tkzpic=0]{{1,-1},{2,-2},{3,-3},{4,-4}}
			\node at(0.85,0.5){\blue{$\scriptscriptstyle{a}$}};
		\end{tikzpicture}
		\begin{tikzpicture}
			\vpartition[type=2,tkzpic=0]{{1,-1},{2,-2},{3,-3},{4,-4}}
			\node at(1.35,0.5){\blue{$\scriptscriptstyle{a}$}};
		\end{tikzpicture}
	\fi\\[-0.35cm]
}

\newcommand{\figtwesix}{
	\centering
	\ifnum\aspdf=1
		\includegraphics[scale=\figscale]{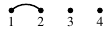}\qquad
		\includegraphics[scale=\figscale]{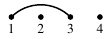}\qquad
		\includegraphics[scale=\figscale]{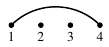}\qquad
		\includegraphics[scale=\figscale]{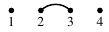}\qquad
		\includegraphics[scale=\figscale]{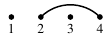}\qquad
		\includegraphics[scale=\figscale]{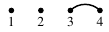}
	\else
		\arcpartition[num=4]{{1,2}}
		\arcpartition[num=4]{{1,3}}
		\arcpartition[num=4]{{1,4}}
		\arcpartition[num=4]{{2,3}}
		\arcpartition[num=4]{{2,4}}
		\arcpartition[num=4]{{3,4}}
	\fi
}

\newcommand{\figtwefiv}{
	\centering
	\ifnum\aspdf=1$
		\vcdraw{\includegraphics[scale=\figscale]{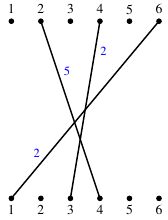}}\,\,\,\,\,=\,\,\,\,\,
		\vcdraw{\includegraphics[scale=\figscale]{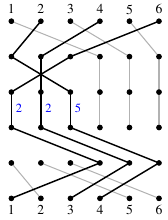}}
	$\else
		\begin{tikzpicture}		
			\vpartition[type=2,height=3,tkzpic=0]{{2,-4},{4,-3},{6,-1}}
			\node at(0.43,0.76){\blue{$\scriptscriptstyle{2}$}};
			\node at(0.94,2.16){\blue{$\scriptscriptstyle{5}$}};
			\node at(1.56,2.49){\blue{$\scriptscriptstyle{2}$}};
		\end{tikzpicture}
		\begin{tikzpicture}
			\tie[style=solid,color=black!30,height=0.6]{{1,5},{4,4},{4,3},{4,2}}
			\tie[style=solid,color=black!30,height=0.6]{{3,5},{5,4},{5,3},{5,2}}
			\tie[style=solid,color=black!30,height=0.6]{{5,5},{6,4},{6,3},{6,2}}
			\tie[style=solid,color=black!30,height=0.6]{{1,1},{2,0}}
			\tie[style=solid,color=black!30,height=0.6]{{2,1},{5,0}}
			\tie[style=solid,color=black!30,height=0.6]{{3,1},{6,0}}
			\vpartition[type=1,height=0.6,tkzpic=0,floor=4]{{2,-1},{4,-2},{6,-3}}
			\vpartition[type=0,height=0.6,tkzpic=0,floor=3,bullb=0,bulla=0]{{1,-3},{2,-2},{3,-1}}
			\vpartition[type=0,height=0.6,tkzpic=0,floor=2,bullb=0,nstr=6]{{1,-1},{2,-2},{3,-3}}
			\vpartition[type=0,height=0.6,tkzpic=0,floor=1,bullb=0]{{1,-4},{2,-5},{3,-6}}
			\vpartition[type=-1,height=0.6,tkzpic=0]{{-1,4},{-3,5},{-4,6}}
			\node at(0.13,1.52){\blue{$\scriptscriptstyle{2}$}};
			\node at(0.63,1.52){\blue{$\scriptscriptstyle{2}$}};
			\node at(1.13,1.52){\blue{$\scriptscriptstyle{5}$}};
		\end{tikzpicture}
	\fi
}

\newcommand{\figtwefou}{
	\centering
	\ifnum\aspdf=1$\begin{array}{cccccc}
		\,\,\includegraphics[scale=\figscale]{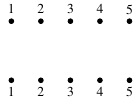}\,\,&
		\,\,\includegraphics[scale=\figscale]{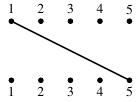}\,\,&
		\,\,\includegraphics[scale=\figscale]{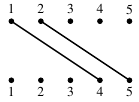}\,\,&
		\,\,\includegraphics[scale=\figscale]{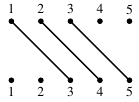}\,\,&
		\,\,\includegraphics[scale=\figscale]{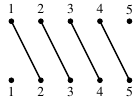}\,\,&
		\,\,\includegraphics[scale=\figscale]{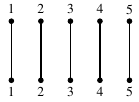}\,\,\\
		\,_{\nu^5\,\,=\,\,\tt00000}&\,_{\nu^4\,\,=\,\,\tt50000}&\,_{\nu^3\,\,=\,\,\tt45000}&\,_{\nu^2\,\,=\,\,\tt34500}&\,_{\nu^1\,\,=\,\,\tt23450}&\,_{\nu^0\,\,=\,\,\tt12345}
	\end{array}$\else
		\vpartition[type=2,nstr=5]{1}
		\vpartition[type=2,nstr=5]{{1,-5}}
		\vpartition[type=2,nstr=5]{{1,-4},{2,-5}}
		\vpartition[type=2,nstr=5]{{1,-3},{2,-4},{3,-5}}
		\vpartition[type=2,nstr=5]{{1,-2},{2,-3},{3,-4},{4,-5}}
		\vpartition[type=2,nstr=5]{{1,-1},{2,-2},{3,-3},{4,-4},{5,-5}}
	\fi
}

\newcommand{\figtwethr}{
	\centering
	\ifnum\aspdf=1
		\includegraphics[scale=\figscale]{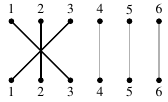}
	\else
		\begin{tikzpicture}
			\tie[style=solid,color=black!30]{{4,1},{4,0}}
			\tie[style=solid,color=black!30]{{5,1},{5,0}}
			\tie[style=solid,color=black!30]{{6,1},{6,0}}
			\vpartition[type=2,tkzpic=0,nstr=6]{{1,-3},{2,-2},{3,-1}}
		\end{tikzpicture}
	\fi
}

\newcommand{\figtwetwo}{
	\centering
	\ifnum\aspdf=1
		\includegraphics[scale=\figscale]{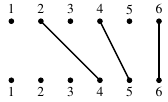}
	\else
		\begin{tikzpicture}
			\vpartition[type=2,tkzpic=0]{{2,-4},{4,-5},{6,-6}}
		\end{tikzpicture}
	\fi
}

\newcommand{\figtweone}{
	\centering
	\ifnum\aspdf=1
		\includegraphics[scale=\figscale]{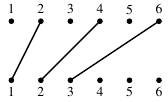}
	\else
		\begin{tikzpicture}
			\vpartition[type=2,tkzpic=0]{{2,-1},{4,-2},{6,-3}}
		\end{tikzpicture}
	\fi
}

\newcommand{\figtwenty}{
	\centering
	\ifnum\aspdf=1
		\includegraphics[scale=\figscale]{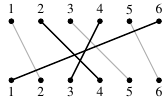}
	\else
		\begin{tikzpicture}
			\tie[style=solid,color=black!30]{{1,1},{2,0}}
			\tie[style=solid,color=black!30]{{3,1},{5,0}}
			\tie[style=solid,color=black!30]{{5,1},{6,0}}
			\vpartition[type=2,tkzpic=0]{{2,-4},{4,-3},{6,-1}}
		\end{tikzpicture}
	\fi
}

\newcommand{\figntn}{
	\centering
	\ifnum\aspdf=1$\begin{array}{ccccccc}
		\,\,\includegraphics[scale=\figscale]{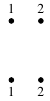}\,\,&
		\,\,\includegraphics[scale=\figscale]{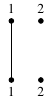}\,\,&
		\,\,\includegraphics[scale=\figscale]{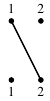}\,\,&
		\,\,\includegraphics[scale=\figscale]{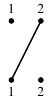}\,\,&
		\,\,\includegraphics[scale=\figscale]{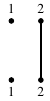}\,\,&
		\,\,\includegraphics[scale=\figscale]{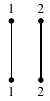}\,\,&
		\,\,\includegraphics[scale=\figscale]{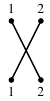}\,\,\\
		\,_{\tt00}&\,_{\tt10}&\,_{\tt20}&\,_{\tt01}&\,_{\tt02}&\,_{\tt12}&\,_{\tt21}
	\end{array}$\else
		\vpartition[type=2,nstr=2]{{1}}
		\vpartition[type=2,nstr=2]{{1,-1}}
		\vpartition[type=2,nstr=2]{{1,-2}}
		\vpartition[type=2,nstr=2]{{2,-1}}
		\vpartition[type=2,nstr=2]{{2,-2}}
		\vpartition[type=2,nstr=2]{{1,-1},{2,-2}}
		\vpartition[type=2,nstr=2]{{1,-2},{2,-1}}
	\fi
}

\newcommand{\figetn}{
	\centering
	\ifnum\aspdf=1$\begin{array}{ccccc}
		\,\,\includegraphics[scale=\figscale]{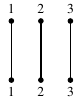}\,\,&
		\,\,\includegraphics[scale=\figscale]{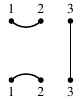}\,\,&
		\,\,\includegraphics[scale=\figscale]{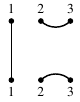}\,\,&
		\,\,\includegraphics[scale=\figscale]{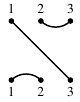}\,\,&
		\,\,\includegraphics[scale=\figscale]{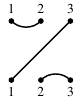}\,\,
	\end{array}$\else
		\vpartition[type=2]{{1,-1},{2,-2},{3,-3}}
		\vpartition[type=2]{{1,2},{-1,-2},{3,-3}}
		\vpartition[type=2]{{1,-1},{2,3},{-2,-3}}
		\vpartition[type=2]{{1,-3},{2,3},{-1,-2}}
		\vpartition[type=2]{{3,-1},{2,1},{-3,-2}}
	\fi
}

\newcommand{\figsvn}{
	\centering
	\ifnum\aspdf=1$\begin{array}{cccccc}
		\,\,\includegraphics[scale=\figscale]{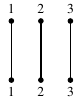}\,\,&
		\,\,\includegraphics[scale=\figscale]{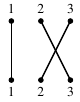}\,\,&
		\,\,\includegraphics[scale=\figscale]{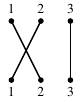}\,\,&
		\,\,\includegraphics[scale=\figscale]{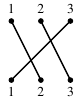}\,\,&
		\,\,\includegraphics[scale=\figscale]{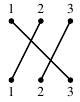}\,\,&
		\,\,\includegraphics[scale=\figscale]{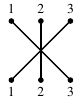}\,\,\\
		\,_{\tt123}&\,_{\tt132}&\,_{\tt213}&\,_{\tt231}&\,_{\tt312}&\,_{\tt321}
	\end{array}$\else
		\permutation{1,2,3}
		\permutation{1,3,2}
		\permutation{2,1,3}
		\permutation{2,3,1}
		\permutation{3,1,2}
		\permutation{3,2,1}
	\fi
}

\newcommand{\figsxn}{
	\centering
	\ifnum\aspdf=1
		\includegraphics[scale=\figscale]{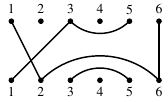}
	\else
		\vpartition[type=2]{{-1,3,5},{1,-2,-6,6},{-3,-5}}
	\fi
}

\newcommand{\figffn}{
	\centering
	\ifnum\aspdf=1$
		\vcdraw{\includegraphics[scale=\figscale]{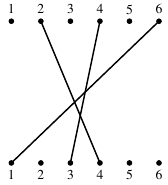}}\,\,\,\,\,=\,\,\,\,\,
		\vcdraw{\includegraphics[scale=\figscale]{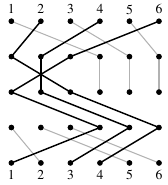}}
	$\else
		\vpartition[type=2,height=2.4]{{2,-4},{4,-3},{6,-1}}
		\begin{tikzpicture}
			\tie[style=solid,color=blue,height=0.6]{{1,4},{4,3},{4,2}}
			\tie[style=solid,color=blue,height=0.6]{{3,4},{5,3},{5,2}}
			\tie[style=solid,color=blue,height=0.6]{{5,4},{6,3},{6,2}}
			\tie[style=solid,color=blue,height=0.6]{{1,1},{2,0}}
			\tie[style=solid,color=blue,height=0.6]{{2,1},{5,0}}
			\tie[style=solid,color=blue,height=0.6]{{3,1},{6,0}}
			\vpartition[type=1,height=0.6,tkzpic=0,floor=3]{{2,-1},{4,-2},{6,-3}}
			\vpartition[type=0,height=0.6,tkzpic=0,floor=2,bullb=0,bulla=0]{{1,-3},{2,-2},{3,-1}}
			\vpartition[type=0,height=0.6,tkzpic=0,floor=1,bullb=0]{{1,-4},{2,-5},{3,-6}}
			\vpartition[type=-1,height=0.6,tkzpic=0]{{-1,4},{-3,5},{-4,6}}
		\end{tikzpicture}
		\begin{tikzpicture}
			\tie[style=solid,color=black!30,height=0.6]{{1,4},{4,3},{4,2}}
			\tie[style=solid,color=black!30,height=0.6]{{3,4},{5,3},{5,2}}
			\tie[style=solid,color=black!30,height=0.6]{{5,4},{6,3},{6,2}}
			\tie[style=solid,color=black!30,height=0.6]{{1,1},{2,0}}
			\tie[style=solid,color=black!30,height=0.6]{{2,1},{5,0}}
			\tie[style=solid,color=black!30,height=0.6]{{3,1},{6,0}}
			\vpartition[type=1,height=0.6,tkzpic=0,floor=3]{{2,-1},{4,-2},{6,-3}}
			\vpartition[type=0,height=0.6,tkzpic=0,floor=2,bullb=0,bulla=0]{{1,-3},{2,-2},{3,-1}}
			\vpartition[type=0,height=0.6,tkzpic=0,floor=1,bullb=0]{{1,-4},{2,-5},{3,-6}}
			\vpartition[type=-1,height=0.6,tkzpic=0]{{-1,4},{-3,5},{-4,6}}
		\end{tikzpicture}
	\fi
}

\newcommand{\figftn}{
	\centering
	\ifnum\aspdf=1
		\includegraphics[scale=\figscale]{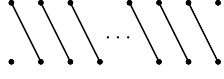}
	\else
		\begin{tikzpicture}
			\vpartition[type=0,nstr=5,bulla=0,bullb=0,tkzpic=0]{{1,-2},{2,-3},{3,-4},{5,-6},{6,-7},{7,-8}}
			\tie[color=black,style=solid]{{1,1},{1,1}}
			\tie[color=black,style=solid]{{2,1},{2,1}}
			\tie[color=black,style=solid]{{3,1},{3,1}}
			\tie[color=black,style=solid]{{5,1},{5,1}}
			\tie[color=black,style=solid]{{6,1},{6,1}}
			\tie[color=black,style=solid]{{7,1},{7,1}}
			\tie[color=black,style=solid]{{8,1},{8,1}}
			\tie[color=black,style=solid]{{1,0},{1,0}}
			\tie[color=black,style=solid]{{2,0},{2,0}}
			\tie[color=black,style=solid]{{3,0},{3,0}}
			\tie[color=black,style=solid]{{4,0},{4,0}}
			\tie[color=black,style=solid]{{6,0},{6,0}}
			\tie[color=black,style=solid]{{7,0},{7,0}}
			\tie[color=black,style=solid]{{8,0},{8,0}}
			\node at(1.84,0.4){$\cdots$};
		\end{tikzpicture}
	\fi\\[-0.3cm]
}

\newcommand{\figttn}{
	\centering
	\ifnum\aspdf=1\[
	\begin{array}{rclll}
		\begin{array}{c}\begin{psmallmatrix}
			\zerc&\zerc&\zerc&\zerc&\zerc\\
			\zerc&\zerc&\zerc&\zerc&1\\
			1&\zerc&\zerc&\zerc&\zerc\\
			\zerc&\zerc&1&\zerc&\zerc\\
			\zerc&\zerc&\zerc&\zerc&\zerc
		\end{psmallmatrix}\\[-0.1cm]_{\tt05130}\\[-0.11cm]\end{array}
		\vcdraw{\includegraphics{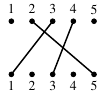}}
		&\stackrel{s_1}{\longrightarrow}&
		\begin{array}{c}\begin{psmallmatrix}
			\zerc&\zerc&\zerc&\zerc&1\\
			\zerc&\zerc&\zerc&\zerc&\zerc\\
			1&\zerc&\zerc&\zerc&\zerc\\
			\zerc&\zerc&1&\zerc&\zerc\\
			\zerc&\zerc&\zerc&\zerc&\zerc
		\end{psmallmatrix}\\[-0.1cm]_{\tt{\red50}130}\\[-0.11cm]\end{array}
		\vcdraw{\includegraphics{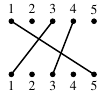}}
		&\stackrel{s_2}{\longrightarrow}&
		\begin{array}{c}\begin{psmallmatrix}
			\zerc&\zerc&\zerc&\zerc&1\\
			1&\zerc&\zerc&\zerc&\zerc\\
			\zerc&\zerc&\zerc&\zerc&\zerc\\
			\zerc&\zerc&1&\zerc&\zerc\\
			\zerc&\zerc&\zerc&\zerc&\zerc
		\end{psmallmatrix}\\[-0.1cm]_{\tt5{\red10}30}\\[-0.11cm]\end{array}
		\vcdraw{\includegraphics{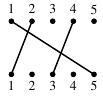}}\\[0.9cm]
		&\stackrel{s_3}{\longrightarrow}&
		\begin{array}{c}\begin{psmallmatrix}
			\zerc&\zerc&\zerc&\zerc&1\\
			1&\zerc&\zerc&\zerc&\zerc\\
			\zerc&\zerc&1&\zerc&\zerc\\
			\zerc&\zerc&\zerc&\zerc&\zerc\\
			\zerc&\zerc&\zerc&\zerc&\zerc
		\end{psmallmatrix}\\[-0.1cm]_{\tt51{\red30}0}\\[-0.11cm]\end{array}		
		\vcdraw{\includegraphics{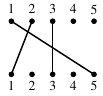}}
		&\stackrel{s_3}{\longrightarrow}&
		\begin{array}{c}\begin{psmallmatrix}
			\zerc&\zerc&\zerc&\zerc&1\\
			1&\zerc&\zerc&\zerc&\zerc\\
			\zerc&\zerc&\zerc&1&\zerc\\
			\zerc&\zerc&\zerc&\zerc&\zerc\\
			\zerc&\zerc&\zerc&\zerc&\zerc
		\end{psmallmatrix}\\[-0.1cm]_{\tt51{\blue40}0}\\[-0.11cm]\end{array}
		\vcdraw{\includegraphics{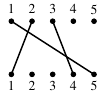}}\\[0.9cm]
		&\stackrel{s_1}{\longrightarrow}&
		\begin{array}{c}\begin{psmallmatrix}
			\zerc&\zerc&\zerc&\zerc&1\\
			\zerc&1&\zerc&\zerc&\zerc\\
			\zerc&\zerc&\zerc&1&\zerc\\
			\zerc&\zerc&\zerc&\zerc&\zerc\\
			\zerc&\zerc&\zerc&\zerc&\zerc
		\end{psmallmatrix}\\[-0.1cm]_{\tt{\blue52}400}\\[-0.11cm]\end{array}
		\vcdraw{\includegraphics{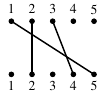}}
		&\stackrel{s_2}{\longrightarrow}&
		\begin{array}{c}\begin{psmallmatrix}
			\zerc&\zerc&\zerc&\zerc&1\\
			\zerc&\zerc&1&\zerc&\zerc\\
			\zerc&\zerc&\zerc&1&\zerc\\
			\zerc&\zerc&\zerc&\zerc&\zerc\\
			\zerc&\zerc&\zerc&\zerc&\zerc
		\end{psmallmatrix}\\[-0.1cm]_{\tt5{\blue34}00}\\[-0.11cm]\end{array}
		\vcdraw{\includegraphics{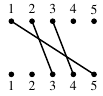}}\\[0.9cm]
		&\stackrel{s_1}{\longrightarrow}&
		\begin{array}{c}\begin{psmallmatrix}
			\zerc&\zerc&1&\zerc&\zerc\\
			\zerc&\zerc&\zerc&\zerc&1\\
			\zerc&\zerc&\zerc&1&\zerc\\
			\zerc&\zerc&\zerc&\zerc&\zerc\\
			\zerc&\zerc&\zerc&\zerc&\zerc
		\end{psmallmatrix}\\[-0.1cm]_{\tt{\red35}400}\\[-0.11cm]\end{array}
		\vcdraw{\includegraphics{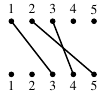}}
		&\stackrel{s_2}{\longrightarrow}&
		\begin{array}{c}\begin{psmallmatrix}
			\zerc&\zerc&1&\zerc&\zerc\\
			\zerc&\zerc&\zerc&1&\zerc\\
			\zerc&\zerc&\zerc&\zerc&1\\
			\zerc&\zerc&\zerc&\zerc&\zerc\\
			\zerc&\zerc&\zerc&\zerc&\zerc
		\end{psmallmatrix}\\[-0.1cm]_{\tt3{\red45}00}\\[-0.11cm]\end{array}
		\vcdraw{\includegraphics{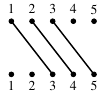}}
	\end{array}
	\]\else
		\begin{tikzpicture}
			\vpartition[tkzpic=0,type=2,nstr=5,width=0.35,height=0.9]{{2,-5},{3,-1},{4,-3}}
		\end{tikzpicture}
		\begin{tikzpicture}
			\vpartition[tkzpic=0,type=2,nstr=5,width=0.35,height=0.9]{{1,-5},{3,-1},{4,-3}}
		\end{tikzpicture}
		\begin{tikzpicture}
			\vpartition[tkzpic=0,type=2,nstr=5,width=0.35,height=0.9]{{1,-5},{2,-1},{4,-3}}
		\end{tikzpicture}
		\begin{tikzpicture}
			\vpartition[tkzpic=0,type=2,nstr=5,width=0.35,height=0.9]{{1,-5},{2,-1},{3,-3}}
		\end{tikzpicture}
		\begin{tikzpicture}
			\vpartition[tkzpic=0,type=2,nstr=5,width=0.35,height=0.9]{{1,-5},{2,-1},{3,-4}}
		\end{tikzpicture}
		\begin{tikzpicture}
			\vpartition[tkzpic=0,type=2,nstr=5,width=0.35,height=0.9]{{1,-5},{2,-2},{3,-4}}
		\end{tikzpicture}
		\begin{tikzpicture}
			\vpartition[tkzpic=0,type=2,nstr=5,width=0.35,height=0.9]{{1,-5},{2,-3},{3,-4}}
		\end{tikzpicture}
		\begin{tikzpicture}
			\vpartition[tkzpic=0,type=2,nstr=5,width=0.35,height=0.9]{{2,-5},{1,-3},{3,-4}}
		\end{tikzpicture}
		\begin{tikzpicture}
			\vpartition[tkzpic=0,type=2,nstr=5,width=0.35,height=0.9]{{3,-5},{1,-3},{2,-4}}
		\end{tikzpicture}
	\fi\\[-0.3cm]
}

\newcommand{\figtwe}{
	\centering
	\ifnum\aspdf=1\[
		\scalebox{0.8}{$\begin{pmatrix}
			\zerc&\zerc&\zerc&\zerc&\zerc\\\zerc&\zerc&\zerc&\zerc&2\\5&\zerc&\zerc&\zerc&\zerc\\\zerc&\zerc&3&\zerc&\zerc\\\zerc&\zerc&\zerc&\zerc&\zerc
		\end{pmatrix}$}\quad\mapsto\quad
		\!\!\!\begin{array}{c}
			\vcdraw{\includegraphics{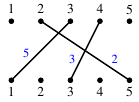}}\\\,_{\tt02530-05130}\\[-0.1cm]
		\end{array}
	\]\else
		\begin{tikzpicture}
			\vpartition[type=2,tkzpic=0]{{2,-5},{3,-1},{4,-3}}
			\node at(0.25,0.45){\blue{$\scriptscriptstyle{5}$}};
			\node at(1.03,0.35){\blue{$\scriptscriptstyle{3}$}};
			\node at(1.75,0.35){\blue{$\scriptscriptstyle{2}$}};
		\end{tikzpicture}
	\fi\\[-0.3cm]
}

\newcommand{\figele}{
	\centering
	\ifnum\aspdf=1\[\begin{array}{cccccc}
		\begin{psmallmatrix}
			\zerc&\zerc&\zerc&\zerc&\zerc\\
			\zerc&\zerc&\zerc&\zerc&\zerc\\
			\zerc&\zerc&\zerc&\zerc&\zerc\\
			\zerc&\zerc&\zerc&\zerc&\zerc\\
			\zerc&\zerc&\zerc&\zerc&\zerc
		\end{psmallmatrix}
		&
		\begin{psmallmatrix}
			\zerc&\zerc&\zerc&\zerc&1\\
			\zerc&\zerc&\zerc&\zerc&\zerc\\
			\zerc&\zerc&\zerc&\zerc&\zerc\\
			\zerc&\zerc&\zerc&\zerc&\zerc\\
			\zerc&\zerc&\zerc&\zerc&\zerc
		\end{psmallmatrix}
		&
		\begin{psmallmatrix}
			\zerc&\zerc&\zerc&1&\zerc\\
			\zerc&\zerc&\zerc&\zerc&1\\
			\zerc&\zerc&\zerc&\zerc&\zerc\\
			\zerc&\zerc&\zerc&\zerc&\zerc\\
			\zerc&\zerc&\zerc&\zerc&\zerc
		\end{psmallmatrix}
		&
		\begin{psmallmatrix}
			\zerc&\zerc&1&\zerc&\zerc\\
			\zerc&\zerc&\zerc&1&\zerc\\
			\zerc&\zerc&\zerc&\zerc&1\\
			\zerc&\zerc&\zerc&\zerc&\zerc\\
			\zerc&\zerc&\zerc&\zerc&\zerc
		\end{psmallmatrix}
		&
		\begin{psmallmatrix}
			\zerc&1&\zerc&\zerc&\zerc\\
			\zerc&\zerc&1&\zerc&\zerc\\
			\zerc&\zerc&\zerc&1&\zerc\\
			\zerc&\zerc&\zerc&\zerc&1\\
			\zerc&\zerc&\zerc&\zerc&\zerc
		\end{psmallmatrix}
		&
		\begin{psmallmatrix}
			1&\zerc&\zerc&\zerc&\zerc\\
			\zerc&1&\zerc&\zerc&\zerc\\
			\zerc&\zerc&1&\zerc&\zerc\\
			\zerc&\zerc&\zerc&1&\zerc\\
			\zerc&\zerc&\zerc&\zerc&1
		\end{psmallmatrix}
		\\[0.2cm]_{\nu_0\,=\,\nu^5}&_{\nu_1\,=\,\nu^4}&_{\nu_2\,=\,\nu^3}&_{\nu_3\,=\,\nu^2}&_{\nu_4\,=\,\nu^1}&_{\nu_5\,=\,\nu^0}\\[0.2cm]
	\end{array}\]\else
		\vpartition[type=2,nstr=5,width=0.35,height=0.9]{1}
		\vpartition[type=2,nstr=5,width=0.35,height=0.9]{{1,-5}}
		\vpartition[type=2,nstr=5,width=0.35,height=0.9]{{1,-4},{2,-5}}
		\vpartition[type=2,nstr=5,width=0.35,height=0.9]{{1,-3},{2,-4},{3,-5}}
		\vpartition[type=2,nstr=5,width=0.35,height=0.9]{{1,-2},{2,-3},{3,-4},{4,-5}}
		\vpartition[type=2,nstr=5,width=0.35,height=0.9]{{1,-1},{2,-2},{3,-3},{4,-4},{5,-5}}
	\fi\\[-0.3cm]
}

\newcommand{\figten}{
	\centering
	\ifnum\aspdf=1\[
		\scalebox{0.8}{$\begin{pmatrix}
			\zerc&\zerc&\zerc&\zerc&\zerc\\\zerc&\zerc&\zerc&\zerc&1\\1&\zerc&\zerc&\zerc&\zerc\\\zerc&\zerc&1&\zerc&\zerc\\\zerc&\zerc&\zerc&\zerc&\zerc
		\end{pmatrix}$}\quad\mapsto\quad
		\!\!\!\begin{array}{cc}
			\vcdraw{\includegraphics{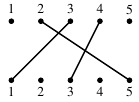}}\\\,_{\tt05130}\\[-0.1cm]
		\end{array}
	\]\else
		\vpartition[type=2]{{2,-5},{3,-1},{4,-3}}
	\fi\\[-0.3cm]
}

\newcommand{\fignin}{
	\centering
	\ifnum\aspdf=1\[
		\vcdraw{\includegraphics{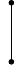}}\cdots
		\vcdraw{\includegraphics{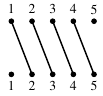}}\,\,\cdots
		\vcdraw{\includegraphics{pics/line.pdf}}
	\]\else
		\vpartition[type=0]{{1,-1}}
		\begin{tikzpicture}
			\vpartition[type=0,bend=60,tkzpic=0]{{1,-1},{2,3},{-2,-3},{4,-4}}
			\tie[snake=true,snakelen=3,style=solid]{{2.5,0.15},{2.5,0.85}}
		\end{tikzpicture}
	\fi
}

\newcommand{\figeig}{
	\centering
	\ifnum\aspdf=1
		\includegraphics[scale=\figscale]{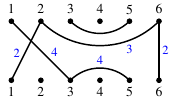}
	\else
		\begin{tikzpicture}
			\vpartition[type=2,tkzpic=0]{{-1,2,6,-6},{1,-3,-5},{3,5}}
			\node at(0.09,0.45){\blue{$\scriptscriptstyle{2}$}};
			\node at(0.73,0.48){\blue{$\scriptscriptstyle{4}$}};
			\node at(1.5,0.34){\blue{$\scriptscriptstyle{4}$}};
			\node at(2,0.52){\blue{$\scriptscriptstyle{3}$}};
			\node at(2.61,0.5){\blue{$\scriptscriptstyle{2}$}};
		\end{tikzpicture}
	\fi
}

\newcommand{\figsev}{
	\centering
	\ifnum\aspdf=1
		\includegraphics[scale=\figscale]{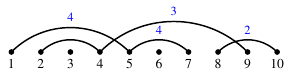}
	\else
		\begin{tikzpicture}
			\arcpartition[tkzpic=0]{{1,5,7},{2,4,9},{8,10}}
			\node at(1,0.59){\blue{$\scriptscriptstyle{2}$}};
			\node at(2.75,0.68){\blue{$\scriptscriptstyle{3}$}};
			\node at(2.5,0.36){\blue{$\scriptscriptstyle{2}$}};
			\node at(4,0.37){\blue{$\scriptscriptstyle{1}$}};
		\end{tikzpicture}
	\fi
}

\newcommand{\figsix}{
	\centering
	\ifnum\aspdf=1$
		\vcdraw{\includegraphics[scale=\figscale]{pics/line.pdf}}\cdots
		\vcdraw{\includegraphics[scale=\figscale]{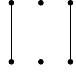}}\,\,\cdots
		\vcdraw{\includegraphics[scale=\figscale]{pics/line.pdf}}
	$\else
		\vpartition[type=0]{{1,-1},{3,-3}}
	\fi
}

\newcommand{\figfiv}{
	\centering
	\ifnum\aspdf=1$
		\vcdraw{\includegraphics[scale=\figscale]{pics/line.pdf}}\cdots
		\vcdraw{\includegraphics[scale=\figscale]{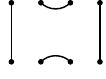}}\,\,\cdots
		\vcdraw{\includegraphics[scale=\figscale]{pics/line.pdf}}
	$\else
		\vpartition[type=0]{{1,-1},{2,3},{-2,-3},{4,-4}}
	\fi
}

\newcommand{\figfou}{
	\centering
	\ifnum\aspdf=1$
		\vcdraw{\includegraphics[scale=\figscale]{pics/line.pdf}}\cdots
		\vcdraw{\includegraphics[scale=\figscale]{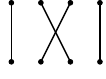}}\,\,\cdots
		\vcdraw{\includegraphics[scale=\figscale]{pics/line.pdf}}
	$\else
		\vpartition[type=0]{{1,-1}}
		\vpartition[type=0]{{1,-1},{2,-3},{3,-2},{4,-4}}
	\fi
}

\newcommand{\figthr}{
	\centering
	\ifnum\aspdf=1$
		\vcdraw{\includegraphics[scale=\figscale]{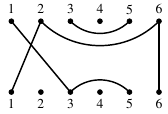}}\,\,\,\ast\,
		\vcdraw{\includegraphics[scale=\figscale]{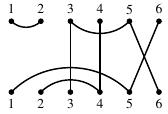}}\,\,\,=\,\,
		\vcdraw{\includegraphics[scale=\figscale]{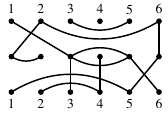}}\,\,\,=\,\,
		\vcdraw{\includegraphics[scale=\figscale]{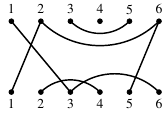}}
	$\else
		\vpartition[height=1.2,type=2]{{-1,2,6,-6},{1,-3,-5},{3,5}}
		\vpartition[height=1.2,type=2]{{-3,3,5,-6},{6,-5,-1},{-2,-4,4},{1,2}}
		\begin{tikzpicture}
			\vpartition[height=0.6,floor=1,tkzpic=0,bullb=0,type=1,bend=31]{{-1,2,6,-6},{1,-3,-5},{3,5}}
			\vpartition[height=0.6,tkzpic=0,type=-1,bend=31]{{-3,3,5,-6},{6,-5,-1},{-2,-4,4},{1,2}}
		\end{tikzpicture}
		\vpartition[height=1.2,type=2]{{1,-3,-6},{-1,2,6,-5},{3,5},{-2,-4}}
	\fi
}

\newcommand{\figtwo}{
	\centering
	\ifnum\aspdf=1$
		\vcdraw{\includegraphics[scale=\figscale]{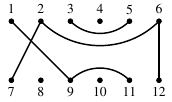}}\,\,\,\to\,\,
		\vcdraw{\includegraphics[scale=\figscale]{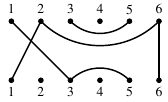}}
	$\else
		\vpartition{{-1,2,6,-6},{1,-3,-5},{3,5}}
		\vpartition[type=2]{{-1,2,6,-6},{1,-3,-5},{3,5}}
	\fi
}

\newcommand{\figone}{
	\centering
	\ifnum\aspdf=1
		\includegraphics[scale=\figscale]{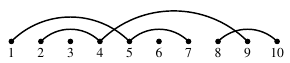}
	\else
		\arcpartition{{1,5,7},{2,4,9},{8,10}}
	\fi
}

\newcommand{\figzer}{
	\centering
	\ifnum\aspdf=1
		\includegraphics[scale=\figscale]{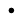}\qquad
		\includegraphics[scale=\figscale]{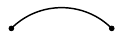}
	\else
		\arcpartition[num=1,type=0]{1}
		\arcpartition[num=1,type=0,width=1.7]{{1,2}}
	\fi
}

\begin{document}

\begin{abstract}
We introduce and study the framed rook algebra, a structure that unifies two significant generalizations of the Iwahori--Hecke algebra. The first one, introduced by Solomon, extends the Hecke algebra to the full matrix monoid, yielding the rook monoid algebra. The second one, developed by Yokonuma, replaces the Borel subgroup with the unipotent subgroup, resulting in the Yokonuma--Hecke algebra. Our concrete algebra is constructed from the double cosets of the unipotent subgroup within the full matrix monoid. We show that this double coset decomposition is indexed by the framed symmetric inverse monoid. We also define the Rook Yokonuma--Hecke algebra as an abstract structure using generators and relations. We then prove the main isomorphism theorem, which establishes that this abstract algebra is isomorphic to the framed rook algebra under a specific parameter specialization. To complete our characterization, we provide a faithful representation on a tensor space and establish a linear basis for the Rook Yokonuma--Hecke algebra.
\end{abstract}

\maketitle

%\setcounter{tocdepth}{2}
%\begin{center}\begin{minipage}{12cm}\tableofcontents\end{minipage}\end{center}

\section{Introduction}

The Iwahori--Hecke algebra of type $A_{n-1}$, denoted by $\H_n(u)$, is a central object in algebraic combinatorics, representation theory, and knot theory. Originally, it arose from the study of finite Chevalley groups \cite{Iw64}. Classically, let $G=\GLbf_n(\Fbb_q)$ be the general linear group over a finite field $\Fbb_q$, and let $B$ be the standard Borel subgroup of upper triangular matrices. The algebra $\H_n(u)$ can be realized as the endomorphism algebra $\End_G(\Ind_B^G(1))$ of the permutation representation of $G$ on the cosets of $B$ \cite{CuRe81}. A fundamental result, based on the Bruhat decomposition of $G$, establishes that the double cosets of $B$ in $G$ are indexed by the Weyl group $W\simeq\S_n$, the symmetric group on $n$ letters \cite{Bru56,Che55}.

This classical construction has been generalized in two distinct but significant directions over the last decades.

The first direction involves extending the underlying algebraic structure from a group to a monoid \cite{Ren86}. In his seminal works, Solomon investigated the Iwahori algebra associated with the full monoid of $n\times n$ matrices, $M:=\Mbf_n(\Fbb_q)$, rather than the group of units $G$ \cite{So90}. He considered the double cosets of the Borel subgroup $B$ within $M$ and demonstrated that the decomposition $M=\bigsqcup_\sigma B\sigma B$ is indexed by the \emph{rook monoid} $\R$, also known as the symmetric inverse monoid $\I_n$ \cite{Mnn57}. The resulting algebra, often referred to as the rook monoid algebra, serves as a semigroup analogue to the Hecke algebra and has been extensively studied for its representation theory and combinatorial properties \cite{So90,So02,HaRa04,BiRaYi12,HaDe14}. In \cite{So04}, Solomon gives a presentation for the rook monoid algebra and provides a tensor representation for such algebra. Subsequently, in \cite{Ha04}, Halverson provides an alternative presentation for this algebra and studies its representation theory.

The second direction involves modifying the observing subgroup. In 1967, Yokonuma introduced a generalization of the Hecke algebra by replacing the Borel subgroup $B$ with its unipotent radical $U$, where $B=UT$ and $T$ is the maximal torus of diagonal matrices \cite{Yo67}. The resulting algebra, known as the \emph{Yokonuma--Hecke algebra} $\Y_{d,n}(u)$, arises from the double cosets in $G$ \cite{Thi04}. Unlike the Iwahori--Hecke algebra, the Yokonuma--Hecke algebra supports a modular structure related to the "framing" of the underlying braid groups, which has led to significant applications in representation theory \cite{Ju98,ChPo14,DaDou17,EsRyH18} and knot theory. More precisely, the Iwahori--Hecke algebra is used to derive the well known HOMFLY--PT polynomial \cite{Jn87}. On the other hand, the Yokonuma--Hecke algebra can also be used to derive invariants of classical and framed links \cite{Ju04,ChLa13,ChJaKaLa16,GoLa17,CJKL18}. It is worth noting that the invariant obtained from $\Y_{d,n}(u)$ distinguishes pairs of links that are  not distinguished by the HOMFLY--PT polynomial.

Motivated for this, Juyumaya and Lambropoulou introduce the notion of \emph{framization of an algebra} in \cite{JuLa15}, which consists in adding framing generators to the algebra with the aim of finding new and powerful invariants of classical links by using the algebra obtained from this procedure. The Yokonuma--Hecke algebra is the prototype of a framization; indeed this algebra can be thought of as a framization of the Iwahori--Hecke algebra $\H_n(u)$. A more formal definition of this procedure was given recently in \cite{AiJuPa25}.

In this paper, we propose a synthesis of these two generalizations. We construct and study the algebra arising from the double cosets of the unipotent subgroup $U$ within the full matrix monoid $M$. We refer to this concrete structure as the \emph{framed rook algebra}, denoted by $H(M,U)$. Furthermore, we introduce an abstract algebra version, defined by generators and relations, which we call the \emph{Rook Yokonuma--Hecke algebra}, denoted by $\RY_{d,n}(u)$. This new algebra can be thought as a framization of the rook monoid algebra in the sense of \cite{AiJuPa25}. In particular, the algebra $\RY_{d,n}(u)$ generalizes both the rook monoid algebra and the Yokonuma--Hecke algebra.

Our approach is motivated by the observation that the transition from the Borel subgroup $B$ to the unipotent subgroup $U$ introduces an action of the torus $T\simeq(\Fbb_q^\times)^n$ on the double cosets $U\sigma U$. This action effectively "colors" or "frames" the combinatorial objects that index them. While Solomon's construction relies on the rook monoid $\R$, we show that the decomposition of the full matrix monoid $M$ into double cosets is naturally indexed by a generalized version of this monoid, which we call the \emph{generalized rook monoid} $\QQ$ (Subsection~\ref{148}). As we establish, $\QQ$ is isomorphic to a combinatorial structure given by coloured set partitions, which we call the \emph{framed symmetric inverse monoid} $\F_q(\I_n)$ (Subsection~\ref{149}). Geometrically, elements of this monoid can be visualized as partial permutations labelled by elements of the cyclic group $\Fbb_q^\times$.

The main contributions of this paper are as follows:

\begin{enumerate}

\item We explicitly define the framed rook algebra $H(M,U)$ as a convolution algebra of $U$-invariant functions on $M$, providing a basis indexed by the generalized rook monoid $\QQ$.

\item We introduce an abstract algebra by generators and relations, which we call the \emph{Rook Yokonuma--Hecke algebra}, denoted by $\RY_{d,n}(u)$. This presentation unifies the defining relations of the classical Yokonuma--Hecke algebra with the idempotent and projection relations characteristic of rook monoid algebras.

\item We prove a main isomorphism theorem (Theorem~\ref{057}), establishing that the concrete algebra $H(M,U)$ is isomorphic to the abstract algebra $\RY_{d,n}(u)$ under the specific parameter specialization $d=q-1$ and $u=q$. This result parallels Solomon's characterization of the Iwahori algebra of $M$.

\item Finally, we construct a faithful tensor space representation for the Rook Yokonuma--Hecke algebra (Theorem~\ref{151}), extending the classical tensor representations of the rook monoid algebra. Furthermore, we establish a linear basis for the Rook Yokonuma--Hecke algebra (Theorem~\ref{163}).
\end{enumerate}

The paper is organized as follows. Section~\ref{152} establishes the combinatorial framework necessary for our construction. We review partition monoids (Subsection~\ref{118}) and formally introduce the concept of framed monoids (Subsection~\ref{157}). Specifically, we define the framed symmetric inverse monoid (Subsection~\ref{149}), denoted by $\F_q(\I_n)$, and provide a presentation by generators and relations (Proposition~\ref{025}).

Section~\ref{153} focuses on the matrix interpretation of these combinatorial structures. We recall the classic rook monoid (Subsection~\ref{158}) and define the generalized rook monoid $\QQ$ as a submonoid of the full matrix monoid $M$ (Subsection~\ref{158}). A key observation in this section is the establishment of an explicit isomorphism between $\QQ$ and the framed symmetric inverse monoid $\F_q(\I_n)$, which allows us to translate geometric properties of diagrams into matrix operations. Furthermore, we analyze the double coset decomposition of $M$ with respect to the unipotent subgroup $U$ (Subsection~\ref{160}), showing that these double cosets are naturally indexed by elements of $\QQ$. We also study the length function and the structure of double classes (Subsection~\ref{159}). We extend the classical length function of the rook monoid to the generalized setting and analyze the convolution product of double cosets, providing explicit multiplication formulas that are crucial for the algebra structure.

Section~\ref{154} contains the main results of the paper. First, we define the concrete framed rook algebra $H(M,U)$ as the algebra of $U$-bi-invariant functions on $M$. Then, we introduce the Rook Yokonuma--Hecke algebra $\RY_{d,n}$ as an abstract algebra defined by generators and relations (Subsection~\ref{007}). We prove our main isomorphism theorem, demonstrating that $H(M,U)$ is isomorphic to the specialization of $\RY_{d,n}(u)$ at $d=q-1$ and $u=q$. We construct a faithful tensor space representation for $\RY_{d,n}$ and use it to establish a standard basis for the algebra, proving that its dimension is equal to the cardinality of the generalized rook monoid (Subsection~\ref{161}). Finally we give some other presentations for the algebra $\RY_{d,n}$ (Subsection~\ref{162}).

\subsubsection*{Notation and conventions}

Throughout this paper, let $\Fbb_q$ denote the finite field with $q$ elements, where $q$ is a power of an odd prime. We recall that the multiplicative group $\Fbb_q^\times$ is cyclic of order $q-1$, that is, $\Fbb_q^\times=\langle a\mid a^{\,q-1}=1\rangle$ for some $a\in\Fbb_q^\times$ (see \cite[Chapter V, Section 5]{Hun11}). We denote by $\Nbb$ the set of positive integers, and for any set $A$, we define $A_0:=A\cup\{0\}$. For integers $m,n\in\Z$, the discrete interval $\{m,\ldots,n\}$ is denoted by $[m,n]$, and we use the abbreviation $[n]$ for $[1,n]$. Furthermore, given a binary relation $R$ on a semigroup, we denote its congruence closure by $\bar{R}$.

\section{Framed monoids}\label{152}

In this section, we establish the combinatorial framework necessary to construct the framed rook algebra. Our approach begins with a review of set partitions and the partition monoid $\CC_n$ emphasizing their diagrammatic representation via strand diagrams. We recall the concept of coloured set partitions and introduce framed monoids by attaching weights from the cyclic group $\Fbb_q^\times$ to the arcs of these diagrams. This construction leads to the definition of the framed symmetric inverse monoid, denoted by $\F_q(\I_n)$. We conclude the section by providing a presentation of $\F_q(\I_n)$ via generators and relations, a result that is crucial for establishing the isomorphism theorem in Section~\ref{154}.

Throughout this section, let $A$ and $B$ denote nonempty sets of integers and $n$ a positive integer. We refer to singletons and to sets containing exactly two integers as \emph{points} and \emph{arcs}, respectively. See Figure~\ref{002}.\begin{figure}[H]
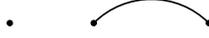

\figzer\caption{A point and an arc.}\label{002}
\end{figure}

\subsection{Set partitions}\label{155}

A \emph{set partition} of $A$ is a collection of pairwise disjoint nonempty subsets of $A$, called \emph{blocks}, whose union is $A$. The collection of all set partitions of $A$ is denoted by $\P(A)$, and we write $\P([n])$ as $\P_n$. It is well-known that the cardinality is $|\P_n|=\b_n$, the $n$th \emph{Bell number} \cite[\href{https://oeis.org/A000110}{A000110}]{OEIS}.

A set partition $I=\{I_1,\ldots,I_k\}$ of $[n]$ is often represented in its \emph{canonical form} as $I=(I_{a_1},\ldots,I_{a_k})$, ordered such that $\min(I_{a_1})<\cdots<\min(I_{a_k})$. Such partitions are commonly depicted by \emph{arc diagrams}, in which elements within the same block are transitively connected by arcs. See Figure~\ref{000}. For $i,j\in[n]$, we write $i\sim_Ij$ to indicate that $i$ and $j$ belong to the same block of $I$. Following \cite[Section 2]{Ea11}, we denote by $[i]_I$ the unique block of $I$ containing $i$. Thus, $i\sim_Ij$ if and only if $[i]_I=[j]_I$.

\begin{figure}[H]
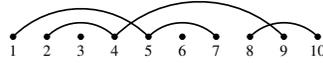
\figone\caption{Arc diagram of a set partition of $\P_7$.}\label{000}\end{figure}

For $I,J\in\P(A)$, we say that $I$ is \emph{finer} than $J$, or equivalently that $J$ is \emph{coarser} than $I$, denoted by $I\prec J$, if and only if every block of $J$ is a union of blocks of $I$. This relation defines a partial order that endows $\P(A)$ with the structure of a lattice, whose \emph{join} operation is denoted by $\vee$. The pair $(\P(A),\vee)$ forms a semigroup with identity element $(\{a\}\mid a\in A)$, referred to as the \emph{monoid of set partitions} of $A$.

As established in \cite[Theorem 2]{Fi03}, the monoid $(\P_n,\vee)$ admits a presentation by generators $e_{i,j}$ for $i,j\in[n]$ with $i<j$, subject to the relations:\begin{equation}\label{001}e_{i,j}^2=e_{i,j};\qquad e_{i,j}e_{r,s}=e_{r,s}e_{i,j};\qquad e_{i,j}e_{i,k}=e_{i,j}e_{j,k}=e_{i,k}e_{j,k}.
\end{equation}Here, $e_{i,j}$ denotes the set partition of $[n]$ whose only non-singleton block is $\{i,j\}$. See Figure~\ref{138}.

\begin{figure}[H]
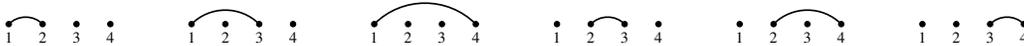
\figtwesix\caption{Generators of $\P_4$.}\label{138}\end{figure}

Every set partition of a subset $X\subset A$ can be viewed as a set partition of $A$ by adjoining singleton blocks for each element of $A\setminus X$. This construction defines a monomorphism from $\P(X)$ into $\P(A)$. Accordingly, given a set $B$ and a pair $(I,J)\in\P(A)\times\P(B)$, we write $I\vee J$ to denote the join of $I$ and $J$, regarded as set partitions of $A\cup B$. Conversely, for any $I\in\P(A)$, we denote by $I|_X$ the \emph{restriction} of $I$ to $X$, defined as the set partition of $X$ whose blocks are the nonempty intersections of the blocks of $I$ with $X$.

\subsection{Partition monoids}\label{018}

Let $\CC_n:=\P_{2n}$, and let $X=\{x_1,\ldots,x_n\}$ be a set of cardinality $n$ disjoint from $[2n]$. Elements of $\CC_n$ are commonly represented by \emph{strand diagrams}, which are constructed from arc diagrams by positioning the first $n$ vertices on an upper row and the remaining $n$ vertices on a lower row, which are subsequently renumbered $[1,n]$ for both levels. See Figure~\ref{003}. An arc intersecting both $[n]$ (upper row) and $[n+1,2n]$ (lower row) is called a \emph{line}. A line of the form $\{i,n+i\}$ is said to be \emph{vertical}. Points contained in $[n]$ are called \emph{up points}, and those contained in $[n+1,2n]$ are called \emph{down points}.

For a subset $B\subseteq[2n]\cup X$, define $B_*=\{b_*\mid b\in B\}$, where $b_*=b$ if $b\in X$, and for $b\in[2n]$, $b_*$ is defined piecewise: $b_*=b+n$ if $b\leq n$ and $b_*=b-n$ if $b>n$. By definition, for each set partition $I\in\CC_n$, the collection $I_*:=\{B_*\mid B\in I\}$ is a set partition of $[2n]$ satisfying $I_{**}=I$. The strand diagram of $I_*$ is obtained by transposing the strand diagram of $I$. Accordingly, $I_*$ is called the \emph{transpose} of $I$. See Figure~\ref{061}.

As introduced in \cite[Section 2]{Ea11}, for each $I\in\CC_n$, the \emph{domain} of $I$, denoted by $\dom(I)$, is the subset of elements $i\in[n]$ such that the block $[i]_I$ intersects $[n+1,2n]$. Similarly, the \emph{codomain} of $I$, denoted by $\codom(I)$, is the subset of elements $i\in[n]$ such that the block $[n+i]_I$ intersects $[n]$. See Figure~\ref{062}.

\begin{figure}[H]
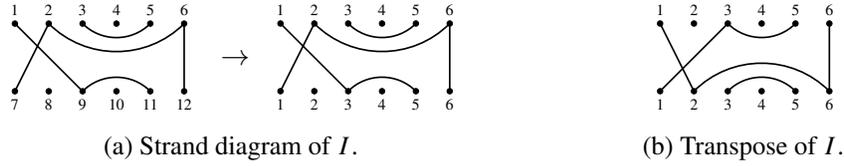
\centering
\begin{subfigure}[b]{0.5\textwidth}\figtwo\caption{Strand diagram of $I$.}\label{003}\end{subfigure}
\begin{subfigure}[b]{0.3\textwidth}\figsxn\caption{Transpose of $I$.}\label{061}\end{subfigure}
\caption{A set partition $I\in\CC_6$ with $\dom(I)=\{1,2,6\}$ and $\codom(I)=\{1,3,5,6\}$.}\label{062}
\end{figure}

For $I,J\in\CC_n$, the \emph{concatenation} of $I$ with $J$ is the set partition $I\ast J:=(I_X\vee J^X)|_{[2n]}$, where $I_X$ is obtained from $I$ by replacing each $n+i$ with $x_i$, and $J^X$ is obtained from $J$ by replacing each $i$ with $x_i$. See Figure~\ref{004}. The pair $(\CC_n,\ast)$ forms a monoid with identity $(\{i,n+i\}\mid i\in[n])$, called the $n$th \emph{partition monoid} \cite{Mr94}.

\begin{figure}[H]
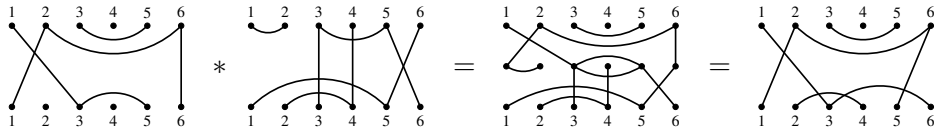
\figthr\caption{Concatenation of two set partitions in $\CC_6$.}\label{004}\end{figure}

There are several important submonoids of $\CC_n$, which play a significant role in semigroup theory and representation theory. One such monoid is the $n$th \emph{partial Brauer monoid} $\PBr_n$, formed by the set partitions of $[2n]$ whose blocks are either arcs or points \cite[Section 2]{Mz98} \cite[Subsection 2.1]{HaDe14}. This monoid contains most of the highly important submonoids of $\CC_n$, including the \emph{Brauer monoid} \cite{Br37} and the \emph{Jones monoid} \cite{Jn83}. Crucially, it also includes the symmetric group and the symmetric inverse monoid, which we introduce in the sequel.

\subsubsection{Permutations}

The collection of set partitions of $[2n]$ whose blocks are all lines forms a submonoid $(\S_n,\ast)$ of $\CC_n$, which coincides with the $n$th \emph{symmetric group} of permutations of $[n]$. Thus, $|\S_n|=n!$, the $n$th factorial number \cite[\href{https://oeis.org/A000142}{A000142}]{OEIS}.
It is also known that $\S_n$ is precisely the group of units of $\CC_n$. Indeed, we have $s^{-1}=s_*$ for all $s\in\S_n$. As shown in \cite{Moo896}, the group $\S_n$ admits a presentation by generators $s_1,\ldots,s_{n-1}$ subject to the following relations:\begin{gather}s_i^2=1;\qquad s_is_js_i=s_js_is_j,\quad|i-j|=1;\qquad s_is_j=s_js_i,\quad|i-j|>1.\label{005}\end{gather}Here, $s_i$ corresponds to the simple transposition exchanging $i$ with $i+1$. In our context, $s_i$ is the set partition where $\{i,n+i+1\}$ and $\{i+1,n+i\}$ are the unique blocks that are not vertical lines. See Figure~\ref{013}.

\begin{figure}[H]
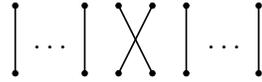
\figfou\caption{Generator $s_i$.}\label{013}\end{figure}

We equip $\S_n$ with the well-known \emph{length function} $\ell:\S_n\to\Nbb_0$, where $\ell(\sigma)$ is the minimal number of generators $s_i$, possibly repeated, needed to express $\sigma$ as a product. An expression of $\sigma$ using exactly $\ell(\sigma)$ generators is called a \emph{reduced expression} of $\sigma$ \cite[Subsection 1.6]{Hm97}.

For each $\sigma\in\S_n$ and $i\in[n]$, we write $i\sigma$ to denote the image of $i$ under $\sigma$. We also represent permutations of $[n]$ as words in $[n]^*$, namely $\sigma=\sigma_1\cdots\sigma_n$, where $\sigma_i=i\sigma$ for all $i\in[n]$. See Figure~\ref{063}.

\begin{figure}[H]
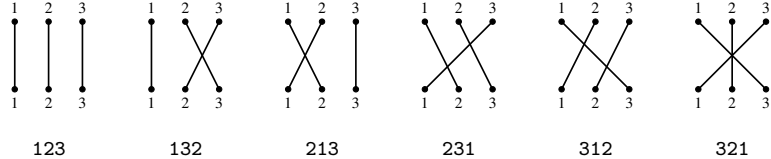
\figsvn\caption{The six permutations in $\S_3$.}\label{063}\end{figure}

Recall that an \emph{inversion} of a permutation $\sigma$ is a pair $(i,j)\in[n]^2$ with $i<j$ such that $\sigma_i>\sigma_j$. The set of inversions of $\sigma$ is denoted by $\inv(\sigma)$. It is well known that $|\inv(\sigma)|=\ell(\sigma)$ \cite[Subsection 1.3]{St97}. Diagrammatically, $|\inv(\sigma)|$ counts the number of times the lines connecting the top row to the bottom row cross each other.

\subsubsection{Partial permutations}

The collection of set partitions of $[2n]$ whose blocks are either lines or points forms a submonoid $(\I_n,\ast)$ of $\CC_n$, called the $n$th \emph{symmetric inverse monoid} \cite{Mnn57}. The elements $\sigma$ of $\I_n$ can be regarded as \emph{partial permutations} of $[n]$, that is, as bijective maps $\sigma:\dom(\sigma)\to\codom(\sigma)$.

As shown in \cite[Remark 4.13]{KuMa06}, the monoid $\mathcal{I}_n$ admits a presentation by generators $s_1,\ldots,s_{n-1}$ satisfying the relations in \eqref{005}, and generators $r_1,\ldots,r_n$ satisfying the relations:
\begin{gather}
r_i^2=r_i;\qquad r_ir_j=r_jr_i;\label{008}\\
s_ir_i=r_{i+1}s_i;\qquad r_is_ir_i=r_ir_{i+1};\qquad s_ir_j=r_js_i,\quad j\not\in\{i,i+1\}.\label{009}
\end{gather}
The generator $r_i$ corresponds to the set partition in which the points $\{i\}$ and $\{n+i\}$ are the unique blocks that are not vertical lines. See Figure~\ref{015}.

\begin{figure}[H]
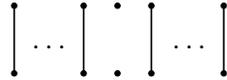
\figsix\caption{Generator $r_i$.}\label{015}\end{figure}

Observe that $(s_i)_*=s_i$ and $(r_i)_*=r_i$ for all $i\in[n]$. Hence, the transpose of any product of these generators is obtained by reversing the order of the factors.

We denote by $\I_n^r$ the set of partial permutations of $[n]$ with exactly $r$ lines. Note that $\I_n^n=\S_n$. The cardinality of $\I_n$ is given in \cite[\href{https://oeis.org/A002720}{A002720}]{OEIS}. See \cite[Corollary 2.3]{Um10}. More precisely, we have:\begin{equation}\label{019}|\I_n|=\sum_{r=0}^n|\I_n^r|;\qquad |\I_n^r|=\binom{n}{r}^2r!.\end{equation}
Accordingly to \cite[Corollary 2]{La98}, we have $\I_n=(\K_n\rtimes\S_n)/R$, where $\K_n$ is the free idempotent commutative monoid generated by $r_1,\ldots,r_n$, subject to the relations in \eqref{008}. The action that defines the semidirect product is given by $s_i\cdot r_j=r_{js_i}$, and $R$ is the congruence closure on $\K_n\rtimes\S_n$ generated by the pairs $(r_is_ir_i,r_ir_{i+1})$.

\begin{rem}\label{112}
By \eqref{009}, we have $r_{i+1}=s_ir_is_i$ for all $i\in[n-1]$. As a consequence, the commutation relations in \eqref{008} can be derived from the commutation of $r_1$ with $r_2$. Thus, the second relation in \eqref{008} can be replaced by the specific relation: $r_1s_1r_1s_1=s_1r_1s_1r_1$, or equivalently, by the relations $r_{i+1}r_is_i=s_ir_{i+1}r_i$.
\end{rem}

As with permutations, for each $\sigma\in\I_n$ and each $i\in\dom(\sigma)$, we write $i\sigma$ to denote the image of $i$ under $\sigma$. We also represent partial permutations of $[n]$ as words in $[n]^*_0$, namely $\sigma=\sigma_1\cdots\sigma_n$, where $\sigma_i=i\sigma$ for all $i\in\dom(\sigma)$, and $\sigma_i=0$ otherwise. See Figure~\ref{065}. An \emph{inversion} of a partial permutation $\sigma$ is a pair $(i,j)\in\dom(\sigma)^2$ with $i<j$ such that $\sigma_i>\sigma_j$. The set of inversions of $\sigma$ is denoted by $\inv(\sigma)$. Let $\sigma^{\scriptscriptstyle+}:=\sigma_1^{\scriptscriptstyle+}\cdots\sigma_r^{\scriptscriptstyle+}$ be the word obtained by removing the zeros from $\sigma$. We write $\sigma^\star$ for the unique permutation in $\S_r$ satisfying $\sigma^{\scriptscriptstyle+}_{1\sigma^{\star-1}}<\cdots<\sigma^{\scriptscriptstyle+}_{r\sigma^{\star-1}}$. For instance, if $\sigma={\tt05203}\in\I_5^3$, then $\sigma^\star={\tt312}\in\S_3$.

\begin{figure}[H]
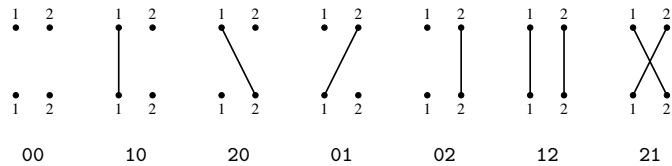
\figntn\caption{The six partial permutations in $\I_2$.}\label{065}\end{figure}

By applying simple \emph{Tietze transformations} \cite[Section 3.2]{Ru95}, we obtain the following result.
\begin{pro}[{\cite[Theorem 6.2]{So02}}]
The monoid $\I_n$ can also be presented by generators $s_1,\ldots,s_{n-1}$ satisfying the relations in \eqref{005}, together with a generator $\nu:=r_ns_{n-1}\cdots s_1r_1$, subject to the relations below.
\begin{gather}\nu^{i+1}s_i=\nu^{i+1}=s_{n-i}\nu^{i+1};\qquad s_i\nu=\nu s_{i+1};\qquad \nu s_1\cdots s_{n-1}\nu=\nu;\label{066}\end{gather}
\end{pro}
The generator $\nu$ corresponds to the set partition whose non-singleton blocks are precisely the lines $\{i,n+i+1\}$ for all $i\in[n-1]$. See Figure~\ref{067} . More generally, observe that $\nu^{n-r}\in\I_n^r$ with $\dom(\nu^{n-r})=[r]$ and $\codom(\nu^{n-r})=[n-r+1,n]$, satisfying $1\nu<\cdots<r\nu$. See Figure~\ref{079}. As a consequence of relations \eqref{008} and \eqref{009}, we have $s_1\cdots s_{n-1}\nu=s_1\cdots s_{n-1}s_{n-1}\cdots s_1r_1^2=r_1$ and $r_i=s_i\cdots s_1r_1s_1\cdots s_i$.

\begin{figure}[H]
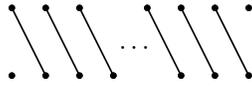
\figftn\caption{Generator $\nu$.}\label{067}\end{figure}
\begin{figure}[H]
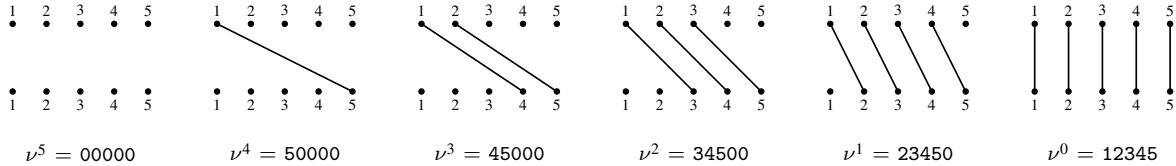
\figtwefou\caption{Powers of the generator $\nu$ in $\I_5$.}\label{079}\end{figure}

\subsection{Coloured set partitions}\label{156}

Let $S=\{q_1<\cdots<q_k\}$ be a finite set of integers ordered canonically. The \emph{arc decomposition} of $S$ is the collection $\arc{S}$ formed by the arcs $\{q_i,q_{i+1}\}$ for all $i\in[k-1]$. We note that $\arc{S}$ is empty if $S$ is a point.

Given a set partition $I:=(I_1,\ldots,I_k)\in\P(A)$ expressed in its canonical form, we define $\arc{I}=\arc{I}_1\sqcup\cdots\sqcup\arc{I}_k$. We refer to the elements of $\arc{I}$ as the \emph{standard arcs} of $I$. The \emph{standard arc diagram} of a set partition $I\in\P_n$ is the arc diagram whose arcs are exactly those in $\arc{I}$. Similarly, the \emph{standard strand diagram} of a set partition in $\CC_n$ is the strand diagram obtained using its standard arc decomposition.

An \emph{$\Fbb_q$-coloured set partition} \cite[Subsection 2.1]{Thi10} \cite{Mar13} of $A$ is a pair $(I,g)$, where $I\in\P(A)$ is a set partition and $g:\arc{I}\to\Fbb_q^\times$ is a map called a \emph{colouring}. The collection of $\Fbb_q$-coloured set partitions of $A$ is denoted by $\P_q(A)$, and we write $\P_{q,n}:=\P_q([n])$. Elements $(I,g)\in\P_{q,n}$ are represented by labelling the standard arcs of $I$ with the corresponding nontrivial elements of $\Fbb_q^\times$ assigned by $g$. See Figure~\ref{016}.

\begin{figure}[H]
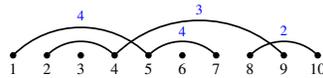
\figsev\caption{Diagram of a coloured set partition $(I,g)$ in $\P_{5,10}$.}\label{016}\end{figure}

We observe that every uncoloured set partition $I$ can be naturally regarded as an $\Fbb_q$-coloured set partition $(I,g)$ in which the standard arcs are trivially coloured by the identity element $1$, that is, $g(b)=1$ for all $b\in\arc{I}$.

\subsection{Framed monoids}\label{157}

A \emph{$q$-framed set partition} is an $\Fbb_q$-coloured set partition whose underlying set partition belongs to $\CC_n$. Such framed set partitions $(I,g)$ are represented by labelling the arcs of the standard strand diagram of $I$ with the nontrivial elements in $\Fbb_q^\times$ assigned by $g$. See Figure~\ref{017} . Given a subset $X$ of $\CC_n$, we denote by $\F_q(X)$ the collection of $q$-framed set partitions with underlying set partition in $X$. The \emph{transpose} of $(I,g)\in\F_q(\CC_n)$ is the framed set partition $(I,g)_*:=(I_*,g_*)$, where $g_*:\arc{I}_*\to\Fbb_q^\times$ is defined by $g_*(B)=g(B_*)$ for each arc $B\in\,\arc{I}\!\!\!_*$.

\begin{figure}[H]
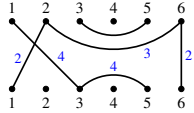
\figeig\caption{Diagram of a framed set partition in $\F_5(\CC_6)$.}\label{017}\end{figure}

The concatenation product for two $q$-framed set partitions is defined naturally within the context of the partial Brauer monoid. Indeed, by the definition of the concatenation product (Subsection~\ref{018}), if $I,J\in\PBr_n$, then each arc $B$ in $I\ast J$ arises as the restriction to $[2n]$ of the union of a unique collection of arcs in $I_X\cup J^X$. Denote by $B(I)$ and $B(J)$ the respective subsets of arcs in $I$ and $J$ that give rise to $B$ via the join $I_X\vee J^X$. Then, for any $q$-framed set partitions $(I,g)$ and $(J,h)$, there is a unique colouring $g\ast h:\,\,\arc{I\ast J}\to\Fbb_q^\times$ defined by:\[(g\ast h)(B)=\prod_{a\in B(I)}g(a)\prod_{a\in B(J)}h(a).\]Thus, the pair $(I\ast J,g\ast h)$ is again a $q$-framed set partition, called the \emph{concatenation} of $(I,g)$ with $(J,h)$, and denoted by $(I,g)\ast(J,h)$. This operation defines a product on $\F_q(\PBr_n)$, which is closed on $\F_q(S)$ for every subsemigroup $S$ of $\PBr_n$. If $M\subseteq\PBr_n$ is a submonoid, then $(\F_q(M),\ast)$ is a monoid with identity given by $1\in\CC_n$, which we call the \emph{$q$-framed monoid} of $M$. Note that the map $I\mapsto(I,1)$, where $1$ denotes the constant colouring, defines a monomorphism from $M$ into $\F_q(M)$. Accordingly, we identify $M$ with its image inside $\F_q(M)$ in what follows.

\begin{pro}\label{020}
$\F_q(\{1\})=(\Fbb_q^\times)^n=\langle a_1,\ldots,a_n\mid a_i^{\,q-1}=1;\quad a_ia_j=a_ja_i\rangle$.
\end{pro}
\begin{proof}
Let $(1,g)\in\F_q(\{1\})$, and for each $i\in[n]$, write $a^{m_i}:=g(\{i,n+i\})$. Define $g_i:\arc{1}\to\Fbb_q^\times$ by setting $g_i(\{i,n+i\})=a$ and $g_i(\{j,n+j\})=1$ for all $j\neq i$. Then, we have $(1,g)=(1,g_1)^{m_1}\cdots(1,g_n)^{m_n}$, so $\F_q(\{1\})$ is generated by the elements $(1,g_1),\ldots,(1,g_n)$. See Figure~\ref{006}. Observe that $(1,g_i)^{\,q-1}=1$ and $(1,g_i)\ast(1,g_j)=(1,g_j)\ast(1,g_i)$, for all $i,j\in[n]$. Hence, the assignment $a_i\mapsto(1,g_i)$ defines a group epimorphism from $(\Fbb_q^\times)^n$ onto $\F_q(\{1\})$. Moreover, since the standard diagram of $1$ contains exactly $n$ arcs, and each arc can be independently coloured by any of the $q-1$ elements of $\Fbb_q^\times$, it follows that $|\F_q(\{1\})|=(q-1)^n=|(\Fbb_q^\times)^n|$. Therefore, the epimorphism $a_i\mapsto(1,g_i)$ is a group isomorphism.
\end{proof}

From now on, as in Proposition~\ref{020}, we identity each element $(1,g_i)$ with the generator $a_i\in(\Fbb_q^\times)^n$.

\begin{figure}[H]
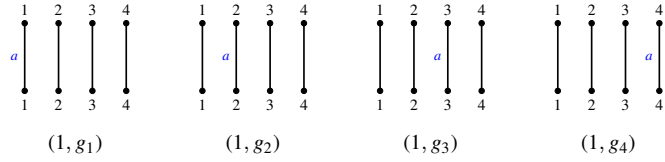
\figtwesev\caption{Generators of $\F_q(\{1\})$ for $n=4$.}\label{006}\end{figure}

\subsection{The framed symmetric inverse monoid}\label{149}

We refer to $\F_q(\I_n)$ as the $n$th \emph{framed symmetric inverse monoid}, and its elements are called \emph{framed partial permutations} of $[n]$. Since each line of a partial permutation in $\I_n^r$ can be independently labelled by one of the $q-1$ nonzero elements of $\Fbb_q^\times$, it follows that $|\F_q(\I_n^r)|=(q-1)^r|\I_n^r|$. In particular, for $r=n$ we obtain $|\F_q(\S_n)|=(q-1)^nn!$. Combining this with identity \eqref{019}, we deduce the following enumeration formula:\begin{equation}|\F_q(\I_n)|=\sum_{r=0}^n|\F_q(\I_n^r)|;\qquad|\F_q(\I_n^r)|=(q-1)^r\binom{n}{r}^2r!.\label{128}\end{equation}Observe that this formula recovers \eqref{019} \cite[\href{https://oeis.org/A000108}{A000108}]{OEIS} whenever $q=2$. Moreover, the resulting sequences for $q=3$ and $q=4$ appear to coincide with \cite[\href{https://oeis.org/A025167}{A025167}]{OEIS} and \cite[\href{https://oeis.org/A102757}{A102757}]{OEIS}, respectively.

\begin{lem}
The monoid $\F_q(\I_n)$ is generated by $s_1,\ldots,s_{n-1}$, $r_1,\ldots,r_n$ and $a_1,\ldots,a_n$.
\end{lem}
\begin{proof}
Let $(I,g)\in\F_q(\I_n)$ be a framed partial permutation. Define $h:\arc{1}\to\Fbb_q^\times$ by setting $h(\{i,n+i\})=g([i]_I)$ if $[i]_I$ is a line, and $h(\{i,n+i\})=1$ if $[i]_I$ is a point. Then $(I,g)=(1,h)\ast(I,1)$, where $(1,h)\in\F_q(\{1\})$ and $(I,1)$ is the element $I\in\I_n$ viewed as a framed partial permutation with trivial colouring. By Proposition~\ref{020}, the element $(1,h)$ is a product of generators $a_1,\ldots,a_n$. Since $I\in\I_n$, it can be written as a product of the generators $s_1,\ldots,s_{n-1}$ and $r_1,\ldots,r_n$. This proves the lemma.
\end{proof}

Observe that the following relations hold in $\F_q(\I_n)$:\begin{gather}
a_i^{\,q-1}=1;\qquad a_ia_j=a_ja_i;\label{022}\\
s_ia_i=a_{i+1}s_i;\qquad s_ia_j=a_js_i,\quad j\not\in\{i,i+1\};\label{023}\\
r_ia_j=a_jr_i;\qquad r_ia_i=r_i.\label{024}
\end{gather}

\begin{pro}\label{025}
The monoid $\F_q(\I_n)$ admits a presentation by generators $s_1,\ldots,s_{n-1}$ satisfying \eqref{005}, generators $r_1,\ldots,r_n$ satisfying \eqref{008} and \eqref{009}, and generators $a_1,\ldots,a_n$, subject to \eqref{022}, \eqref{023} and \eqref{024}.
\end{pro}
\begin{proof}
Let $(1,g),(1,h)\in\F_q(\{1\})$, and $I,J\in\I_n$ be such that $(1,g)\ast(I,1)=(1,h)\ast(J,1)$. Then necessarily $I=J$ and $g\ast1=h\ast1$, that is, $g(\{i,n+i\})=h(\{i,n+i\})$ for all $i\in[n]$ such that $[i]_I$ is a line. As noted in Subsection~\ref{018}, there exist $r\in\I_n$ and $s\in\S_n$ such that $(I,1)=rs$. Moreover, by Proposition~\ref{020} we can write $(1,g)=a_1^{p_1}\cdots a_n^{p_n}$ and $(1,h)=a_1^{q_1}\cdots a_n^{q_n}$ for some $p_i,q_i\in[q-1]_0$, with $g(\{i,n+i\})=h(\{i,n+i\})$ if and only if $p_i=q_i$. Hence $a_1^{p_1}\cdots a_n^{p_n}rs=a_1^{q_1}\cdots a_n^{q_n}rs$. On the other hand, if $p_i\neq q_i$ for some $i\in[n]$, then $I$ contains the blocks $\{i\}$, so the generator $r_i$ occurs in $r$. In this case, the pair $(a_1^{p_1}\cdots a_n^{p_n}rs,a_1^{q_1}\cdots a_n^{q_n}rs)$ lies in the congruence closure generated by $K:=\{(r_ia_i,r_i)\mid i\in[n]\}$.

Now consider the semidirect product $(\Fbb_q^\times)^n\rtimes\I_n$ defined by the action $s_i\cdot a_j=a_{js_i}$ and $r_i\cdot a_j=a_j$. By \cite[Corollary 2]{La98} and Proposition~\ref{018}, this monoid has presentation with generators $s_1,\ldots,s_{n-1}$ satisfying \eqref{005}, generators $r_1,\ldots,r_n$ satisfying \eqref{008} and \eqref{009}, and generators $a_1,\ldots,a_n$ satisfying \eqref{022}, \eqref{023} and the first relation in \eqref{024}. Since these relations hold in $\F_q(\I_n)$, the map $s_i\mapsto s_i$, $r_i\mapsto r_i$ and $a_i\mapsto a_i$ defines an epimorphism $(\Fbb_q^\times)^n\rtimes\I_n\to\F_q(\I_n)$. From the discussion above, together with the fact that the second relation in \eqref{024} holds in $\F_q(\I_n)$, the kernel of this epimorphism is precisely the congruence closure of $K$. Therefore $((\Fbb_q^\times)^n\rtimes\I_n)/\bar{K}\simeq\F_q(\I_n)$, as required.
\end{proof}

Observe that $\F_q(\I_n)$ is one of the monoids studied in \cite[Section 3]{ChEa23}, and also belongs to the class of \emph{abacus monoids} introduced in \cite[Subsubsection 3.5.1]{AiJuPa25}.

\begin{crl}\label{026}
The group $\F_q(\S_n)$ admits presentation by generators $s_1,\ldots,s_{n-1}$ satisfying \eqref{005}, and generators $a_1,\ldots,a_n$ satisfying \eqref{022} and \eqref{023}. Consequently, $\F_q(\S_n)=(\Fbb_q^\times)^n\rtimes\S_n$.
\end{crl}

Moreover, by applying simple \emph{Tietze transformations} \cite[Section 3.2]{Ru95}, we can show that $\F_q(\I_n)$ also admits a presentation with generators $s_1,\ldots,s_{n-1}$ satisfying \eqref{005}, the generator $\nu$ satisfying \eqref{066}, and generators $a_1,\ldots,a_{n-1}$ satisfying \eqref{022} and \eqref{023}, together with the relations:\begin{gather}
a_i\nu=\nu a_{i+1},\quad i<n;\qquad \nu a_1=a_n\nu=\nu.\label{085}
\end{gather}In particular, $a_{n-i+1}\nu^{\,j}=\nu^{\,j}a_i=\nu^{\,j}$ whenever $i\leq j$.

Observe that for any $\sigma\in\F_q(\I_n)$ with $|\dom(\sigma)|=r$, there are unique $\omega_1,\omega_2\in\S_n$, $\sigma^\star\in\S_r$ and $f\in \F_q(\{1\})$ such that $\sigma=\omega_1\,f\sigma^\star v^{n-r}\omega_2$. Furthermore, for $j,k\in[n]$ with $j\leq k$ and $\mbf:=(m_1,\ldots,m_r)\in(\Fbb_q^{\times})^r$, define\[s_{k,\,j}=\begin{cases} 
1&\text{if }k=j\\s_{k-1}\cdots s_j&\text{otherwise}\end{cases}\qquad\text{and}\qquad f_\mbf=a_1^{m_1}\cdots a_r^{m_r}.\]Additionally, for a subset $A=\{i_1<\cdots<i_k\}\subseteq[n]$, we define\[\omega_A=s_{i_1,1}\ldots s_{i_k,k}\qquad\text{and}\qquad\overline{\omega}_A=s_{n-k+1,i_1}\ldots s_{n,i_k}.\]Thus, we have $\omega_A=\overline{\omega}_A=1$ whenever $|A|=n$. See Figure~\ref{165}.\begin{figure}[H]
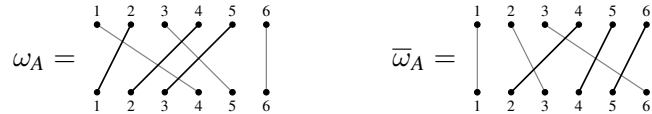
\figtweeig\caption{Elements $\omega_A$ and $\overline{\omega}_A$ for $n=6$ and $A=\{2,4,5\}$.}\label{165}\end{figure}

Let $\sigma\in\F_q(\I_n)$ with $A:=\dom(\sigma)=\{i_1<\cdots<i_r\}$ and $B:=\codom(\sigma)=\{j_1<\cdots<j_r\}$. By applying the defining relations we can write $\sigma=f_\sigma\tilde{\sigma}$, where $\tilde{\sigma}\in\I_n$ and $f_\sigma=a_{i_1}^{m_1}\dots a_{i_r}^{m_r}$ for some $\mbf=(m_1,\ldots,m_r)\in(\Fbb_q^{\times})^r$, which implies $\sigma=\omega_A\,f_\mbf\,\sigma^\star v^{n-r}\,\overline{\omega}_B$. For instance,\begin{figure}[H]
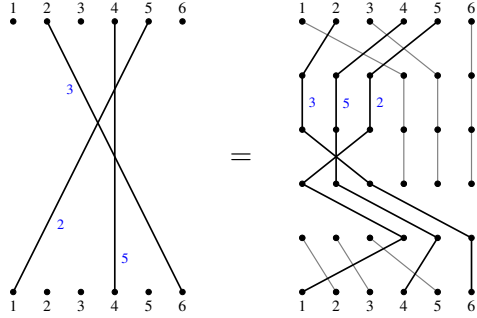
\figtwenin\caption{Element $\sigma=\omega_{\{2,4,5\}}\,f_{(3,5,2)}\,(s_1s_2s_1)\,\nu^3\,\overline{\omega}_{\{1,4,6\}}$.}\end{figure}

\section{Rook monoids}\label{153}

In this section, we focus on the matrix interpretation of the combinatorial structures established in Section~\ref{152}. We begin by recalling the classical rook monoid and introducing the generalized rook monoid $\QQ$ as a submonoid of the full matrix monoid $M$. 
A central result of this section is the establishment of an explicit isomorphism between $\QQ$ and the framed symmetric inverse monoid $\F_q(\I_n)$. Furthermore, we analyze the double coset decomposition of $M$ with respect to the unipotent subgroup $U$, demonstrating that these double cosets are naturally indexed by elements of $\QQ$. Finally, we extend the classical length function to this generalized setting and study the convolution product of double cosets, providing the necessary multiplication formulas for the algebra structure defined in the sequel.

Let $M=\Mbf_n(\Fbb_q)$ be the monoid of $n\times n$ matrices over the field $\Fbb_q$. For each $r\in[n]_0$, we denote by $M^r$ the subset of $M$ consisting of matrices of rank $r$. If $A\subset M$, we set $A^r:=A\cap M^r$.

For $i,j\in[n]$, let $E_{i,j}$ be the $n\times n$ matrix with all entries equal to zero, except for a $1$ in position $(i,j)$. Every matrix $\sigma=(\sigma_{i,\,j})\in M$ can be written uniquely as a linear combination of these elementary matrices:
\begin{equation}\label{029}
\sigma=\sum_{i,\,j\in[n]}\sigma_{i,\,j}E_{i,j}.
\end{equation}
The matrix product is given by the following rule for the elementary matrices:
\begin{equation}
E_{i,j}E_{h,k}=\begin{cases}E_{i,k}&\text{if }j=h\\0&\text{otherwise}.\end{cases}\label{081}
\end{equation}

\subsection{The rook monoid}\label{158}

The $n$th \emph{rook monoid} \cite{So90} is the submonoid $\R\subset M$ consisting of $n\times n$ matrices with entries in $\{0,1\}\subset\Fbb_q$ having at most one nonzero entry in each row and column. It is known that $|\R^r|$ is \cite[\href{https://oeis.org/A144084}{A144084}]{OEIS} and $|\R|$ is \cite[\href{https://oeis.org/A002720}{A002720}]{OEIS}. As mentioned in \cite{So90}, $\R$ is isomorphic to the $n$th symmetric inverse monoid $\I_n$. See Figure~\ref{027}. Each matrix $\sigma=(\sigma_{i,j})\in\R$ corresponds to the partial permutation where the $\{i,n+j\}$ is a line if and only if $\sigma_{i,j}=1$. The sets $I(\sigma)$ and $J(\sigma)$ thus coincide with the domain and codomain of the associated partial permutation. The group of units of $\R$ is the $n$th \emph{Weyl group} $W\simeq\S_n$, of order $n!$ \cite[\href{https://oeis.org/A000142}{A000142}]{OEIS}, consisting of the \emph{permutation matrices} in $\R$, that is, those with no zero rows. By this isomorphism, $W$ inherits the \emph{length function} $\ell:W\to\Nbb_0$ from $\S_n$, which extends to the entire rook monoid \cite[Section 2]{So90} by defining $\ell(\sigma)$ for $\sigma\in\R^r$ as the minimal number of simple transpositions needed to transform $\sigma$ into $v_r:=(E_{1,2}+\cdots+E_{n-1,n})^{n-r}=E_{1,1+n-r}+\cdots+E_{r,n}$. See Figure~\ref{028}. Explicitly, the length is defined by:\[\ell(\sigma)=\min\{\ell(u)+\ell(w)\mid u,w\in W\quad\text{and}\quad u\sigma w=v_r\}.\]This is well-defined because the action of $W\times W$ on $\R^r$ given by $(u,w)\sigma=u\sigma w^{-1}$ is transitive \cite[Equation (2.2)]{So90}. Under the identification above, the generators of $W$ are $s_1,\ldots,s_{n-1}$, where each $s_k$ corresponds to the permutation matrix $1-(E_{k,k}-E_{k+1,k+1})+(E_{k,k+1}+E_{k+1,k})$, and $\nu_{n-1}$ corresponds to the partial permutation $\nu$ in Figure~\ref{067}, so that $\nu_r=\nu^{n-r}$. Moreover, the partial permutation corresponding to the transpose matrix $\sigma_*$ of $\sigma\in\R$ is the transpose of the partial permutation associated with $\sigma$. For instance, if $\sigma$ is the element ${\tt05130}$ in Figure~\ref{027}, then the partial permutation corresponding to $\sigma_*$ is ${\tt05130}_*={\tt30402}$.

\begin{figure}[H]
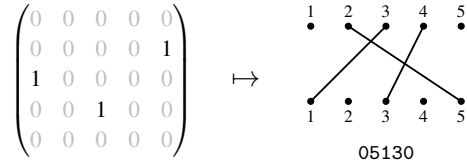
\figten\caption{Correspondence between $\R$ and $\I_n$.}\label{027}\end{figure}

\begin{figure}[H]
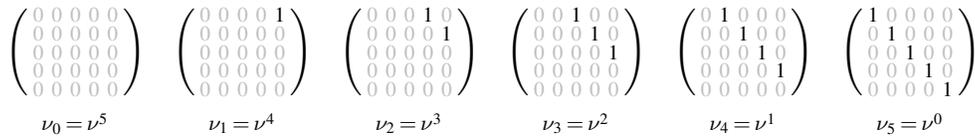
\figele\caption{Elements $v_r$ and their corresponding partial permutations.}\label{028}\end{figure}

As shown in \cite[Theorem 4.12 and (4.20)]{So90}, every matrix in $\R$ can be written as a product of permutations $s_i$ and powers of the matrix $\nu$. More precisely,\begin{equation}\R=\bigsqcup_{r=0}^nW\nu^rW.\label{070}\end{equation}For instance,\[\begin{array}{ccccc}
\begin{psmallmatrix}\zerc&\zerc&\zerc&\zerc&\zerc\\\zerc&\zerc&\zerc&\zerc&1\\1&\zerc&\zerc&\zerc&\zerc\\\zerc&\zerc&1&\zerc&\zerc\\\zerc&\zerc&\zerc&\zerc&\zerc\end{psmallmatrix}&=&
\begin{psmallmatrix}\zerc&\zerc&\zerc&1&\zerc\\\zerc&\zerc&1&\zerc&\zerc\\1&\zerc&\zerc&\zerc&\zerc\\\zerc&1&\zerc&\zerc&\zerc\\\zerc&\zerc&\zerc&\zerc&1\end{psmallmatrix}&
\begin{psmallmatrix}\zerc&\zerc&1&\zerc&\zerc\\\zerc&\zerc&\zerc&1&\zerc\\\zerc&\zerc&\zerc&\zerc&1\\\zerc&\zerc&\zerc&\zerc&\zerc\\\zerc&\zerc&\zerc&\zerc&\zerc\end{psmallmatrix}&
\begin{psmallmatrix}\zerc&1&\zerc&\zerc&\zerc\\\zerc&\zerc&\zerc&1&\zerc\\1&\zerc&\zerc&\zerc&\zerc\\\zerc&\zerc&1&\zerc&\zerc\\\zerc&\zerc&\zerc&\zerc&1\end{psmallmatrix}\\&&_{s_1s_2s_3s_1s_2}&_{\nu^2}&_{s_2s_1s_3}\end{array}\]

\subsection{The generalized rook monoid}\label{148}

The $n$th \emph{generalized rook monoid} \cite[Definici\'on 4.0.1]{Jor22} is the submonoid $\QQ\subset M$ consisting of those $n\times n$ matrices having at most one nonzero entry in each row and column. 

For $\sigma\in\QQ$ as in \eqref{029}, we define the following subsets of indices: $I(\sigma)\subset[n]$ is the set of indices $i$ such that the $i$th row of $\sigma$ is nonzero, and $J(\sigma)\subset[n]$ is the set of indices $j$ such that the $j$th column of $\sigma$ is nonzero. We denote by $\sigma^\star$ the $|I(\sigma)|\times|J(\sigma)|$ matrix obtained from $\sigma$ by removing all zero rows and zero columns. Crucially, we define the $n\times n$ matrix $\tilde{\sigma}=(\tilde{\sigma}_{i,j})$ by setting $\tilde{\sigma}_{i,j}=1$ if $\sigma_{i,j}\neq0$ and $\tilde{\sigma}_{i,j}=0$ otherwise. This matrix essentially records the positions of the nonzero entries of $\sigma$. For instance,
\[\sigma=\scalebox{0.8}{$\begin{pmatrix}\zerc&\zerc&\zerc&\zerc&\zerc\\\zerc&\zerc&\zerc&\zerc&{\blue2}\\{\blue5}&\zerc&\zerc&\zerc&\zerc\\\zerc&\zerc&{\blue3}&\zerc&\zerc\\\zerc&\zerc&\zerc&\zerc&\zerc\end{pmatrix}$};\qquad\sigma^\star=\scalebox{0.8}{$\begin{pmatrix}\zerc&\zerc&{\blue2}\\{\blue5}&\zerc&\zerc\\\zerc&{\blue3}&\zerc\end{pmatrix}$};\qquad\tilde{\sigma}=\scalebox{0.8}{$\begin{pmatrix}\zerc&\zerc&\zerc&\zerc&\zerc\\\zerc&\zerc&\zerc&\zerc&1\\1&\zerc&\zerc&\zerc&\zerc\\\zerc&\zerc&1&\zerc&\zerc\\\zerc&\zerc&\zerc&\zerc&\zerc\end{pmatrix}$}.\]
Observe that $\QQ$ is the preimage of $\R$ under the map $\sigma\mapsto\tilde{\sigma}$. More precisely, each subset $\QQ^r$ is the preimage of $\R^r$. Following \cite{So90}, given $\sigma=(\sigma_{i,j})\in\QQ$ and $i\in I(\sigma)$, we denote by $i\sigma$ the unique element in $J(\sigma)$ such that $\sigma_{i,i\sigma}$ is nonzero. Similarly, for each $j\in J(\sigma)$ we denote by $\sigma j$ the unique element in $I(\sigma)$ such that $\sigma_{\sigma j,j}$ is nonzero. With this notation, $\sigma$ admits the decompositions:\[\sigma=\sum_{i\in I(\sigma)}\sigma_{i,i\sigma}E_{i,i\sigma}=\sum_{j\in J(\sigma)}\sigma_{\sigma j,j}E_{\sigma j,j}.\]Therefore, analogously to \cite[Equation (3.9)]{So90}, we obtain:\begin{equation}\label{035}E_{i,j}\sigma=\begin{cases}
\sigma_{j,\,j\sigma}E_{i,j\sigma}&\text{if }j\in I(\sigma)\\
0&\text{otherwise}
\end{cases}
\qquad
\sigma E_{i,j}=\begin{cases}
\sigma_{\sigma i,i}E_{\sigma i,j}&\text{if }i\in J(\sigma)\\
0&\text{otherwise}.
\end{cases}\end{equation}

As mentioned in \cite[Cap\'itulo 4]{Jor22}, the elements of $\QQ$ can be represented by the diagrams of $\F_q(\I_n)$. Indeed, each $\sigma=(\sigma_{i,j})\in\QQ$ corresponds to the framed partial permutation in which $\{i,n+j\}$ is a line labelled by $\sigma_{i,j}$ if and only if $\sigma_{i,j}$ is nonzero. This correspondence yields an isomorphism between $\QQ$ and the $n$th framed symmetric inverse monoid. See Figure~\ref{030}. In this setting, labels are encoded by diagonal matrices. Namely, for every $\sigma\in\QQ$ there is a, not necessarily unique, diagonal matrix $t=(t_{i,j})$ such that $t_{i,i}=\sigma_{i,i\sigma}$ for all $i\in I(\sigma)$, and $\sigma=t\tilde{\sigma}=\tilde{\sigma}t^\sigma$, where\begin{gather}\label{044}t=\sum_{i\in[n]}t_{i,i}E_{i,i}\quad\text{ and }\quad t^\sigma=\sum_{i\in[n]}t_{i,i}E_{i\hat{\sigma},i\hat{\sigma}}=\sum_{j\in[n]}t_{\hat{\sigma}j,\hat{\sigma}j}E_{j,j}.\end{gather}Here, $\hat{\sigma}$ denotes the unique permutation matrix satisfying $i\hat{\sigma}=i\sigma$ for all $i\in I(\sigma)$, and $i\hat{\sigma}<j\hat{\sigma}$ whenever $i,j\in[n]\setminus I(\sigma)$ with $i<j$. In particular, we write $t$ as $t_I$ when it is the unique diagonal matrix satisfying $\sigma=t\tilde{\sigma}$ and $t_{i,i}=1$ for all $i\in[n]\setminus I(\sigma)$. Similarly, we write $t^\sigma$ as $t_J$ when it is the unique diagonal matrix satisfying $\sigma=\tilde{\sigma}t_J$ and $t^\sigma_{i,i}=1$ for all $i\in[n]\setminus J(\sigma)$. For instance,
\[\begin{array}{ccccccc}
\begin{psmallmatrix}\zerc&\zerc&\zerc&\zerc&\zerc\\\zerc&\zerc&\zerc&\zerc&{\blue2}\\{\blue5}&\zerc&\zerc&\zerc&\zerc\\\zerc&\zerc&{\blue3}&\zerc&\zerc\\\zerc&\zerc&\zerc&\zerc&\zerc\end{psmallmatrix}&=&
\begin{psmallmatrix}1&\zerc&\zerc&\zerc&\zerc\\\zerc&{\blue2}&\zerc&\zerc&\zerc\\\zerc&\zerc&{\blue5}&\zerc&\zerc\\\zerc&\zerc&\zerc&{\blue3}&\zerc\\\zerc&\zerc&\zerc&\zerc&1\end{psmallmatrix}&
\begin{psmallmatrix}\zerc&\zerc&\zerc&\zerc&\zerc\\\zerc&\zerc&\zerc&\zerc&1\\1&\zerc&\zerc&\zerc&\zerc\\\zerc&\zerc&1&\zerc&\zerc\\\zerc&\zerc&\zerc&\zerc&\zerc\end{psmallmatrix}&=&
\begin{psmallmatrix}\zerc&\zerc&\zerc&\zerc&\zerc\\\zerc&\zerc&\zerc&\zerc&1\\1&\zerc&\zerc&\zerc&\zerc\\\zerc&\zerc&1&\zerc&\zerc\\\zerc&\zerc&\zerc&\zerc&\zerc\end{psmallmatrix}&
\begin{psmallmatrix}{\blue5}&\zerc&\zerc&\zerc&\zerc\\\zerc&1&\zerc&\zerc&\zerc\\\zerc&\zerc&{\blue3}&\zerc&\zerc\\\zerc&\zerc&\zerc&1&\zerc\\\zerc&\zerc&\zerc&\zerc&{\blue2}\end{psmallmatrix}\\
_{\sigma}&&_{t_I}&_{\tilde{\sigma}}&&_{\tilde{\sigma}}&_{t_J}
\end{array}\]
\begin{figure}[H]\figtwe\caption{Correspondence between $\QQ$ and $\F_q(\I_n)$.}\label{030}\end{figure}

Thus, for all $r\in[n]_0$, we conclude that:\begin{equation}\label{164}\QQ^r=T\R^r=T\R^r=T\R^rT.\end{equation}

Due to the isomorphism mentioned above, in what follows we will identify each $a_i$ with the corresponding diagonal matrix $1+(a-1)E_{i,i}$. For instance, for $n=5$, we have\[a_4=\begin{psmallmatrix}
1&\zerc&\zerc&\zerc&\zerc\\
\zerc&1&\zerc&\zerc&\zerc\\
\zerc&\zerc&1&\zerc&\zerc\\
\zerc&\zerc&\zerc&{\blue a}&\zerc\\
\zerc&\zerc&\zerc&\zerc&1
\end{psmallmatrix}\]

By applying \eqref{070} and \eqref{044}, we obtain a decomposition:\begin{equation}\QQ=\bigsqcup_{k=0}^nWa_1^{m_1}\cdots a_{n-k}^{m_{n-k}}\nu^kW=\bigsqcup_{k=0}^nW\nu^ka_{k+1}^{m_{k+1}}\cdots a_n^{m_n}W.\label{071}\end{equation}

For instance, $q=7$ and $a=3$, we have:\[\begin{array}{cccccc}
\begin{psmallmatrix}\zerc&\zerc&\zerc&\zerc&\zerc\\\zerc&\zerc&\zerc&\zerc&{\blue2}\\{\blue5}&\zerc&\zerc&\zerc&\zerc\\\zerc&\zerc&{\blue3}&\zerc&\zerc\\\zerc&\zerc&\zerc&\zerc&\zerc\end{psmallmatrix}&=&
\begin{psmallmatrix}\zerc&\zerc&\zerc&1&\zerc\\\zerc&\zerc&1&\zerc&\zerc\\1&\zerc&\zerc&\zerc&\zerc\\\zerc&1&\zerc&\zerc&\zerc\\\zerc&\zerc&\zerc&\zerc&1\end{psmallmatrix}&
\begin{psmallmatrix}{\blue5}&\zerc&\zerc&\zerc&\zerc\\\zerc&{\blue3}&\zerc&\zerc&\zerc\\\zerc&\zerc&{\blue2}&\zerc&\zerc\\\zerc&\zerc&\zerc&1&\zerc\\\zerc&\zerc&\zerc&\zerc&1\end{psmallmatrix}&
\begin{psmallmatrix}\zerc&\zerc&1&\zerc&\zerc\\\zerc&\zerc&\zerc&1&\zerc\\\zerc&\zerc&\zerc&\zerc&1\\\zerc&\zerc&\zerc&\zerc&\zerc\\\zerc&\zerc&\zerc&\zerc&\zerc\end{psmallmatrix}&
\begin{psmallmatrix}\zerc&1&\zerc&\zerc&\zerc\\\zerc&\zerc&\zerc&1&\zerc\\1&\zerc&\zerc&\zerc&\zerc\\\zerc&\zerc&1&\zerc&\zerc\\\zerc&\zerc&\zerc&\zerc&1\end{psmallmatrix}\\&&_{s_1s_2s_3s_1s_2}&_{a_1^5a_2a_3^2}&_{\nu^2}&_{s_2s_1s_3}\end{array}\]
Observe that the framed partial permutation corresponding to the transpose matrix $\sigma_*$ of $\sigma\in\QQ$ is precisely the transpose of the framed partial permutation associated to $\sigma$. For instance, if $\sigma$ is the element ${\tt02530}$--${\tt05130}$ in Figure~\ref{030}, then the framed partial permutation associated with $\sigma_*$ is $(\tt{02530}$--$\tt{05130})_*=\tt{50302}$--$\tt{30402}$.

\subsection{Length function}\label{159}

The \emph{length function} of the rook monoid extends naturally to the generalized rook monoid by setting the length of $\sigma\in\QQ$ to be the length of its support matrix $\tilde{\sigma}$, that is, $\ell(\sigma):=\ell(\tilde{\sigma})$.

Via the isomorphism $\QQ\simeq\F_q(\I_n)$ described above, the notion of inversion for partial permutations carries over naturally to matrices. Specifically, an \emph{inversion} of $\sigma\in\QQ$ is a pair $(i,j)$ with $i<j$ such that $i,j\in I(\sigma)$ and $i\sigma>j\sigma$. The set of inversions of $\sigma$ is denoted by $\inv(\sigma)$, and we write $n(\sigma):=|\inv(\sigma)|$. The number of inversions is preserved under transposition, that is, $n(\sigma_*)=n(\sigma)$.

A \emph{left-inversion} (resp. \emph{right-inversion}) of $\sigma$ is a pair $(i,j)$ with $i<j$ such that $i\not\in I(\sigma)$ and $j\in I(\sigma)$ (resp. $i\in J(\sigma)$ and $j\not\in J(\sigma)$). We denote by $\inv_L(\sigma)$ and $\inv_R(\sigma)$ the sets of left- and right-inversions, respectively, and define\[m_L(\sigma)=|\inv_L(\sigma)|,\qquad m_R(\sigma)=|\inv_R(\sigma)|,\qquad m(\sigma)=m_L(\sigma)+m_R(\sigma).\]Since the length function satisfies $\ell(\sigma)=\ell(\tilde{\sigma})$, the following result is an immediate consequence of \cite[Proposition 2.43 and Lemma 2.25]{So90}.
\begin{pro}\label{031}
For every $\sigma\in\QQ$, we have $\ell(\sigma)=n(\sigma)+m(\sigma)$. Furthermore, if $\sigma\in\QQ^r$, then\[m_L(\sigma)=\sum_{i\in I(\sigma)}(i-1)-\frac{r(r-1)}{2}\quad\text{ and }\quad m_R(\sigma)=\sum_{j\in J(\sigma)}(n-j)-\frac{r(r-1)}{2}.\]
\end{pro}

\begin{exm}
Let $\sigma$ to be the element in Figure~\ref{030}. In this case, $n(\sigma)=2$, $m_L(\sigma)=m_R(\sigma)=3$ and hence $m(\sigma)=6$, where\[\inv(\sigma)=\{(2,3),(2,4)\};\qquad\inv_L(\sigma)=\{(1,2),(1,3),(1,4)\};\qquad \inv_R(\sigma)=\{(1,2),(1,4),(3,4)\}.\]According to Proposition~\ref{031}, we get $\ell(\sigma)=8$. Therefore, eight simple transpositions are required to transform $\tilde{\sigma}$ into $v_r$. See Figure~\ref{032}. Note that $\tilde{\sigma}$ is the element in Figure~\ref{027}. The first three red moves correspond to row operations on the matrix, or equivalently, to upward vertex movements on the strand diagram, when multiplying on the left to eliminate left-inversions. Similarly, the blue moves correspond to column operations on the matrix, or to downward vertex movements on the strand diagram, when multiplying on the right to eliminate right-inversions. The last two red moves correspond to multiplication on the left to remove the remaining inversions.
\begin{figure}[H]
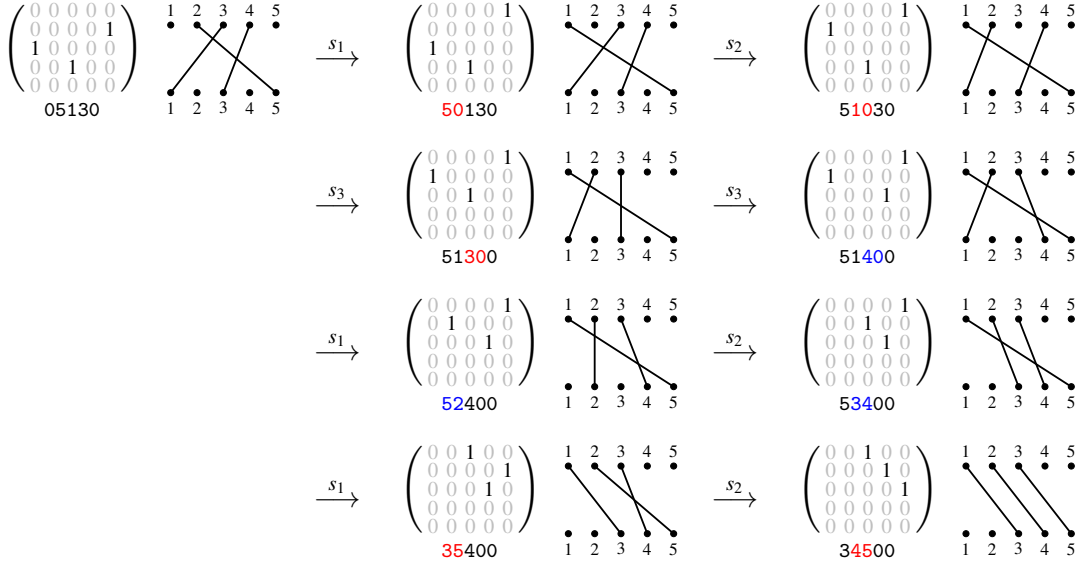
\figttn\caption{Steps to transform $\tilde{\sigma}$ into $v_r=(s_2s_1)(s_3s_2s_1)\tilde{\sigma}(s_3s_1s_2)$.}\label{032}\end{figure}
\end{exm}

\begin{lem}\label{049}
Let $p_1,\ldots,p_{h+k}$ and $q_1,\ldots,q_{h+k}$ be indices in $[n]$, and let $\alpha_1,\ldots,\alpha_{h+k}\in\Fbb_q^\times$. Define the matrix $\tau=\sigma+\varsigma$, where\[\sigma=\alpha_1E_{p_1,q_1}+\cdots+\alpha_hE_{p_h,q_h}\quad\text{and}\quad\varsigma=\alpha_{h+1}E_{p_{h+1},q_{h+1}}+\cdots+\alpha_{h+k}E_{p_{h+k},{q_{h+k}}}.\]If the pairs $(p_i,q_i)$ are all distinct for $i\in\{1,\ldots,h+k\}$, then $\tilde{\tau}=\tilde{\sigma}+\tilde{\varsigma}$, where\[\tilde{\sigma}=E_{p_1,q_1}+\cdots+E_{p_h,q_h}\quad\text{and}\quad\tilde{\varsigma}=E_{p_{h+1},q_{h+1}}+\cdots+E_{p_{h+k},{q_{h+k}}}.\]
\end{lem}
\begin{proof}
As each $E_{p_i,q_i}$ occurs only once in $\sigma$, then $\alpha_i=\sigma_{p_i,q_i}$ for all $i\in[h]$. Hence, $\tilde{\sigma}=E_{p_1,q_1}+\cdots+E_{p_h,q_h}$. A similar argument yields $\tilde{\varsigma}=E_{p_{h+1},q_{h+1}}+\cdots+E_{p_{h+k},{q_{h+k}}}$, and therefore $\tilde{\tau}=E_{p_1,q_1}+\cdots+E_{p_{h+k},{q_{h+k}}}$.
\end{proof}

\begin{lem}\label{050}
Let $\sigma\in\R$ and $\varsigma\in M$.
\begin{enumerate}
\item If $\tau=\sigma\varsigma$, then $\tilde{\tau}=\sigma\tilde{\varsigma}$.
\item If $\tau=\varsigma\sigma$, then $\tilde{\tau}=\tilde{\varsigma}\sigma$.
\end{enumerate}
\end{lem}
\begin{proof}
Let $I(\sigma)=\{p_1,\ldots,p_r\}$ and $J(\sigma)=\{q_1,\ldots,q_r\}$. Then\[\sigma\varsigma=\sum_{k\in[r]}E_{p_k,q_k}\varsigma=\sum_{k\in[r]}\sum_{i,j\in[n]}\varsigma_{i,j}E_{p_k,q_k}E_{i,j}=\sum_{k\in[r]}\sum_{j\in[n]}\varsigma_{q_k,j}E_{p_k,j}.\]Since $\sigma\in\R$, then $p_1,\ldots,p_r$ are all distinct. Thus, by applying Lemma~\ref{049}, we get\[\tilde{\tau}=\sum_{k\in[r]}\sum_{j\in[n]}\tilde{\varsigma}_{q_k,j}E_{p_k,j}=\sigma\tilde{\varsigma}.\]The second identity follows because $q_1,\ldots,q_r$ are also distinct, which allow us to apply Lemma~\ref{049} again.
\end{proof}

\begin{lem}\label{051}
Let $\sigma\in\QQ$ and $i\in[n-1]$.
\begin{enumerate}
\item If $\ell(s_i\sigma)=\ell(\sigma)$, then $s_i\sigma=\sigma$.\label{052}
\item If $\ell(\sigma s_i)=\ell(\sigma)$, then $\sigma s_i=\sigma$.\label{053}
\item $\ell(\nu_{n-1}\sigma)\leq\ell(\sigma)$.\label{054}
\end{enumerate}
\end{lem}
\begin{proof}
\eqref{052} Since $s_i\in\QQ$, Lemma~\ref{050} implies $\ell(s_i\sigma)=\ell(s_i\tilde{\sigma})$. Thus $\ell(s_i\tilde{\sigma})=\ell(\tilde{\sigma})$, and by \cite[Corollary 2.44]{So90} we have $s_i\tilde{\sigma}=\tilde{\sigma}$. Therefore $s_i\sigma=s_i\tilde{\sigma}t_J=\tilde{\sigma}t_J=\sigma$. Claim~\eqref{053} follows by a similar argument.

\eqref{054} Since $\nu_{n-1}\in\R$, Lemma~\ref{050} and \cite[Lemma 2.50]{So90} yield $\ell(\nu_{n-1}\sigma)=\ell(\nu_{n-1}\tilde{\sigma})\leq\ell(\tilde{\sigma})=\ell(\sigma)$.
\end{proof}

\subsection{Double classes}\label{160}

Recall that $B=UT=TU$, where $T$ is the subgroup of diagonal matrices and $U$ is the subgroup of upper unitriangular matrices.

\begin{pro}[{\cite[Teorema 4.0.1]{Jor22}}]\label{033}
We have\[M=\bigsqcup_{\sigma\in\QQ}U\sigma U.\]Moreover, if $\sigma,\sigma'\in\QQ$ satisfy $U\sigma U=U\sigma' U$, then $\sigma=\sigma'$.
\end{pro}
\begin{proof}
According to \cite[Proposition 3.1]{So90}, for every $m\in M$, there are matrices $u_1,u_2\in U$ such that $\sigma=u_1mu_2$ for some $\sigma\in\QQ$. This is achieved by applying a sequence of elementary operations in which addition of rows is done from below to above and addition of columns is done from left to right.

Furthermore, suppose $\sigma,\sigma'\in\QQ$ such that $\sigma'\in U\sigma U$. Then $\sigma'$ may be obtained from $\sigma$ by a sequence of elementary operations as it is described above. We conclude that the nonzero entries of $\sigma$ and $\sigma'$ are in the same positions. Since $T\cap U=\{1\}$, it follows that $\sigma=\sigma'$, which completes the proof.
\end{proof}

For each $\sigma\in\QQ$ and $a,b\in\{0,1\}$, define\[I_{a,b}(\sigma)=\{(i,j)\in I_a(\sigma)\times I_b(\sigma)\mid i<j\}\quad\text{and}\quad J_{a,b}(\sigma)=\{(i,j)\in J_a(\sigma)\times J_b(\sigma)\mid i<j\},\]where $I_1(\sigma)=I(\sigma)$, $I_0(\sigma)=[n]\setminus I(\sigma)$, $J_1(\sigma)=J(\sigma)$ and $J_0(\sigma)=[n]\setminus J(\sigma)$. Diagrammatically, the set $I_0(\sigma)$ represents the upper points of $\sigma$, and similarly, the set $J_0(\sigma)$ represents the bottom points of $\sigma$. Observe that $I_{0,1}(\sigma)=\inv_L(\sigma)$ and $J_{1,0}(\sigma)=\inv_R(\sigma)$. Also, the set $I_{1,1}(\sigma)$ contains $\inv(\sigma)$, and the set $J_{1,1}(\sigma)$ contains $\inv(\sigma_*)$. For instance, for $\sigma=\tt{02530}$--$\tt{05130}$ as in Figure~\ref{030}, we have:\begin{gather*}I_{0,0}(\sigma)=\{(1,5)\};\quad I_{1,0}(\sigma)=\{(2,5),(3,5),(4,5)\};\quad I_{1,1}(\sigma)=\{(2,3),(2,4),(3,4)\};\\ J_{0,0}(\sigma)=\{(2,4)\};\quad J_{0,1}(\sigma)=\{(2,3),(2,5)\};\qquad J_{1,1}(\sigma)=\{(1,3),(1,5),(3,5)\}.\end{gather*}

\begin{lem}\label{034}
Let $\sigma,\tau\in\QQ^r$ be such that $\tilde{\sigma}=\tilde{\tau}$. Then, the map $u\sigma u'\mapsto u\tau u'$ defines a bijection between the double classes $U\sigma U$ and $U\tau U$. Consequently, we have $|U\sigma U|=|U\tau U|$.
\end{lem}
\begin{proof}
Since $\tilde{\sigma}=\tilde{\tau}$, there are $t,t'\in T$ such that $\sigma=t\tilde{\sigma}$ and $\tau=t'\tilde{\sigma}$. By \eqref{164}, there are $d,d'\in T$ such that $\sigma=\tilde{\sigma}d$ and $\tau=\tilde{\sigma}d'$. Now, take $u,u',v,v'\in U$ such that $(ut)\tilde{\sigma} u'=u\sigma u'=v\sigma v'=(vt)\tilde{\sigma} v'$. Since $ut,vt\in UT=B$, \cite[Lemma 3.15]{So90} implies that $(ut)\tilde{\sigma}=(vt)\tilde{\sigma}$ and $u'=v'$. Thus $u\tilde{\sigma}d=v\tilde{\sigma}d$, which yields $u\tilde{\sigma}=v\tilde{\sigma}$. Multiplying by $d'$ on the right gives $u\tilde{\sigma}d'u'=v\tilde{\sigma}d'u'=v\tilde{\sigma}d'v'$, so $u\tau u'=v\tau v'$. The converse implication follows by the same argument, showing that the map $u\sigma u'\mapsto u\tau u'$ is indeed a bijection.
\end{proof}

\begin{lem}\label{040}
For each $\sigma\in\QQ^r$, we have $|U\sigma U|=q^{\frac{r(r-1)}{2}}q^{\ell(\sigma)}$.
\end{lem}
\begin{proof}
Combining \eqref{164} with Proposition~\ref{033}, Lemma~\ref{034} and the fact that $|T\tilde{\sigma}|=(q-1)^r$, we obtain:\[|B\tilde{\sigma}B|=|UT\tilde{\sigma}TU|=|UT\tilde{\sigma}U|=\sum_{\tau\in T\tilde{\sigma}}|U\tau U|=(q-1)^r|U\sigma U|.\]Moreover, \cite[Lemma 3.18]{So90} shows that $|B\tilde{\sigma}B|=(q-1)^rq^{\frac{r(r-1)}{2}}q^{\ell(\tilde{\sigma})}$. Therefore $|U\sigma U|=q^{\frac{r(r-1)}{2}}q^{\ell(\sigma)}$.
\end{proof}

Given $i,j\in[n]$ with $i\neq j$ and $r\in\Fbb_q^\times$, we define $x_{i,j}(r)=1+rE_{i,j}$. As shown in \cite[Equation (3.8)]{So90}, for every $\sigma\in\QQ$ we have the decomposition $U=U_\sigma\bar{U}_\sigma$ with $U_\sigma\cap\bar{U}_\sigma=\{1\}$, where $U_\sigma$ is the subgroup of $U$ generated by the matrices $x_{i,j}(r)$ with $(i,j)\not\in\inv_R(\sigma)\sqcup\inv(\sigma_*)$, and $\bar{U}_\sigma$ is the subgroup of $U$ generated by the matrices $x_{i,j}(r)$ with $(i,j)\in\inv_R(\sigma)\sqcup\inv(\sigma_*)$. From \eqref{035}, and as in \cite[Equation (3.11)]{So90}, we obtain the following relations:\begin{equation}\label{036}\begin{array}{rcll}
x_{i,j}(r)\sigma&=&\sigma&\quad\text{if }j\not\in I(\sigma);\\[0.1cm]
\sigma x_{i,j}(r)&=&\sigma&\quad\text{if }i\not\in J(\sigma);\\
\end{array}\qquad\begin{array}{rcll}
x_{i,j}(r)\sigma&=&\sigma x_{i\sigma,j\sigma}(\sigma_{j,j\sigma}\sigma_{i,i\sigma}^{-1}r)&\quad\text{if }i,j\in I(\sigma);\\[0.1cm]
\sigma x_{i,j}(r)&=&x_{\sigma i,\sigma j}(\sigma_{\sigma i,i}\sigma_{\sigma j,j}^{-1}r)\sigma&\quad\text{if }i,j\in J(\sigma);\\
\end{array}
\end{equation}

For each $k\in[n-1]$ and $r\in\Fbb_q^\times$, we define the diagonal matrix:\[h_k(r)=1+(-r^{-1}-1)E_{k,k}+(r-1)E_{k+1,k+1}.\]Observe that if $r=a^k$, then\begin{equation}\label{080}h_i(r)=a_i^{\frac{q-1}{2}}a_i^{-k}a_{i+1}^k.\end{equation}

\begin{pro}\label{037}
For each $k\in[n-1]$ and $\sigma\in\QQ$, we have:\[Us_kU\cdot U\sigma U=\begin{cases}
U\sigma U&\text{if }(k,k+1)\in I_{0,0}(\sigma)\\
Us_k\sigma U&\text{if }(k,k+1)\in I_{1,0}(\sigma)\sqcup I_{1,1}(\sigma)\setminus\inv(\sigma)\\
{\displaystyle Us_k\sigma U\,\,\,\cup\,\bigcup_{r\in\Fbb_q^\times}Uh_k(r)\sigma U}&\text{if }(k,k+1)\in I_{0,1}(\sigma)\sqcup\inv(\sigma)
\end{cases}\]
\end{pro}
\begin{proof}
The argument follows the same pattern as in \cite[Proposition 3.12]{So90}. First, note that $s_ki<s_k\,j$ for all $(i,j)\neq(k,k+1)$. From \eqref{036} it follows that $s_kU_{s_k}s_k=U_{s_k}$, hence:\begin{equation}\label{038}Us_kU\cdot U\sigma U=Us_kU\sigma U=Us_k(U_{s_k}\bar{U}_{s_k})\sigma U=U(s_kU_{s_k}s_k)s_k\bar{U}_{s_k}\sigma U=UU_{s_k}s_k\bar{U}_{s_k}\sigma U=Us_k\bar{U}_{s_k}\sigma U.\end{equation}Here, $\bar{U}_{s_k}=\{x_{k,k+1}(r)\mid r\in\Fbb_q^\times\}$. If $(k,k+1)\in I_{0,0}(\sigma)\sqcup I_{1,0}(\sigma)$, we have $k+1\not\in I(\sigma)$. From \eqref{038} and \eqref{036} we deduce $Us_kU\cdot U\sigma U=Us_k\sigma U$. In particular, $Us_kU\cdot U\sigma U=U\sigma U$ when $(k,k+1)\in I_{0,0}(\sigma)$.

Now, if $(k,k+1)\in I_{1,1}\setminus\inv(\sigma)$, then $k,k+1\in I(\sigma)$ and $k\sigma<(k+1)\sigma$. Using \eqref{038} and \eqref{036} we get $Us_kU\cdot U\sigma U=Us_k\sigma\bar{U}_{s_k} U=Us_k\sigma U$.

If $(k,k+1)\in I_{0,1}(\sigma)$, we have $s_kx_{k,k+1}(0)\sigma=s_k\sigma\in Us_k\sigma U$. However, due to \eqref{036}, for $r\in\Fbb_q^\times$ we get $s_kx_{k,k+1}(r)\sigma=x_{k,k+1}(r^{-1})h_k(r)x_{k+1,k}(r^{-1})\sigma=x_{k,k+1}(r^{-1})h_k(r)\sigma\in Uh_k(r)\sigma U$. From \eqref{038} we conclude\[Us_kU\cdot U\sigma U=Us_k\sigma U\,\,\,\cup\,\bigcup_{r\in\Fbb_q^\times}Uh_k(r)\sigma U.\]Finally, if $(k,k+1)\in\inv(\sigma)$, then $(k,k+1)\in I_{1,1}(s_k\sigma)\setminus\inv(s_k\sigma)$. By applying the previous case, we obtain $Us_kU\cdot U(s_k\sigma)U=Us_k^2\sigma U=U\sigma U$. Moreover, \cite[Lemma 3]{Yo67} gives:\[s_kUs_k\subset U\cup\bigcup_{r\in\Fbb_q^\times}Uh_k(r)s_kU.\]Therefore,\[Us_kU\cdot U\sigma U=U(s_kUs_k)s_k\sigma U=Us_k\sigma U\,\cup\bigcup_{r\in\Fbb_q^\times}Uh_k(r)s_kUs_k\sigma U=Us_k\sigma U\,\cup\bigcup_{r\in\Fbb_q^\times}Uh_k(r)\sigma U.\]This completes the proof.
\end{proof}

\begin{crl}\label{039}
For each $k\in[n-1]$ and $\sigma\in\QQ$, we have:\[Us_kU\cdot U\sigma U=\begin{cases}
U\sigma U&\text{if }\ell(s_k\sigma)=\ell(\sigma)\\
Us_k\sigma U&\text{if }\ell(s_k\sigma)=\ell(\sigma)+1\\
{\displaystyle Us_k\sigma U\,\,\,\cup\,\bigcup_{r\in\Fbb_q^\times}Uh_k(r)\sigma U}&\text{if }\ell(s_k\sigma)=\ell(\sigma)-1
\end{cases}\]
\end{crl}

\begin{rem}\label{043}
Since $tU=Ut$ for all $t\in T$, we have $UtU\cdot U\sigma U=Ut\sigma U$ and $U\sigma U\cdot Ut U=U\sigma tU$ for all $t\in T$ and $\sigma\in\QQ$. In particular, for any $t,t',t_I\in T$ and $\sigma\in\QQ$, we get\[\begin{array}{c}
(UtU)^m=Ut^mU;\qquad UtU\cdot Ut'U=Ut'U\cdot UtU;\qquad Ut_IU\cdot U\sigma U=Ut_I\sigma U=U\sigma t_JU=U\sigma U\cdot Ut_JU.
\end{array}\]Moreover, $(UtU)^m=U$ whenever $m$ is a power of $q$.
\end{rem}

\section{The framed rook algebra}\label{154}

This section contains the main results of the paper. We first define the framed rook algebra $H(M,U)$ as the algebra generated by the characteristic functions of the double cosets of $U$ in $M$. Parallel to this, we introduce the Rook Yokonuma--Hecke algebra, denoted by $\RY_{d,n}(u)$, as an abstract structure defined by generators and relations. Our central result establishes that the algebra $H(M,U)$ is isomorphic to the abstract algebra $\RY_{d,n}(u)$ under a specific parameter specialization. To complete the characterization, we construct a faithful representation of $\RY_{d,n}(u)$ on a tensor space, which allows us to determine a standard basis and prove that the dimension of the algebra coincides with the cardinality of the generalized rook monoid.

For each $\sigma\in\QQ^r$, we set $T_\sigma=q^{\frac{r(1-r)}{2}}[\sigma]$, where ${\displaystyle[\sigma]=\sum_{x\in U\sigma U}x}$.

Consider the homomorphism $\pi:\Fbb_q[M]\to\Fbb_q$ defined by $\pi(\sigma)=1$ for all $\sigma\in M$. By Lemma~\ref{040}, we have\begin{gather}\label{041}\pi(T_\sigma)=q^{\frac{r(1-r)}{2}}|U\sigma U|=q^{\ell(\sigma)}.\end{gather}

The \emph{framed rook algebra} is the $\Z$-module\begin{equation}H(M,U)=\bigoplus_{\sigma\in\QQ}\Z T_\sigma.\label{125}\end{equation}

For each $k\in[n-1]$ and $r\in\Fbb_q^\times$, we define $h_k^s(r)=s_kh_k(r)s_k\in T$, so that $h_k^s(r)s_k=s_kh_k(r)$, as in \eqref{044}.

The following theorem is analogous to \cite[Theorem 4.12]{So90} and \cite[Th\'eor\`eme 2]{Yo67}.
\begin{thm}\label{055}
The module $H(M,U)$ is a unital ring with unit $T_1$, generated by $T_{s_k}$ with $k\in[n-1]$, $T_t$ with $t\in T$, and $T_\nu$. Moreover,
\begin{enumerate}
\item For each $t,t',t_I\in T$ and $\sigma\in\QQ$, we have
\[T_t^{\,q-1}=1;\qquad T_tT_{t'}=T_{t'}T_t\,;\qquad T_{t_I}T_\sigma=T_{t_I\sigma}=T_{\sigma t_J}=T_\sigma T_{t_J}.\]
where $t_J$ is defined as in \eqref{044}.\label{046}
\item For each $k\in[n-1]$ and each $\sigma\in\QQ$, we have
\begin{gather*}
T_{s_k}T_\sigma=\begin{cases}
qT_\sigma&\text{if }\ell(s_k\sigma)=\ell(\sigma)\\
T_{s_k\sigma}&\text{if }\ell(s_k\sigma)=\ell(\sigma)+1\\
{\displaystyle qT_{s_k\sigma}+\sum_{r\in\Fbb_q^\times}T_{h_k(r)}T_\sigma}&\text{if }\ell(s_k\sigma)=\ell(\sigma)-1
\end{cases}
\\
T_\sigma T_{s_k}=\begin{cases}
qT_\sigma&\text{if }\ell(\sigma s_k)=\ell(\sigma)\\
T_{\sigma s_k}&\text{if }\ell(\sigma s_k)=\ell(\sigma)+1\\
{\displaystyle qT_{\sigma s_k}+\sum_{r\in\Fbb_q^\times}T_\sigma T_{h_k^s(r)}}&\text{if }\ell(\sigma s_k)=\ell(\sigma)-1
\end{cases}
\end{gather*}\label{047}
\item For each $\sigma\in\QQ$, we have\[T_\nu T_\sigma=q^{\ell(\sigma)-\ell(\nu\sigma)}T_{\nu\sigma};\qquad T_\sigma T_\nu=q^{\ell(\sigma)-\ell(\sigma\nu)}T_{\sigma\nu}.\]\label{048}
\end{enumerate}
\end{thm}
\begin{proof}
We first stablish \eqref{046}--\eqref{048}. Since $T_1=U$, we have $T_\sigma T_1=T_\sigma=T_1T_\sigma$ for all $\sigma\in\QQ$.

\eqref{046} By Remark~\ref{043} and the fact that $\ell(t)=0$ for all $t\in T$, we obtain\[\begin{array}{c}
T_t^{\,q-1}=[t]^q=[t^q]=U=T_1=1;\qquad T_tT_{t'}=[t][t']=[tt']=[t't]=[t'][t]=T_{t'}T_t\,;\\
T_{t_I}T_\sigma=q^{\frac{r(1-r)}{2}}[t_I][\sigma]=q^{\frac{r(1-r)}{2}}[t_I\sigma]=T_{t_I\sigma}=T_{\sigma t_J}=q^{\frac{r(1-r)}{2}}[\sigma t_J]=q^{\frac{r(1-r)}{2}}[\sigma][t_J]=T_\sigma T_{t_J}.
\end{array}\]

\eqref{047} If $U\sigma U\cdot U\tau U=U\sigma\tau U$ for some $(\sigma,\tau)\in\QQ^r\times\QQ^s$, then
\begin{gather}\label{045}T_\sigma T_\tau=q^{\frac{r(1-r)}{2}+\frac{s(1-s)}{2}}[\sigma\tau]=q^{\frac{r(1-r)}{2}+\frac{s(1-s)}{2}-\frac{t(1-t)}{2}}T_{\sigma\tau}\quad\text{where}\quad\sigma\tau\in\QQ^t.\end{gather}Applying $\pi$ to both sides and using\eqref{041} yields $q^{\frac{r(1-r)}{2}+\frac{s(1-s)}{2}-\frac{t(1-t)}{2}}=q^{\ell(\sigma)+\ell(\tau)-\ell(\sigma\tau)}$.

From this and Corollary~\ref{031}, we deduce\begin{gather}\label{042}T_{s_k}T_\sigma=\begin{cases}
qT_\sigma&\text{if }\ell(s_k\sigma)=\ell(\sigma)\\
T_{s_k\sigma}&\text{if }\ell(s_k\sigma)=\ell(\sigma)+1
\end{cases}\end{gather}

If $\ell(s_k\sigma)=\ell(\sigma)-1$, set $\tau=s_k\sigma$, so $\sigma=s_k\tau$ with $\ell(s_k\tau)=\ell(\sigma)=\ell(\tau)+1$. Then, using \eqref{042}, Corollary~\ref{031} and part~\eqref{046}, we get\[T_{s_k}T_\sigma=T_{s_k}^2T_\tau=\left(qT_1+\sum_{r\in\Fbb_q^\times}T_{h_k(r)s_k}\right)T_\tau=qT_\tau+\sum_{r\in\Fbb_q^\times}T_{h_k(r)}T_{s_k}T_\tau=qT_{s_k\sigma}+\sum_{r\in\Fbb_q^\times}T_{h_k(r)}T_\sigma.\]The formula for $T_\sigma T_{s_k}$ follows analogously, using part~\eqref{046}.

\eqref{048} Since $U\nu U\cdot U\sigma U=U\nu\sigma U$, because $U\nu=\nu U$, applying $\pi$ in \eqref{045} gives\[T_\nu T_\sigma=q^{\ell(\nu)+\ell(\sigma)-\ell(\nu\sigma)}T_{\nu\sigma}=q^{\ell(\sigma)-\ell(\nu\sigma)}T_{\nu\sigma},\]since $\ell(\nu)=0$. The second identity is shown similarly.

We now prove that $H(M,U)$ is generated by $T_{s_i}$, $T_t$, and $T_\nu$. First, observe that by applying induction on part \eqref{047}, we obtain $T_w=T_{s_{i_1}}\cdots T_{s_{i_k}}$ for all $w\in W$, where $s_{i_1}\cdots s_{i_k}$ is a reduced expression of $w$.

On the other hand, by applying induction on part \eqref{048}, for each integer $i\geq0$, we obtain\begin{equation}\label{056}T_\nu^iT_\sigma=q^{\ell(\sigma)-\ell(\nu^i\sigma)}T_{\nu^i\sigma}.\end{equation}In particular, taking $\sigma=1$ shows $T_\nu^i=T_{\nu^i}$.

By \eqref{044} and the transitivity of the $W\times W$ action on $\R^r$ \cite[Equation (2.2)]{So90}, each $\sigma\in\QQ^r$ can be written as $\sigma=tu\nu^i w$, where $t\in T$, $u,w\in W$, $i=n-r$ and $\ell(\sigma)=\ell(u)+\ell(w)$. We use induction on $\ell(\sigma)$ to show that $T_\sigma=T_tT_uT_\nu^iT_w$.

Observe that $\tilde{\sigma}=u\nu^i w\in\R$, and that, by applying \eqref{046}, we have $T_\sigma=T_tT_{u\nu^iw}$.

If $\ell(\sigma)=0$ then $\tilde{\sigma}=\nu^i$. As shown above, we have $T_{\tilde{\sigma}}=T_\nu^i$. If $\ell(\sigma)>0$ we assume the assertion is true for elements of smaller length. Since $\ell(\sigma)=\ell(u)+\ell(w)$, then $\ell(u)>0$ or $\ell(w)>0$. Suppose that $\ell(u)>0$, and consider $k\in[n-1]$ such that $\ell(s_ku)<\ell(u)$. Then $s_k\tilde{\sigma}=(s_ku)\nu^iw$ with $\ell(s_k\tilde{\sigma})=\ell(\tilde{\sigma})-1$. By applying part \eqref{047} and the induction hypothesis, we get $T_{\tilde{\sigma}}=T_{s_k}T_{s_k\tilde{\sigma}}=T_{s_k}T_{s_ku}T_\nu^iT_w=T_uT_\nu^iT_w$. The case $\ell(w)>0$ is analogous. Therefore $T_\sigma=T_tT_{\tilde{\sigma}}=T_tT_uT_\nu^iT_w$.

Finally, by \eqref{046} and \eqref{047}, we get $T_t\cdot H(M,U)\subseteq H(M,U)$ and $T_{s_k}\cdot H(M,U)\subseteq H(M,U)$. Lemma~\ref{051}\eqref{054} ensures that the exponent $q^{\ell(\sigma)-\ell(\nu\sigma)}$ in \eqref{048} is an integer. Thus, $T_\nu\cdot H(M,U)\subseteq H(M,U)$ as well. Similarly, we have $H(M,U)\cdot T_t\subseteq H(M,U)$, $H(M,U)\cdot T_{s_k}\subseteq H(M,U)$ and $H(M,U)\cdot T_\nu\subseteq H(M,U)$. This shows that $H(M,U)$ is a unital ring with unit $T_1=U$, generated as stated.
\end{proof}

\begin{rem}\label{086}
Notice that, in the cases where $\ell(\sigma s_k)\geq\ell(\sigma)$ and $\ell(s_k\sigma)\geq\ell(\sigma)$, the multiplication rules in Theorem~\ref{055}\eqref{047} agree exactly with those in \cite[Theorem 4.12]{So90}.
\end{rem}

\begin{crl}\label{126}
$H(M,U)$ is a $\Z$--algebra whose dimension is $|\QQ|$.
\end{crl}
\begin{proof}
It is a direct consequence of Theorem~\ref{055} and the definition of the framed rook algebra in \eqref{125}.
\end{proof}

Observe that $T_0$ is not a zero element of the ring $H(M,U)$. Indeed, by applying \eqref{056} and using the fact that $\nu^n=0$, we obtain\[T_0T_\sigma=T_\nu^nT_\sigma=q^{\ell(\sigma)-\ell(\nu^n\sigma)}T_{\nu^n\sigma}=q^{\ell(\sigma)}T_0.\]

In what follows, we write $T_i$, $F_i$, and $N$ to denote $T_{s_i}$, $T_{a_i}$, and $T_\nu$, respectively. We also set\[E_i=\frac{1}{q-1}\sum_{k=0}^{q-2}F_i^kF_{i+1}^{-k}.\]

\begin{pro}\label{078}
The algebra $H(M,U)$ is generated by $T_1,\ldots,T_{n-1}$, $F_1,\ldots,F_n$ and $N$. Moreover, the following relations hold.\begin{gather}
T_iT_jT_i=T_jT_iT_j,\quad |i-j|=1;\qquad T_iT_j=T_jT_i,\quad |i-j|>1;\label{072}\\
F_i^{q-1}=1;\qquad F_iF_j=F_jF_i;\label{073}\\[-0.19cm]
T_i^2=q+(q-1)F_i^{\frac{q-1}{2}}E_iT_i;\qquad T_iF_j=F_{js_i}T_i;\label{074}\\
F_iN=NF_{i+1},\quad i<n;\qquad NF_1=F_nN=N;\label{075}\\
N^{i+1}T_i=qN^{i+1};\qquad T_iN^{n-i+1}=qN^{n-i+1};\label{076}\\
T_iN=NT_{i+1};\qquad NT_1\cdots T_{n-1}N=q^{n-1}N.\label{077}
\end{gather}
\end{pro}
\begin{proof}
The generating set follows from Theorem~\ref{055} and \eqref{071}. The relations are derived from Theorem~\ref{055} and Proposition~\ref{025}. By Remark~\ref{086}, relations \eqref{072}, \eqref{076} and \eqref{077} are proved exactly as in \cite[Lemma 2.6]{So04}. Equation \eqref{073} follows directly from Theorem~\ref{055}\eqref{046}.

For the first identity in \eqref{074}, write $r=a^k$, then, using \eqref{080} we have\[T_i^2=q+\sum_{r\in\Fbb_q^\times}T_{h_k(r)}T_{s_i}=q+(q-1)F_i^{\frac{q-1}{2}}\left(\frac{1}{q-1}\sum_{r\in\Fbb_q^\times}F_i^{\frac{q-1}{2}}T_{h_i(r)}\right)T_i=q+(q-1)F_i^{\frac{q-1}{2}}E_iT_i.\]The second identity in \eqref{074} follows directly from Theorem~\ref{055}\eqref{046} and \eqref{044}. To prove \eqref{075}, we use \eqref{085}:\[F_iN=T_{a_i}T_\nu=T_{a_i\nu}=T_{\nu a_{i+1}}=T_\nu T_{a_{i+1}}=NF_{i+1}.\]The second relation in \eqref{075} is shown analogously.
\end{proof}

\subsection{The Rook Yokonuma--Hecke algebra}\label{007}

In what follows, $\Kbb$ denotes a field of characteristic zero and $u$ denotes an indeterminate. %{\red Douglass \cite{DaDou17}}

Given integers $n,d\geq1$, the \emph{Rook Yokonuma--Hecke algebra} is the $\Kbb(u)$-algebra $\RY_{d,n}(u)$ presented by generators $\Tsf_1,\ldots,\Tsf_{n-1}$, $\Fsf_1,\ldots,\Fsf_n$, and $\Nsf$, subject to the relations:
\begin{gather}
\Tsf_i\Tsf_j\Tsf_i=\Tsf_j\Tsf_i\Tsf_j,\quad|i-j|=1;\qquad\Tsf_i\Tsf_j=\Tsf_j\Tsf_i,\quad |i-j|>1;\label{087}\\
\Fsf_i^d=1;\qquad \Fsf_i\Fsf_j=\Fsf_j\Fsf_i;\label{088}\\[-0.22cm]
\Tsf_i^2=u+(u-1)\Fsf_i^{\frac{d^2-d}{2}}\Esf_i\Tsf_i;\qquad\Tsf_i\Fsf_j=\Fsf_{js_i}\Tsf_i;\label{089}\\
\Fsf_i\Nsf=\Nsf\Fsf_{i+1},\quad i<n;\qquad\Nsf\Fsf_1=\Fsf_n\Nsf=\Nsf;\label{090}\\
\Nsf^{i+1}\Tsf_i=u\Nsf^{i+1};\qquad \Tsf_i\Nsf^{n-i+1}=u\Nsf^{n-i+1};\label{091}\\
\Tsf_i\Nsf=\Nsf\Tsf_{i+1};\qquad\Nsf\Tsf_1\cdots\Tsf_{n-1}\Nsf=u^{n-1}\Nsf;\label{092}
\end{gather}
where\begin{equation}\Esf_i=\frac{1}{d}\sum_{k=0}^{d-1}\Fsf_i^k\Fsf_{i+1}^{-k}.\label{131}\end{equation}
By \eqref{080}, we note that $\Fsf_i^{\frac{d^2-d}{2}}=\Fsf_i^{\,\frac{d}{2}}$ when $d$ is even, and $\Fsf_i^{\frac{d^2-d}{2}}=1$ otherwise. Moreover, each $\Tsf_i$ is invertible; more precisely,\begin{gather}\Tsf_i^{-1}=u^{-1}\Tsf_{i}-\left(1-u^{-1}\right)\Fsf_i^{\frac{d^{2}-d}{2}}\Esf_i.\label{130}\end{gather}

Applying \eqref{092}, most relations in \eqref{091} become superfluous and can be reduced to the specialized relations:\begin{equation}
\Nsf^2\Tsf_1=u\Nsf^2;\qquad \Tsf_{n-1}\Nsf^2=u\Nsf^2.\label{107}\end{equation}Indeed, we have\[\Nsf^{i+1}\Tsf_i\stackrel{\eqref{092}}{=}\Nsf^2\Tsf_1\Nsf^{i-1}\stackrel{\eqref{107}}{=}u\Nsf^{i+1}\qquad\text{and}\qquad\Tsf_i\Nsf^{n-i+1}\stackrel{\eqref{092}}{=}\Nsf^{n-i-1}\Tsf_{n-1}\Nsf^2\stackrel{\eqref{091}}{=}u\Nsf^{n-i+1}.\]

\begin{rem}
The quadratic relation in \eqref{089} differs from the original relation in \cite{Yo67}. Our formulation is instead inspired by the generic presentation developed in \cite{DaDou17}. Specifically, one can recover the original Yokonuma-type quadratic structure by defining a new set of generators:\[\Tsf_i':=\Fsf_{i+1}^{\frac{d^2-d}{2}}\Tsf_i.\]More precisely, with these generators we obtain the following relation:\[\Tsf_i'\,^2=u\left(\Fsf_i\Fsf_{i+1}\right)^{\frac{d^2-d}{2}}+(u-1)\Esf_i\Tsf_i'.\]
\end{rem}

Our next goal is to construct a spanning set for the algebra $\RY_{d,n}(u)$. See Theorem~\ref{123}. Before defining our proposed spanning set, we first introduce some notation and preliminary definitions.

Given indexes $j,k\in[n]$ with $j\leq k$, and a sequence $\mbf=(m_1,\ldots,m_r)\in(\Z/d\Z)^r$, we define\[\Tsf_{k,\,j}=\begin{cases}
1&\text{if }k=j\\
\Tsf_{k-1}\cdots\Tsf_j&\text{otherwise},
\end{cases}\qquad\text{and}\qquad\Fsf_\mbf=\Fsf_1^{m_1}\cdots \Fsf_r^{m_r}.\]Additionally, for $\omega\in\S_n$ and $s_{i_1}\cdots s_{i_{l}}$ a reduced expression for $\omega$, we set $\Tsf_\omega=\Tsf_{i_1}\cdots \Tsf_{i_l}$. Matsumoto's theorem \cite{Mt64} ensures that this definition does not depend on the choice of the reduced expression.

For a subset $A=\{a_1<\cdots<a_k\}\subseteq[n]$, we define\[\Tsf_A=\Tsf_{a_{1,1}}\ldots\Tsf_{a_{k,k}}\qquad\text{and}\qquad\bTsf_A=\Tsf_{n-k+1,a_1}\ldots\Tsf_{n,a_k}.\]Observe that $\Tsf_A=\bTsf_A=1$ whenever $|A|=n$.

Set $\Csf_{n,0}=\{\Nsf^n\}$, and for $r\in[n]$ we define\begin{equation}\Csf_{n,r}=\left\{\Tsf_A\Fsf_\mbf\Tsf_\omega\Nsf^{n-r}\bTsf_B\mid|A|=|B|=r,\,\mbf\in(\Z/d\Z)^r,\,\omega\in\S_r\right\}.\label{129}\end{equation}We then define our proposed spanning set as ${\displaystyle\Csf_n=\bigsqcup_{r=0}^n\Csf_{n,r}}$. Observe that\begin{equation}|\Csf_{n,r}|=d^r\binom{n}{r}^2r!,\quad\text{and hence}\quad|\Csf_n|=\sum_{r=0}^nd^r\binom{n}{r}^2r!.\label{127}\end{equation}

The following theorem will be proved in Subsubsection~\ref{113}.
\begin{thm}\label{123}
The set $\Csf_n$ spans the algebra $\RY_{d,n}(u)$ over $\Kbb(u)$. Consequently, $\dim(\RY_{d,n}(u))\leq|\Csf_n|$.
\end{thm}

As a consequence of Theorem~\ref{123}, we derive the following results:
\begin{thm}\label{057}
The algebra $H_\Kbb(M,U):=\Kbb\otimes_\Z H(M,U)$ is isomorphic to the Rook Yokonuma--Hecke algebra $\RY_{d,n}(u)$ under the specialization $d=q-1$ and $u=q$.
\end{thm}
\begin{proof}
By Proposition~\ref{078} and the defining relations of $\RY_{d,n}(u)$, for $d=q-1$ and $u=q$, the assignment $\Tsf_i\mapsto T_i$, $\Fsf_i\mapsto F_i$ and $\Nsf\mapsto N$ defines an algebra epimorphism $\RY_{d,n}(u)\to H_\Kbb(M,U)$. Then, Corollary~\ref{126} yields the inequalities $|\QQ|=\dim(H_\Kbb(M,U))\leq\dim(\RY_{d,n}(u))\leq|\Csf_n|$. However, by \eqref{128} and \eqref{127}, $|\Csf_n|=|\QQ|$ whenever $d=q-1$ and $u=q$. Therefore the algebras are isomorphic.
\end{proof}

From this isomorphism, we immediately obtain a presentation and a linear basis for $H(M,U)$:
\begin{crl}\label{124}
The framed rook algebra $H(M,U)$ is presented by generators $T_1,\ldots,T_{n-1}$, $F_1,\ldots,F_n$ and $N$, subject to the relations in \eqref{072} through \eqref{077}.
\end{crl}

\begin{crl}
The set $C_n:=C_{n,0}\sqcup\cdots\sqcup C_{n,n}$ is a linear basis of $H(M,U)$, where each $C_{n,r}$ is defined analogously to \eqref{129} by replacing $\Tsf_i,\Fsf_i,\Nsf$ with $T_i,F_i$ and $N$, respectively.
\end{crl}

\subsubsection{Proof of Theorem~\ref{123}}\label{113}
To prove that $\Csf_n$ forms a spanning set of $\RY_{d,n}(u)$, it suffices to show that $\Csf_n$ spans $\RY_{d,n}(u)$ as a $\Kbb(u)$-algebra. To this end, we first establish some preliminary results.

By applying \eqref{089}, for $i\in[n]$ and $\omega\in\S_n$, we have the relations\[\Tsf_\omega\Fsf_i=\Fsf_{\omega i}\Tsf_\omega\qquad\text{and}\qquad\Fsf_i\Tsf_\omega=\Tsf_\omega\Fsf_{i\omega}.\]Recall that there are natural left and right actions of $\S_n$ on $(\Z/d\Z)^n$ given, respectively, by\[\omega\mbf=(m_{\omega1},\dots,m_{\omega n})\qquad\text{and}\qquad\mbf\omega=(m_{1\omega},\dots,m_{n\omega}).\]Therefore, for every $\mbf=(m_1, \dots, m_n)\in (\Z/d\Z)^n$ and $\omega\in\S_n$, we obtain\[\Fsf_\mbf\Tsf_\omega=\Tsf_\omega\Fsf_{1\omega}^{m_1}\cdots\Fsf_{n\omega}^{m_n}=\Tsf_\omega\Fsf_{\omega\mbf},\quad\text{and similarly,}\quad\Tsf_\omega\Fsf_\mbf=\Fsf_{\mbf\omega}\Tsf_\omega.\]

For instance, for $d=6$, $\mbf=(3,5,4)$ and $\omega=s_1s_2$, we have\[\Fsf_\mbf\Tsf_\omega=\Fsf_1^3\Fsf_2^5\Fsf_3^4\Tsf_1\Tsf_2=\Tsf_1\Tsf_2\Fsf_1^5\Fsf_2^4
\Fsf_3^3=\Tsf_\omega\Fsf_{\mbf'},\]where $\mbf'=\omega\mbf=(m_2,m_3,m_1)=(5,4,3)$.

Using relation \eqref{087}, we verify that
\begin{equation}\label{111}
\Tsf_{j,k}\Tsf_l=\left\{\begin{array}{ll}
\Tsf_l\Tsf_{j,k}&\text{if }l<k-1 \text{ or}\ l>j+1 \\
\Tsf_{j,k-1}&\text{if }l=k-1\\[-0.21cm]
u\Tsf_{j,k+1}+(u-1)\Fsf_k^{\frac{d^2-d}{2}}\Esf_{k,j}\Tsf_{j,k}&\text{if }l=k\\
\Tsf_{l-1}\Tsf_{j,k}&\text{if }k<l<j.
\end{array}\right.\end{equation}
where ${\displaystyle\Esf_{k,j}=\frac{1}{d}\sum_{s=0}^{d-1}\Fsf_k^s\Fsf_j^{-s}}$.

Applying \eqref{111}, one has for all $j,k\in[n]$ with $j\leq k$ that
\begin{equation}
\Tsf_{k,\,j}\Tsf_{k+1,\,j+1}\Tsf_j=\Tsf_k\Tsf_{k,\,j}\Tsf_{k+1,\,j+1}.\label{120}
\end{equation}

Next, we prove two technical lemmas.

\begin{lem}\label{121}
For $\mbf\in(\Z/d\Z)^n$ and $\omega\in\S_n$ with $n\geq2$, the element $\Fsf_\mbf\Tsf_\omega$ is a linear combination of elements of the form $\Fsf_{\mbf'}\Tsf_x\alpha\Tsf_y$, where $x,y\in\S_{n-1}$, $\mbf'\in(\Z/d\Z)^n$ and $\alpha\in\{1,\Tsf_{n-1}\}$.
\end{lem}
\begin{proof}
We proceed by induction on $n$. For $n=2$ the statement directly holds. Assume the statement is true for $n-1\geq2$. Let $\Fsf_\mbf\Tsf_\omega$ with $\mbf\in(\Z/d\Z)^n$ and $\omega\in\S_n$. It suffices to consider the case $\Fsf_\mbf\Tsf_\omega=\Fsf_\mbf\Tsf_x\Tsf_{n-1}\Tsf_y\Tsf_{n-1}\Tsf_z$, where $x,y,z\in\S_{n-1}$. If $\Tsf_y\in\S_{n-2}$, then\[\begin{array}{rclll}
\Fsf_\mbf\Tsf_x\Tsf_{n-1}\Tsf_y\Tsf_{n-1}\Tsf_z&=&\Fsf_\mbf\Tsf_x\Tsf_y\Tsf_{n-1}^2\Tsf_z&=&\Fsf_\mbf(u\Tsf_x\Tsf_y\Tsf_z+(u-1)\Tsf_x\Tsf_y\Fsf_{n-1}^{\frac{d^2-d}{2}}\Esf_{n-1}\Tsf_{n-1}\Tsf_z)\\
&&&=&u\Fsf_\mbf\Tsf_x\Tsf_y\Tsf_z+(u-1)\Fsf_\mbf\Fsf_{\tau(n-1)}^{\frac{d^2-d}{2}}\Esf_{\tau(n-1),n}\Tsf_{x}\Tsf_y\Tsf_{n-1}\Tsf_z,
\end{array}\]where $\tau=xy$. If $\Tsf_{y}=\Tsf_{y_1}\Tsf_{n-2}\Tsf_{y_2}$ with $y_1,y_2\in\S_{n-2}$, then
\[\begin{array}{rclll}
\Fsf_\mbf\Tsf_x\Tsf_{n-1}\Tsf_y\Tsf_{n-1}\Tsf_z
&=&\Fsf_\mbf\Tsf_x\Tsf_{n-1}\Tsf_{y_1}\Tsf_{n-2}\Tsf_{y_2}\Tsf_{n-1}\Tsf_z\\
&=&\Fsf_\mbf\Tsf_x\Tsf_{y_1}\Tsf_{n-1}\Tsf_{n-2}\Tsf_{n-1}\Tsf_{y_2}\Tsf_z\\
&=&\Fsf_\mbf\Tsf_x\Tsf_{y_1}\Tsf_{n-2}\Tsf_{n-1}\Tsf_{n-2}\Tsf_{y_2}\Tsf_z
&=&\Fsf_\mbf\Tsf_{x'}\Tsf_{n-1}\Tsf_{y'}
\end{array}\]
where $x'=xy_1s_{n-2}$ and $y'=s_{n-2}y_2z$. By the induction hypothesis, we obtain that $\Fsf_\mbf\Tsf_\omega$ is a linear combination of elements as required.
\end{proof}

\begin{lem}\label{122}
For $r\in[n]$ and every $X\in\Csf_{n,r}$, the following hold:
\begin{enumerate}
\item\label{117} $X\Fsf_i\in\Csf_{n,r}$ for all $i\in[n]$.
\item\label{118} $X\Tsf_i\in\langle\Csf_{n,r}\rangle$ for all $i\in[n-1]$.
\item\label{119} $X\Nsf\in\langle\Csf_{n,r}\cup\Csf_{n,r-1}\rangle$; that is, $X\Nsf$ is a linear combination of elements from $\Csf_{n,r}\cup\Csf_{n,r-1}$.
\end{enumerate}
\end{lem}
\begin{proof}
Observe that the result holds trivially for $r=0$. Let $r\in[n]$ and consider\[X=\Tsf_A\Fsf_\mbf\Tsf_\omega\Nsf^{n-r}\bTsf_B\in\Csf_{n,r}.\]

For \eqref{117}, let $B=\{b_1,\ldots,b_r\}$ and  $i\in[n]$. Consider the complement $[n]\setminus B=\{b_1',\ldots,b_{n-r}'\}$. Notice that $\bTsf_B=\Tsf_{\tau_B}$, where $\tau_B\in\S_n$ is given by $j\tau_B=b_j'$ if $j\leq n-r$, and $j\tau_B=b_{j-(n-r)}$ otherwise.

If $i\in B$, that is, $b_k=i$ for some $k\in[r]$. Then $\bTsf_B\Fsf_i=\bTsf_B\Fsf_{b_k}=\Fsf_{\tau_B(b_k)}\bTsf_B=\Fsf_{n-r+k}\bTsf_B$.

By \eqref{090} we have $\Nsf^{n-r}\Fsf_{n-r+k}=\Fsf_k\Nsf^{n-r}$. Therefore, $X\Fsf_i=\Tsf_A\Fsf_{\mbf'}\Tsf_\omega\Nsf^{n-r}\bTsf_B\in\Csf_{n,r}$, where $\mbf'=(m_1,\dots,m_{j-1},m_j+1,m_{j+1},\dots,m_r)$ and $j=\omega k$.

If $i\not\in B$, then $\tau_B(i)\leq n-r$, and so $\bTsf_B\Fsf_i=\Fsf_j\bTsf_B$ for some $j\leq n-r$. By applying (\ref{090}), we get $\Nsf^{n-r}\Fsf_j=\Nsf^{n-r}$ and hence $X\Fsf_i=X$.

For \eqref{118}, let be $i\in[n-1]$. First suppose $i,i+1\not\in B$. Then, there exists $k\in[r]$ such that $b_k<i$ and $b_{k+1}>i+1$. Thus, we can write $\bTsf_B=YZ$, where $Y=\Tsf_{n-r+1,b_1}\cdots\Tsf_{j,b_k}$ and $Z=\Tsf_{j+1,b_{k+1}}\cdots\Tsf_{n,b_r}$, with $j=n-r+k$. Applying \eqref{111} we obtain $\bTsf_B\Tsf_i=YZ\Tsf_i=\Tsf_{i-k}\bTsf_B$. Hence,
\[X\Tsf_i=\Tsf_A\Fsf_{\mbf}\Tsf_\omega\Nsf^{n-r}\Tsf_{i-k}\bTsf_B=u\Tsf_A\Fsf_1^{m_1}\cdots\Fsf_r^{m_r}\Tsf_\omega\Nsf^{n-r}\bTsf_B.\]
Now suppose that $b_k=i+1\in B$ and $i\not\in B$. Then, $\bTsf_B=Y\Tsf_{j,i+1}Z$, where $Z=\Tsf_{j+1,b_{k+1}}\cdots\Tsf_{n,b_r}$ and $Y=\Tsf_{n-r+1,b_1}\cdots\Tsf_{j-1,b_{k-1}}$, with $j=n-r+k$. Thus, $\bTsf_B\Tsf_i=Y\Tsf_{j,i+1}\Tsf_iZ=Y\Tsf_{j,i}Z=\bTsf_{B'}$, where $B'=\{b_1,\ldots,b_{k-1},b_k-1,b_{k+1}\dots,b_r\}$. Therefore $X\Tsf_i=\Tsf_A\Fsf_{\mbf}\Tsf_\omega\Nsf^{n-r}\bTsf_{B'}$.

If $i\in B$ and $i+1\not\in B$, let $b_k=i\in B$. Since $i+1\not\in B$, we have $b_{k+1}>i+1$. Thus,\[\bTsf_B\Tsf_i=Y\Tsf_{j,i}Z\Tsf_i%=Y\Tsf_{j,i}\Tsf_iZ
=Y(u\Tsf_{j,i+1}+(u-1)\Fsf_i^{\frac{d^2-d}{2}}\Esf_{i,j}\Tsf_{j, i})Z=uY\Tsf_{j,i+1}Z+(u-1)Y\Fsf_i^{\frac{d^2-d}{2}}\Esf_{i,j}\Tsf_{j, i}Z.\]
Applying (\ref{090}), for $B'=\{b_1,\ldots,b_{k-1},b_k+1,b_{k+1}\dots,b_r\}$ and some $\mbf_s\in(\Z/d\Z)^r$, we get\[X\Tsf_i=u\Tsf_A\Fsf_\mbf\Tsf_\omega\Nsf^{n-r}\bTsf_{B'}+\frac{u-1}{d}\sum_{s=0}^{d-1}\Tsf_A\Fsf_{\mbf_s}\Tsf_\omega\Nsf^{n-r}\bTsf_B.\]

If $i,i+1\in B$, let $k\in[r-1]$ such that $b_k=i$ and $b_{k+1}=i+1$. Then $\bTsf_B=Y\Tsf_{j,i}\Tsf_{j+1),i+1}Z$, where $Z=\Tsf_{j+2,b_{m+2}}\cdots\Tsf_{n,b_r}$ and $Y=\Tsf_{n-r+1,b_1}\cdots\Tsf_{j-1,b_{m-1}}$, with $j=n-r+k$. Therefore, by \eqref{120}, we get\[\bTsf_BT_i=Y\Tsf_{j,i}\Tsf_{j+1,i+1}\Tsf_iZ=Y\Tsf_j\Tsf_{j,i}\Tsf_{j+1,i+1}\Tsf_iZ=\Tsf_j\bTsf_B.\]Hence,\[X\Tsf_i=\Tsf_A\Fsf_\mbf\Tsf_\omega\Nsf^{n-r}\bTsf_B\Tsf_i
=\Tsf_A\Fsf_\mbf\Tsf_\omega\Nsf^{n-r}\Tsf_j\bTsf_B
=\Tsf_A\Fsf_\mbf\Tsf_\omega\Tsf_k\Nsf^{n-r}\bTsf_B
=\Tsf_A\Fsf_\mbf\Tsf_\tau\Nsf^{n-r}\bTsf_B,\]where $\tau=\omega s_k\in\S_r$, since $k\in[r-1]$.

For \eqref{119}, suppose first that $n\not\in B$. Then,
\[\begin{array}{rcl}
\bTsf_B\Nsf&=&\Tsf_{n-r+1,b_1}\cdots\Tsf_{n-1,b_{r-1}}\Tsf_{n,b_r}\Nsf\\
&=&\Tsf_{n-r+1,b_1}\cdots\Tsf_{n-2}\Tsf_{n-3}\cdots\Tsf_{b_{r-1}}\Tsf_{n-1}\Tsf_{n-2}\cdots\Tsf_{b_r}\Nsf\\
&\stackrel{(\ref{091})}{=}&\Tsf_{n-r+1,b_1}\cdots\Tsf_{n-2}\Tsf_{n-3}\cdots \Tsf_{b_{r-1}}\Tsf_{n-1}\Nsf\Tsf_{n-1}\cdots\Tsf_{b_r+1}\\
&=&\Tsf_{n-r+1,b_1}\cdots\Tsf_{n-2}\Tsf_{n-1}\Tsf_{n-3}\cdots\Tsf_{b_{r-1}} \Nsf\Tsf_{n,b_r+1}\\
&=&\Tsf_{n-r+1,b_1}\cdots\Tsf_{n-2}\Tsf_{n-1}\Nsf\Tsf_{n-1,b_{r-1}+1}\Tsf_{n,b_r+1}.
\end{array}\]
Repeating this argument yields $\bTsf_B\Nsf=\Tsf_{n-r}\cdots\Tsf_{n-1}\Nsf \bTsf_{B'}$, where $B'=\{b_1+1,\dots, b_r+1\}$. Hence,
\[\begin{array}{rclll}
X\Nsf&=&\Tsf_A\Fsf_\mbf\Tsf_\omega\Nsf^{n-r}\bTsf_B\Nsf\\
&=&\Tsf_A\Fsf_\mbf\Tsf_\omega\Nsf^{n-r}\Tsf_{n-r}\cdots\Tsf_{n-1}\Nsf\bTsf_{B'}\\
&=&u^{1-n+r}\Tsf_A\Fsf_\mbf\Tsf_\omega\Nsf^{n-r}\Tsf_1\cdots\Tsf_{n-r-1}\Tsf_{n-r}\cdots\Tsf_{n-1}\Nsf\bTsf_{B'}&=&u^r\Tsf_A\Fsf_\mbf\Tsf_\omega\Nsf^{n-r}\bTsf_{B'}.
\end{array}\]
Suppose finally that $b_r=n$. Then,
\[\begin{array}{rclll}
\bTsf_B\Nsf&=&\Tsf_{n-r+1,b_1}\cdots\Tsf_{n-1,b_{r-1}}\Tsf_{n,n}\Nsf
&=&\Tsf_{n-r+1,b_1}\cdots\Tsf_{n-1,b_{r-1}}\Nsf\\
&&&=&\Nsf\Tsf_{n-r+2,b_1+1}\cdots\Tsf_{n,b_{r-1}+1}=\Nsf\bTsf_{B'}
\end{array}\]
where $B'=\{b_1+1,\ldots,b_{r-1}+1\}$ and $|B'|=r-1$. Hence,\[X\Nsf=\Tsf_A\Fsf_\mbf\Tsf_\omega\Nsf^{n-r}\bTsf_B\Nsf\\
=\Tsf_A\Fsf_\mbf\Tsf_\omega\Nsf^{n-r+1}\bTsf_{B'}.\]Let $\rho:\S_r\to\S_{r+1}$ be the homomorphism defined by $s_k\mapsto s_{k+1}$ for all $k\in[r-1]$. For simplicity, write $\rho(\tau)$ as $\nu\tau$, and for $\mbf=(m_1,\dots,m_n)\in(\Z/d\Z)^r$, write $\Fsf_2^{m_1}\cdots\Fsf_{n+1}^{m_n}$ as $\Fsf_{\nu\mbf}$. Suppose $\Tsf_\omega\in\S_{r-1}$, then
\[\begin{array}{rcl}\Tsf_A\Fsf_\mbf\Tsf_\omega\Nsf^{n-r+1}&=&\Tsf_{a_1,1}\cdots\Tsf_{a_{r-1},r-1}\Tsf_{a_{r},r}\Fsf_\mbf\Nsf^{n-r+1}\Tsf_{\nu\omega}\\
&=&\Tsf_{a_1,1}\cdots\Tsf_{a_{r-1},r-1}\Tsf_{a_{r},r}\Nsf^{n-r+1}\Fsf_{\nu\mbf}\Tsf_{\nu\omega}\\
&=&u^{a_r-r}\Tsf_{a_1,1}\cdots\Tsf_{a_{r-1},r-1}\Nsf^{n-r+1}\Fsf_{\nu\mbf}\Tsf_{\nu\omega}\\
&=&u^{a_r-r}\Tsf_{a_1,1}\cdots\Tsf_{a_{r-1},r-1}\Fsf_{\mbf\setminus\{r\}}\Nsf^{n-r+1}\Tsf_{\nu\omega}\\
&=&u^{a_r-r}\Tsf_{a_1,1}\cdots\Tsf_{a_{r-1},r-1}\Fsf_{\mbf\setminus\{r\}}\Tsf_\omega\Nsf^{n-r+1},\end{array}\]
where $\mbf\setminus\{r\}=(m_1,\ldots,m_{r-1})\in(\Z/d\Z)^{r-1}$.

Now suppose $\Tsf_\omega=\Tsf_x\Tsf_{r-1}\Tsf_y$ with $x,y\in\S_{r-1}$. Then,\[\begin{array}{rcl}\Tsf_A\Fsf_\mbf\Tsf_\omega\Nsf^{n-r+1}&=&\Tsf_{a_1,1}\cdots\Tsf_{a_{r-1},r-1}\Tsf_{a_r,r}\Fsf_\mbf\Tsf_x\Tsf_{r-1}\Tsf_y\Nsf^{n-r+1}\\
&=&\Tsf_{a_1,1}\cdots\Tsf_{a_{r-1},r-1}\Tsf_{a_r,r}\Fsf_\mbf\Tsf_x\Tsf_{r-1}\Nsf^{n-r+1}\Tsf_{\nu y}\\
&=&\Tsf_{a_1,1}\cdots\Tsf_{a_{r-1},r-1}\Tsf_x\Tsf_{a_r,r}\Tsf_{r-1} \Fsf_{\tau\mbf}\Nsf^{n-r+1}\Tsf_{\nu y}\\
&=&\Tsf_{a_1,1}\cdots\Tsf_{a_{r-1},r-1}\Tsf_x\Tsf_{a_r,r-1}\Fsf_{\tau\mbf}\Nsf^{n-r+1}\Tsf_{\nu y},\end{array}\]where $\tau=xs_{r-1}$. Note that if $\Tsf_x\in\S_{r-2}$, we get\[\begin{array}{rcl}
\Tsf_A\Fsf_\mbf\Tsf_\omega\Nsf^{n-r+1}&=&\Tsf_{a_1,1}\cdots\Tsf_{a_{r-1},r-1}\Tsf_x\Tsf_{a_r,r-1}\Fsf_{\tau\mbf}\Nsf^{n-r+1}\Tsf_{\nu y}\\
&=&\Tsf_{a_1,1}\cdots\Tsf_{a_{r-1},r-1}\Nsf^{n-r+1}\Tsf_{\nu x}\Tsf_{a_r+1,r} \Fsf_{\nu\tau\mbf}\Tsf_{\nu y}\\
&=&u^{a_{r-1}-(r-1)}\Tsf_{a_1,1}\cdots\Tsf_{a_{r-2},r-2}\Nsf^{n-r+1}\Tsf_{\nu x}\Tsf_{a_r+1,r}\Fsf_{\nu\tau\mbf}\Tsf_{\nu y}\\
&=&u^{a_{r-1}-(r-1)}\Tsf_{a_1,1}\cdots\Tsf_{a_{r-2},r-2}\Tsf_x\Tsf_{a_r,r-1} \Fsf_{\tau\mbf}\Nsf^{n-r+1}\Tsf_{\nu y}\\
&=&u^{a_{r-1}-(r-1)}\Tsf_{a_1,1}\cdots\Tsf_{a_{r-2},r-2}\Tsf_{a_r,r-1} \Fsf_{s_{r-1}\mbf}\Nsf^{n-r+1}\Tsf_{\nu x}\Tsf_{\nu y}\\
&=&u^{a_{r-1}-(r-1)}\Tsf_{a_1,1}\cdots\Tsf_{a_{r-2},r-2}\Tsf_{a_r,r-1}\Fsf_{m'}\Nsf^{n-r+1}\Tsf_{\nu x}\Tsf_{\nu y}\\
&=&u^{a_{r-1}-(r-1)}\Tsf_{A'}\Fsf_{m'}\Tsf_{\omega'}\Nsf^{n-r+1}
\end{array}\]where $A'=A\setminus\{a_{r-1}\}$, $\mbf=(m_1,\ldots,m_{n-2},m_n)\in (\Z/d\Z)^{r-1}$ and $\omega'=xy\in\S_{r-1}$.

Applying Lemma~\ref{122} and iterating the previous approach, we obtain\[X\Nsf=\sum_{k=0}^s\Tsf_{A_k}\Fsf_{\mbf_k}\Tsf_{\omega_k}\Nsf^{n-r+1}\bTsf_{B'}\]
with $|A_k|=r-1$, $\mbf_k\in(\Z/d\Z)^{r-1}$ and $\omega_k\in\S_{r-1}$ for all $k$.
\end{proof}

\begin{proof}[Proof of Theorem~\ref{123}]
Since $1\in\Csf_n$, by applying Lemma~\ref{122} it follows that any word formed by the generators of the algebra $\RY_{d,n}(u)$ lies in $\langle\Csf_n\rangle$. Consequently, this implies that $\langle\Csf_n\rangle=\RY_{d,n}(u)$.
\end{proof}

\subsection{Representation on tensors}\label{161}

In the sequel, we set $\Sbb=\Kbb(u^{\frac{1}{2}})$. For each $s\in[d]$, let  $V_s$ be the free $\Sbb$-module spanned by $\{v_i^s\mid i\in[n]\}$. Let $V=\bigoplus_{s\in [d]}V_s$ and $U=\Sbb\oplus V$. In particular, if $v_0$ denotes the identity of $\Sbb$, then $\{v_i^s\}\cup\{v_0\}$ is a $\Sbb$-basis of $U$.

Now, we define three operators $\Fbf\in\End_\Sbb(U)$, $\Ebf\in\End_\Sbb(U^{\otimes 2})$ and $\Tbf\in\End_\Sbb(U^{\otimes2})$, as follows:
\begin{equation}\label{098}
\Fbf(v_i^s)=\xi^sv_i^s,\qquad\Fbf(v_0)=v_0,
\end{equation}
\begin{equation}\label{099}
\Ebf(x\otimes y)=\begin{cases}
v_i^t\otimes v_j^s&\text{if }x=v_i^s,\;y=v_j^t\text{ and }s=t\\
v_0\otimes v_0&\text{if }x=v_0\text{ and }y=v_0\\
0&\text{otherwise},
\end{cases}
\end{equation}and
\begin{equation}\label{101}
\Tbf(x\otimes y)=\begin{cases}
u^{\frac{1}{2}}\cdot v_j^t\otimes v_i^s&\text{if }x=v_i^s,\;y=v_j^t\text{ and }s\neq t\\
(-1)^{(d-1)s}u\cdot v_i^s\otimes v_j^t&\text{if }x=v_i^s,\;y=v_j^t,\;s=t,\;i=j\\
(-1)^{(d-1)s}u^{\frac{1}{2}}\cdot v_j^t\otimes v_i^s&\text{if }x=v_i^s,\;y=v_j^t,\;s=t,\;i>j\\
(-1)^{(d-1)s}(u-1)\cdot v_i^s\otimes v_j^t+(-1)^{(d-1)s}u^{\frac{1}{2}}\cdot v_j^t\otimes v_i^s &\text{if }x=v_i^s,\;y=v_j^s,\;s=t,\;i<j\\
u\cdot v_0\otimes v_0&\text{if }x=v_0,\;y=v_0\\
u^{\frac{1}{2}}\cdot y\otimes x&\text{if }x=v_0,\;y\in V\text{ or }x\in V,\;y=v_0
\end{cases}
\end{equation}where $\xi$ is a primitive $d$th root of unity.

We extend them to operators $\Fbf_i$, $\Ebf_i$ and $\Tbf_i$ acting on the tensor space $U^{\otimes n}$ by letting $\Fbf$ act in the $i$th factor, $\Ebf$ in the $i$th and $(i+1)$st factors and $\Tbf$ in the $i$th and $(i+1)$st factors, respectively.

It was shown in \cite[Lemma 6]{EsRyH18} that\[\Ebf_i=\dfrac{1}{d}\sum_{k=0}^{d-1}\Fbf_i^k\Fbf_{i+1}^{-k}.\]

In order to define the operator $\Nbf$ associated with the generator $\Nsf$ of $\RY_{d,n}(u)$, for $x=x_1\otimes\cdots\otimes x_n\in U^{\otimes n}$, we set $\varepsilon(x)=|\{k\mid x_k \in V\}|$, which counts the number of non-scalar components in the tensor.

The operator $\Nbf$ on $U^{\otimes n}$ is then defined by\[\Nbf(x_1\otimes\cdots\otimes x_n)=u^{\frac{\varepsilon(x)}{2}}x_2\otimes\cdots\otimes x_n\otimes x_1^0,\]where $x_1^0=v_0$ whenever $x_1\in\Sbb$, and $x_1^0=0$ whenever $x_1\in V$.

\begin{thm}\label{151}
There is a representation $\rho$ of $\RY_{d,n}(u)$ in $U^{\otimes n}$ given by $\Fsf_i\mapsto\Fbf_i$, $\Tsf_i\mapsto\Tbf_i$ and $\Nsf\mapsto\Nbf$.
\end{thm}
\begin{proof}
It suffices to show that the operators defined above satisfy the defining relations in \eqref{087} to \eqref{092}. We first observe that relations in \eqref{088} follow immediately from the commutativity of the operators $\Fbf_i$. On the other hand, the second relation in \eqref{089} is verified exactly as in the proof of \cite[Theorem 8]{EsRyH18}, upon identifying our operators $\Tbf$ and $\Fbf$ with $\Gbf$ and $\Tbf$, respectively.

Similarly for relations \eqref{090} to \eqref{092}, involving $\Tbf$ and $\Nbf$, the proof is the same as in \cite[Lemma 3.8]{So04}. Indeed, the argument depends only on the structure of $\Nbf$, which is the same as the operator $N$ in \cite[Equation (3.6)]{So04}, and on the property that $\Tbf$ maps $v_i^s\otimes v_j^t$ into the $\Sbb$-submodule spanned by $v_i^s\otimes v_j^t$ and $v_j^t\otimes v_i^s$, that is, $\Tbf(v_i^s\otimes v_j^t)=\alpha\,v_i^s\otimes v_j^t+\beta\,v_j^t\otimes v_i^s$ for suitable scalars $\alpha,\beta$. This is exactly the case for our $\Tbf$, so the proofs of relations \eqref{090} to \eqref{092} go through unchanged.

To verify the quadratic relation in \eqref{087} it suffices to consider a tensor $x\in U\otimes U$. We split into cases:

If $x=v_0\otimes v_0$, it follows by\[\Tbf^2(v_0\otimes v_0)=u^2(v_0\otimes v_0)=(v_0\otimes v_0)u+(u^2-u)(v_0\otimes v_0)=(v_0\otimes v_0)u+(u-1)\Fbf^{\frac{d^2-d}{2}}\Ebf\Tbf(v_0\otimes v_0).\]

If $x=v_i^s\otimes v_j^s$ and $d$ even. Whenever $s$ is even, the quadratic relation reduces to the classical rook case, and our operator $\Tbf$ coincides with the operator $T$ in \cite[Equation (3.2)]{So04}. Hence it holds. Whenever $s$ is odd, the quadratic relation becomes $\Tbf^2(v_i^s\otimes v_j^s)=(v_i^s\otimes v_j^s)u-(u-1)\Tbf(v_i^s\otimes v_j^s)$. Furthermore, in this case the operator $\Tbf$ coincides with $-T$. On the other hand, if $d$ is odd, then $\Tbf$ again coincides with the operator $T$, so the equality holds in this case as well.

If $x=v_i^s\otimes v_j^t$ with $s\neq t$, or $x=v_0\otimes v_i^s$, or $x=v_i^s\otimes v_0$, then $\Tbf^2(x)=ux$, and $\Ebf$ vanishes on them, since $\Tbf$ acts as $u^{\frac{1}{2}}$ times the flip. Consequently $\Tbf^2(x)=ux=ux+(u-1)\Fbf^{\frac{d^2-d}{2}}\Ebf\Tbf(x)$.

We now verify the relations in \eqref{087}.

First, we verify the commutation relation for distant generators; it suffices to work on $U^{\otimes4}$. In this setting $\Tbf_i$ and $\Tbf_j$ act on disjoint tensor factors, hence they commute, and the identity $\Tbf_i\Tbf_j=\Tbf_j\Tbf_i$ follows immediately.

To verify the relation $\Tbf_i\Tbf_{i+1}\Tbf_i=\Tbf_{i+1}\Tbf_i\Tbf_{i+1}$, it suffices to work on $U^{\otimes 3}$. We split into cases:

If the upper indices are all distinct, more generally, whenever a pair has distinct upper indices or involves a scalar, each $\Tbf$ on its adjacent pair acts as $q^{\frac{1}{2}}$ times the flip. Both sides then implement the same permutation of the three tensor factors, multiplied by the same overall scalar, so the relation holds immediately.

If the upper indices are all equal, say $s$, then on $V_s^{\otimes 3}$ we have $\Tbf_k=(-1)^{(d-1)s}T_k$ for $k=i,i+1$, where $T$ is the classical rook operator. Hence both sides are $(-1)^{3(d-1)s}$ times the classical rook relation and the sign cancels on both sides. Similarly, for the case $v_0\otimes v_0\otimes v_0$ the relation follows from the classical rook case.

For the mixed pattern with two equal upper indices and one different, by symmetry it suffices to consider a representative vector $x=u_i\otimes u_j\otimes v$, where each $u_\ell$ denotes either $v_\ell^{\,s}$ (for a fixed $s$) or $v_0$, and $v$ denotes either $v_k^{\,t}$ with $t\neq s$ or $v_0$, as appropriate. In all these cases the operator $\Tbf$ acts on the first two factors by $\Tbf(w_i\otimes w_j)=\alpha w_i\otimes w_j+\beta w_j\otimes w_i$ with $\alpha,\beta\in\Sbb$. A direct use of the defining rules of $\Tbf$ then gives\[\Tbf_1\Tbf_2\Tbf_1(x)=u\alpha v\otimes w_i\otimes w_j+u\beta v\otimes w_j\otimes w_i=
\Tbf_2\Tbf_1\Tbf_2(x).\]Therefore the relation holds in this case as well. By symmetry, the other placements of the distinct upper index are analogous, so the result follows.
\end{proof}

We now use this tensor representation to prove that the spanning set
$\Csf_n$ of $\RY_{d,n}(u)$ in Theorem~\ref{123} is in fact a basis. For this purpose we introduce the notion of support for tensors in $U^{\otimes n}$ and establish a few elementary properties.

We define the \emph{support} of a tensor $x=x_1\otimes\cdots\otimes x_n\in U^{\otimes n}$ by $\operatorname{supp}(x)=\{k\in[n]\mid x_k\in V\}$. If $\supp(x)=A$, we say that $x$ \emph{has support} $A$. For instance, $v_0\otimes v_2^1\otimes v_0\otimes v_3^0\otimes v_0\otimes v_1^2 \otimes v_0$ has support $\{2,4,6\}$.

From now on, for $A,B\subseteq[n]$ and $\omega\in\S_n$, we write $\Tbf_A$, $\overline{\Tbf}_B$ and $\Tbf_\omega$ to denote the respective images of $\Tsf_A$, $\bTsf_B$ and $\Tsf_\omega$ in Subsection~\ref{007}, under the tensor representation in \eqref{101}.

\begin{lem}\label{150}
Let $r\in[n]$, and let $B=\{b_1<\cdots<b_r\}\subseteq[n]$. Then:
\begin{enumerate}
\item If $x\in U^{\otimes n}$ has support $B$, then $\overline{\Tbf}_B(x)$ is nonzero and has support $[n-r+1,n]$.\label{012}
\item If $y\in U^{\otimes n}$ has support of size $r$, then $\Nbf^{n-r}(y)\neq 0$ if and only if $\supp(y)=[n-r+1,n]$. Moreover, in this case $\Nbf^{n-r}(y)$ has support $[r]$.\label{014}
\item If $z\in U^{\otimes n}$ has support $[r]$ and its non scalar
factors have strictly decreasing lower indices, then for every
$\omega\in\S_r$ the tensor $\Tbf_\omega(z)$ is nonzero, and
$\Tbf_\omega$ acts on $z$ exactly as the permutation $\omega$ acts
on tensors: the $r$ non scalar factors of $z$ are permuted according to $\omega$,
up to nonzero scalar multiples.\label{064}
\end{enumerate}
\end{lem}
\begin{proof}
\eqref{012}
Let $x\in U^{\otimes n}$ with support $B:=\{b_1<\cdots<b_r\}$. Recall that
\[\overline{\Tbf}_B=\Tbf_{n-r+1,b_1}\Tbf_{n-r+2,b_2}\cdots\Tbf_{n,b_r},\quad\text{where}\quad\Tbf_{k,j}=\Tbf_{k-1}\cdots\Tbf_j.\]Each operator $\Tbf_{k,j}$ moves the non scalar factor sitting in position $j$ successively across adjacent scalar factors until it reaches
position $k$, never creating nor removing non scalar entries. Thus each position $b_j$ is moved to position $n-r+j$, and no other positions become non scalar. Hence $\overline{\Tbf}_B(x)$ is nonzero and has support $[n-r+1,n]$.

\eqref{014} Let $y\in U^{\otimes n}$ having support $B'$ with $|B'|=r$. If $B'=[n-r+1,n]$, then during the first $n-r$ applications of $\Nbf$ the first tensor factor is always scalar, and $\Nbf$ simply rotates the tensor left while preserving the number of non scalar entries. After $n-r$ steps, the $r$ non scalar factors occupy the positions $[r]$, proving that $\Nbf^{n-r}(y)\neq 0$ and has support $[r]$.

Conversely, if $B'\neq[n-r+1,n]$, then some non scalar factor of $y$ lies among the first $n-r$ positions. Since $\Nbf$ rotates the tensor left and annihilates it as soon as the first factor becomes non scalar, one of the first $n-r$ applications of $\Nbf$ sends $y$ to $0$.
Therefore $\Nbf^{n-r}(y)=0$ in this case.

\eqref{064} Let $z\in U^{\otimes n}$ having support $[r]$ and strictly decreasing
lower indices on these $r$ non scalar tensor factors. For each simple transposition $s_i\in\S_r$, the operator $\Tbf_i$ acts only on positions $i$ and $i+1$, where the two non scalar factors have distinct lower indices. By the explicit formula for $\Tbf$, this action is a nonzero scalar multiple of the flip, and it does not interact with the remaining components. Hence $\Tbf_i(z)$ is nonzero and its non scalar factors are obtained from
those of $z$ by exchanging the $i$th and $(i+1)$st positions.

Since every $\omega\in\S_r$ is a product of simple transpositions, $\Tbf_\omega$ acts on $z$ by composing the corresponding flips, and therefore coincides with the natural permutation action of $\omega$ on tensors, up to nonzero scalar multiples.
\end{proof}

\begin{thm}\label{163}
The set $\Csf_n$ introduced in \eqref{129} is linearly independent over $\Kbb(u)$. Consequently,\[\dim_{\Kbb(u)}\!\bigl(\RY_{d,n}(u)\bigr)=|\Csf_n|=\sum_{r=0}^nd^r\binom{n}{r}^{\!2}r!.\]
In particular, the tensor representation $\rho:\RY_{d,n}(u)\to \End_{\Kbb(u)}\bigl(U^{\otimes n}\bigr)$ is faithful whenever $d\geq n$.
\end{thm}
\begin{proof}
Using the defining relations, each element of $\Csf_n$ can be rewritten by moving the generators $\Fsf_i$ to the left. More precisely, for every $r\in [n]_0$, $A,B\subseteq[n]$ with $|A|=|B|=r$, $\omega\in\S_r$ and
$\mbf'\in(\Z/d\Z)^r$, we write\[\Tsf_A\Fsf_{\mbf'}\Tsf_\omega\Nsf^{n-r}\bTsf_B=\Fsf_\mbf X_{A,\omega,B},\quad\text{where}\quad X_{A,\omega,B}=\Tsf_A\Tsf_\omega\Nsf^{n-r}\bTsf_B,\]for some $\mbf\in(\Z/d\Z)^r$ depending on $(A,B,\omega,\mbf')$. Thus, it suffices to show that any relation\begin{equation}\label{142}
\lambda_0\Nsf^n+\sum_{r=1}^n\mathop{\sum_{\substack{|A|=|B|=r\\\omega\in\S_r,\ \mbf\in(\Z/d\Z)^r}}}\lambda_{A,B,\omega,\mbf}\Fsf_\mbf X_{A,\omega,B}=0\end{equation}with coefficients $\lambda_{A,B,\omega,\mbf}\in\Kbb(u)$ is trivial.

We first deal with the case $r=0$. Evaluating \eqref{142} on the $n$ tensor $x=v_0\otimes\cdots\otimes v_0$, all terms in the double sum vanish, whereas $\Nsf^n(x)=u^{\frac{n}{2}}x$. Hence the equality forces $\lambda_0u^{\frac{n}{2}}x=0$, which shows that $\lambda_0=0$. Thus, no nontrivial relation can occur at level $r=0$.

We now assume, for a contradiction, that the sum in \eqref{142} is nontrivial. Let $r_0\geq 1$ be maximal such that $\lambda_{A,B,\omega,\mbf}\neq0$ for some $A,B\subseteq[n]$ with $|A|=|B|=r_0$, $\omega\in\S_{r_0}$ and $\mbf\in(\Z/d\Z)^{r_0}$. Fix such a subset $B_0=\{b_1<\cdots<b_{r_0}\}$ with $\lambda_{A,B_0,\omega,\mbf}\neq0$ for at least one $(A,\omega,\mbf)$. Passing to the tensor representation $\rho$ gives a nontrivial identity:
\begin{equation}\label{143}
\sum_{r=1}^n\mathop{\sum_{\substack{|A|=|B|=r\\\omega\in\S_r,\ \mbf\in(\Z/d\Z)^r}}}
\lambda_{A,B,\omega,\mbf}\Fbf_\mbf\Tbf_A\Tbf_\omega\Nbf^{n-r}\overline{\Tbf}_B=0.
\end{equation}

Choose a basis tensor $x_{B_0}\in U^{\otimes n}$ with support $B_0$ and strictly decreasing lower indices on those positions. Evaluating \eqref{143} at $x_{B_0}$ yields the following:\begin{equation}\label{144}
\sum_{r=1}^n\mathop{\sum_{\substack{|A|=|B|=r\\\omega\in\S_r,\ \mbf}}}\lambda_{A,B,\omega,\mbf}\Fbf_\mbf\Tbf_A\Tbf_\omega\Nbf^{n-r}\overline{\Tbf}_B(x_{B_0})=0.
\end{equation}

Since $x_{B_0}$ has exactly $r_0$ non scalar components, Lemma~\ref{150}\eqref{014} implies that $\Nbf^{n-r}\overline{\Tbf}_B(x_{B_0})=0$ for all $r<r_0$. All coefficients with $r>r_0$ vanish by maximality of $r_0$. Thus, \eqref{144} reduces to
\begin{equation}\label{145}
\sum_{\substack{|A|=|B|=r_0\\\omega\in\S_{r_0},\ \mbf}}\lambda_{A,B,\omega,\mbf}
\Fbf_\mbf\Tbf_A\Tbf_\omega\Nbf^{n-r_0}\overline{\Tbf}_B(x_{B_0})=0.
\end{equation}

Lemma~\ref{150}\eqref{012}--\eqref{014} gives $\Nbf^{n-r_0}\overline{\Tbf}_B(x_{B_0})=z_0$ if $B=B_0$, and $\Nbf^{n-r_0}\overline{\Tbf}_B(x_{B_0})=0$ otherwise, where $z_0$ is a nonzero basis tensor supported on $[r_0]$ with strictly decreasing lower indices, rescaling $x_{B_0}$ if necessary. Hence, \eqref{145} becomes
\begin{equation}\label{146}
\sum_{\substack{|A|=r_0\\\omega\in\S_{r_0},\ \mbf}}\lambda_{A,B_0,\omega,\mbf}
\Fbf_\mbf\Tbf_A\Tbf_\omega(z_0)=0.
\end{equation}

By Lemma~\ref{150}\eqref{064}, the tensors $\Tbf_\omega(z_0)$ are nonzero and obtained by permuting the $r_0$ non scalar factors of $z_0$. For a fixed subset $A=\{a_1<\dots<a_{r_0}\}\subseteq[n]$, Lemma~\ref{150}\eqref{012} shows that $\Tbf_A$ moves these $r_0$ factors to the positions in $A$ without mixing them with scalar components. Thus $\Tbf_A\Tbf_\omega(z_0)$ has support exactly $A$, and different choices of $A$ give disjoint supports. Therefore \eqref{146} splits into independent relations, one for each fixed $A$:
\begin{equation}
\sum_{\substack{\omega\in\S_{r_0},\ \mbf}}\lambda_{A,B_0,\omega,\mbf}\Fbf_\mbf\Tbf_A\Tbf_\omega(z_0)=0.
\end{equation}

Multiplying on the left by $\Tbf_A^{-1}$ and using that $\Tbf_A^{-1}\Fbf_\mbf\Tbf_A=\Fbf_{\mathbf m'}$ for a permutation $\mathbf m'$ of the entries of $\mathbf m$ (together with re-indexing of $\mathbf m$), we obtain
\begin{equation}\label{147}
\sum_{\substack{\omega\in\S_{r_0}\\\mbf}}\lambda_{A,B_0,\omega,\mbf}\Fbf_\mbf\Tbf_\omega(z_0)
=0.
\end{equation}

To conclude, we analyse \eqref{147} inside the $\langle\Fbf_1,\ldots,\Fbf_n\rangle$-submodule generated by $\Tbf_\omega(z_0)$, for a fixed $\omega\in\S_{r_0}$.

Write $\Tbf_\omega(z_0)$ in the standard cyclic basis $\{w_{i_\ell}^{(a)}\}_{a\in\Z/d\Z}$, where\[w_{i_\ell}^{(a)}=\sum_{t\in\Z/d\Z}\xi^{-at}v_{i_\ell}^t,\qquad\Fbf_\ell\bigl(w_{i_\ell}^{(a)}\bigr)=w_{i_\ell}^{(a+1)}.\]For a fixed $\omega\in\S_{r_0}$, the operator $\Tbf_\omega$ acts on $z_0$ by permuting the lower indices of the tensor, up to multiplication by a power of $u$. In particular, $\Tbf_\omega(z_0)$ still generates a $\langle\Fbf_1,\ldots,\Fbf_n\rangle$-submodule spanned by tensors written in the cyclic basis $\{w_{i_\ell}^{(a)}\}$.

Since $\Tbf_\omega$ does not affect the upper indices, we may assume without loss of generality that all upper indices of $z_0$ are equal to zero. With this convention, the situation reduces exactly to that considered in the proof of \cite[Theorem~10]{EsRyH18}. Hence the tensors
$\Fbf_\mbf\Tbf_\omega(z_0)$, with $\omega\in\S_{r_0}$ and $\mbf\in(\Z/d\Z)^{r_0}$, are pairwise distinct basis elements of this submodule.

Therefore \eqref{147} forces $\lambda_{A,B_0,\omega,\mbf}=0$ for all $\omega\in\S_{r_0}$ and $\mbf\in(\Z/d\Z)^{r_0}$. Since $A$ was arbitrary, this contradicts the assumption that some coefficient $\lambda_{A,B_0,\omega,\mbf}$ was nonzero. Thus no nontrivial relation \eqref{142} exists, and $\Csf_n$ is linearly independent.
\end{proof}

\subsection{Other presentations}\label{162}

The Rook Yokonuma--Hecke algebra $\RY_{d,n}(u)$ can also presented by generators $\Tsf_1,\ldots,\Tsf_{n-1}$ and $\Fsf_1,\ldots,\Fsf_n$ satisfying \eqref{087} to \eqref{089}, and the generator $\Rsf:=\Tsf_1\cdots\Tsf_{n-1}\Nsf$, subject to the following relations:
\begin{gather}
\Rsf^2=u^{n-1}\Rsf;\label{108}\\
\Tsf_i\Rsf=\Rsf\Tsf_i,\quad i>1;\qquad\Rsf\Tsf_1^{-1}\Rsf\Tsf_1=u\Rsf\Tsf_1^{-1}\Rsf=\Tsf_1\Rsf\Tsf_1^{-1}\Rsf;\label{109}\\
\Fsf_i\Rsf=\Rsf\Fsf_i;\qquad\Rsf\Fsf_1=\Rsf.\label{110}
\end{gather}
Relations \eqref{108} to \eqref{110} are derived by substituting $\Nsf=\Tsf_{n-1}^{-1}\cdots\Tsf_1^{-1}\Rsf$ into relations \eqref{090}, \eqref{092} and \eqref{107}. Specifically, relation \eqref{110} follows from \eqref{090}, the first relation in \eqref{109} corresponds to the first relation in \eqref{092}, relation \eqref{108} is obtained from the second relation in \eqref{092}, and the second relation in \eqref{109} comes from relation \eqref{107}.

If we define $\Rsf_1=\Rsf$ and $\Rsf_{i+1}=\Tsf_i\Rsf_i\Tsf_i^{-1}$ for all $i\in[n-2]$, we achieve a presentation of $\RY_{d,n}(u)$, as the one mentioned in Remark~\ref{112}, by generators $\Tsf_1,\ldots,\Tsf_{n-1}$ and $\Fsf_1,\ldots,\Fsf_n$ satisfying \eqref{087} to \eqref{089}, and generators $\Rsf_1,\ldots,\Rsf_n$, subject to the following relations:
\begin{gather}
\Rsf_i^2=u^{n-1}\Rsf_i;\label{114}\\
\Tsf_i\Rsf_i=\Rsf_{i+1}\Tsf_i\,;\qquad\Tsf_i\Rsf_j=\Rsf_j\Tsf_i,\quad j\not\in\{i,i+1\};\qquad\Rsf_{i+1}\Rsf_i\Tsf_i=u\Rsf_{i+1}\Rsf_i=\Tsf_i\Rsf_{i+1}\Rsf_i;\label{115}\\
\Fsf_i\Rsf_j=\Rsf_j\Fsf_i;\qquad\Rsf_i\Fsf_i=\Rsf_i.\label{116}
\end{gather}
Specifically, relation \eqref{114} follows directly from \eqref{108} together with the definition of each $\Rsf_i$, the first relations in \eqref{115} arise immediately from the definition of $\Rsf_{i+1}$, the remaining two are obtained inductively from \eqref{109} and from the second relation in \eqref{110}, and relation \eqref{116} follows from \eqref{110}.

Defining $\Psf=u^{1-n}\Rsf$ and substituting into relations \eqref{108} to \eqref{110}, it follows that $\RY_{d,n}(u)$ admits a presentation by generators $\Tsf_1,\ldots,\Tsf_{n-1}$ and $\Fsf_1,\ldots,\Fsf_n$ satisfying \eqref{087} to \eqref{089}, together with the generator $\Psf$ subject to the following relations:
\begin{gather}
\Psf^2=\Psf;\label{135}\\
\Tsf_i\Psf=\Psf\Tsf_i,\quad i>1;\qquad\Psf\Tsf_1^{-1}\Psf\Tsf_1=u\Psf\Tsf_1^{-1}\Psf=\Tsf_1\Psf\Tsf_1^{-1}\Psf;\label{136}\\
\Fsf_i\Psf=\Psf\Fsf_i;\qquad\Psf\Fsf_1=\Psf.\label{137}
\end{gather}

Following \cite[Section 2]{Ha04}, define $\Psf_i=(u^{1-n})^i(\Tsf_1\cdots\Tsf_{n-1})^i\Nsf^i$ for all $i\in[n]$. Since $\Psf_1=\Psf$, we can present $\RY_{d,n}$ by generators $\Tsf_1,\ldots,\Tsf_{n-1}$ and $\Fsf_1,\ldots,\Fsf_n$ satisfying \eqref{087} to \eqref{089}, together with the inverses $\Tsf_1^{-1},\ldots,\Tsf_{n-1}^{-1}$ satisfying \eqref{130}, and the family of generators $\Psf_1,\ldots,\Psf_n$, subject to the following relations:\begin{gather}
\Psf_i^2=\Psf_i;\label{132}\\
\Psf_{i+1}=u\Psf_i\Tsf_i^{-1}\Psf_i;\qquad\Psf_i\Tsf_j=\Tsf_j\Psf_i,\quad i<j;\qquad\Psf_i\Tsf_j=\Tsf_j\Psf_i=u\Psf_i,\quad j<i;\label{133}\\
\Fsf_i\Psf_j=\Psf_j\Fsf_i;\qquad\Psf_i\Fsf_j=\Psf_i,\quad j\leq i.\label{134}
\end{gather}
Relations \eqref{132} and \eqref{133} follow exactly as in \cite[Lemma 1.4]{Ha04} by using \eqref{087} and \eqref{089} to \eqref{092}, and relation \eqref{134} follows from \eqref{089}, \eqref{137} and the first relation in \eqref{133}.

\subsubsection*{Acknowledgements}

The first named author acknowledges the financial support of DIDULS/ULS, through
the project PR2553853. The third named author was partially supported by the grant UVA22991 (Proyecto PUENTE UV).

%\label{166}

\bibliographystyle{plainurl}
\bibliography{../../../../Documents/Projects/LaTeX/bibtex.bib}

\end{document}